\documentclass[a4paper,11pt]{amsart}

\usepackage{amsfonts}
\usepackage{amssymb}
\usepackage{graphics}

\newcommand{\dl}{\lambda}

\newcommand{\C}{\mathbb{C}}
\newcommand{\N}{\mathbb{N}}
\newcommand{\Z}{\mathbb{Z}}
\newcommand{\R}{\mathbb{R}}

\newcommand{\tX}{\widetilde{X}}

\newcommand{\tS}{\widetilde{S}}

\newcommand{\fol}{\mathcal{F}}
\newcommand{\tilf}{\widetilde{\mathcal{F}}}

\newcommand{\partx}{{\partial /\partial x}}
\newcommand{\party}{{\partial /\partial y}}
\newcommand{\partz}{{\partial /\partial z}}

\newcommand{\bu}{{\bar{u}}}
\newcommand{\bv}{{\bar{v}}}
\newcommand{\bw}{{\bar{w}}}

\def\picill#1by#2(#3)#4
{\vbox to #2
{\hrule width #1 height 0pt depth 0pt
\vfill\special{illustration #3 scaled #4}}}

\newtheorem{theorem}{Theorem}
\newtheorem{prop}{Proposition}
\newtheorem{corol}{Corollary}
\newtheorem{lemma}{Lemma}

\newtheorem{obs}{Remark}

\theoremstyle{definition}
\newtheorem{defi}{Definition}

\theoremstyle{remark}
\newtheorem{rem}{Remark}

\begin{document}

\title[On the resolution of foliations]{On resolution of 1-dimensional foliations on 3-manifolds}

\author{Julio C. Rebelo \, \, \, \& \, \, \, Helena Reis}
\address{}
\thanks{}

\begin{abstract}
This paper is devoted to the resolution of singularities of holomorphic vector fields and of one-dimensional holomorphic
foliations in dimension~$3$ and it has two main objectives. First, from the general perspective of one-dimensional foliations,
we build upon the work of Cano-Roche-Spivakovsky \cite{canorochetc} and essentially complete it.
As a consequence, we obtain a general resolution theorem comparable to the resolution theorem of
McQuillan-Panazzolo \cite{danielMcquillan} but proved by means of rather different methods.

The second objective of this paper consists of looking at a special class of singularities of foliations containing,
in particular, all singularities of complete holomorphic vector fields on complex manifolds of dimension~$3$. We then
prove that for this class of holomorphic foliations, there holds a much sharper resolution theorem.
This second result was the initial
motivation of this paper and it relies on the combination of the previous resolution theorems for (general) foliations
with some classical material on asymptotic expansions for solutions of differential equations.
\end{abstract}

\maketitle

\section{Introduction}


The purpose of this introduction is to state the main results obtained in the course of this paper along with the basic notions
needed to make their statements intelligible to the non-expert reader. A more detailed discussion about the place of our results
in the current state-of-art in the area as well as an outline of our methods and of the structure of the paper
will follow in Section~\ref{Reviewliterature}.

Recall first that a singular, one-dimensional holomorphic foliation $\fol$ on $(\C^3,0)$ is nothing but the
(singular) foliation defined by the local
orbits of a holomorphic vector field defined on a neighborhood of the origin and having zero-set of codimension at least~$2$.
Unless otherwise stated, throughout this paper the phrase {\it singular holomorphic
foliation} means a singular, {\it one-dimensional}\, holomorphic foliation.
A simple consequence of Hilbert nullstellensatz is that, up to multiplying vector fields by a meromorphic function,
every {\it meromorphic vector field $X$ on $(\C^3,0)$}\, induces a singular {\it holomorphic foliation}\, on a neighborhood of the origin.
This foliation will be called the foliation associated with $X$. Clearly two (meromorphic) vector fields have the same associated foliation
if and only if they differ by a multiplicative (meromorphic) function.

Conversely, a vector field $X$ inducing a given foliation $\fol$
will be called a {\it representative of $\fol$}\, if $X$ is {\it holomorphic} and {\it the set of zeros of $X$}\,
has codimension at least two. In other words,
a representative vector field of $\fol$ is any holomorphic vector field tangent to $\fol$ and having a zero-set of codimension
at least~$2$.

There follows from the preceding that there is no point in considering ``singular meromorphic foliations'' since all
foliations in this category would, in fact, {\it be holomorphic}. Similarly, (singular) holomorphic
foliations have {\it empty}\, zero-divisor since their singular sets have
codimension at least~$2$. In other words, whenever we are exclusively concerned with foliations, we can freely eliminate
any (meromorphic) common factor between the components of a vector field tangent to the foliation to obtain a representative
vector field. Naturally this cannot be done
if we are focusing on an actual fixed vector field $X$ as it so often happens (more on this below).

In the above mentioned context of singular points, {\it resolution theorems} - also known as {\it desingularization theorems} -
are geared towards {\it foliations} in that we are ``free'' to eliminate non-trivial
common factors between components of a vector field whenever these common factors arise from transforming
a representative vector field by a birational map.
To further clarify these issues, we may recall that the
prototype of all ``resolution theorems'' for foliations is provided by
Seidenberg's theorem \cite{seiden} which is valid for foliations defined on a two-dimensional ambient. More precisely, if
$\fol$ denotes a singular holomorphic foliation defined
on a neighborhood of $(0,0) \in \C^2$, then Seidenberg's theorem asserts the existence of a finite sequence of one-point
blow-up maps, along with transformed foliations $\fol_i$ ($i=1, \ldots, n$)
$$
\fol=\fol_0 \stackrel{\Pi_1}\longleftarrow \fol_1 \stackrel{\Pi_2}\longleftarrow \cdots
\stackrel{\Pi_l}\longleftarrow \fol_n
$$
such that the following holds:
\begin{itemize}
  \item Each blow-up map $\Pi_i$ ($i=1, \ldots, n$) is centered at a singular point of $\fol_{i-1}$.

  \item All singular points of $\fol_n$ are {\it elementary}, i.e. $\fol_n$ is locally given by a representative
  vector field $X_n$ whose linear part at the singular point in question has at least one eigenvalue different from zero (cf. below).
\end{itemize}

Whereas Seidenberg's theorem is directly concerned with foliations, it is also very effective when applied to vector fields defined
on complex surfaces.
The general principle to use Seidenberg theorem to study vector fields - as opposed to foliations - consists of
applying Seidenberg theorem to the associated foliation while also keeping track of the divisor of zeros/poles of the
transformed vector field. In line with this point of view, Seidenberg's theorem is equally satisfying: the structure of the resolution map
(the composition of the blow-ups $\Pi_i$) is such that the transform of holomorphic vector fields retains its holomorphic
character (here the reader is reminded that the transform of a holomorphic vector field by a birational map is, in general,
a meromorphic vector field). More generally, Seidenberg's procedure allows for an immediate computation of
the zero-divisor of the transformed vector field. For example, if we blow-up a vector field $X$ having an isolated
singularity at $(0,0) \in \C^2$ and denote by $k$ the degree of the first non-zero homogeneous
component of the Taylor series of $X$ at $(0,0)$, then the zero-divisor of the blow-up of $X$ coincides with
the exceptional divisor and has multiplicity $k-1$ (unless $X$ is actually a multiple of the radial vector field
$x\partial /\partial x + y \partial /\partial y$ in which case the multiplicity is~$k$).

As it will be seen in the course of the discussion, the generalization of Seidenberg's theorem to foliations on $(\C^3,0)$ is a
very subtle problem. A very satisfactory answer is provided in \cite{danielMcquillan}, \cite{danielMcquillanpreprint} and it
relies heavily on a previous result by Panazzolo in \cite{P}. Slightly later, the topic was revisited from the point
of view of valuations in \cite{canorochetc}. The ``final models'' in the resolution theorem proved in \cite{canorochetc}
are, however, not as
accurate as those in \cite{danielMcquillan} (see Section~\ref{Reviewliterature} for further details).
The present paper grew out of an attempt to use the mentioned results to obtain a sharper resolution
result which {\it would hold for the special class of holomorphic foliations}\, which is associated with semicomplete vector fields
(cf. Theorem~B below). Whereas the class of foliations associated with semicomplete vector fields is rather special, it contains
the underlying foliations of all complete vector fields as well as many foliations arising in the context of Mathematical Physics
and the importance of these examples justifies the interest in a sharper (or ``simpler'') resolution statement valid only
for this class of foliations (see Section~\ref{Reviewliterature}).

However, from the point of view mentioned above, it turned out that the resolution theorem in \cite{canorochetc}
was not really suited to our needs because the corresponding ``final models'' were not accurate enough. As to the resolution
theorem in \cite{danielMcquillan}, we were unsure of the behavior of vector fields - as opposed to foliations - under
their procedure. Basically, we did not know if Panazzolo's algorithm in \cite{P} had one specific property and we raised
the issue in a preliminary version of this paper: we are grateful to the referee for having confirmed that the
algorithm in \cite{P} does have the required property. These issues will be detailed in Section~\ref{Reviewliterature}.
In any event, when studying the papers in question, we felt it would be nice to try and complete the work
of Cano-Roche-Spivakovsky \cite{canorochetc}
by deriving ``final models'' similar to those of \cite{danielMcquillan} (cf. Theorem~A below).

To state the mentioned resolution theorems, let us first recall some standard terminology.
Let $\fol$ be a singular holomorphic foliation defined
on a neighborhood of the origin in $\C^3$ (or more generally in $\C^n$). The {\it eigenvalues} of $\fol$ at the origin are defined
as the eigenvalues of the linear part at the origin of a representative vector field $X$ for $\fol$. Since two representative
vector fields for a same foliation $\fol$ must differ by an invertible multiplicative holomorphic function, it follows
that the eigenvalues of $\fol$ are well defined only up to a multiplicative constant.

Next, a singular point $p$ for $\fol$ is said to be {\it elementary}\, if $\fol$ has at least one eigenvalue different
from zero at~$p$. Similarly, we will say that
$\fol$ has a {\it nilpotent singularity} at $p$ if the linear part at the origin of a (local) representative vector field of $\fol$
is {\it nilpotent but non-zero}.
Finally, if the linear part of a representative vector field at the origin is equal to zero,
then $p$ is said to be a {\it degenerate}\, singular point of $\fol$.

We are now able to state Theorem~A. Whereas this theorem is in most respects equivalent to
the main result in \cite{danielMcquillan}, the corresponding proofs are very different.

\vspace{0.2cm}

\noindent {\bf Theorem A}. {\sl Let $\fol$ denote a (one-dimensional) singular holomorphic foliation defined on a neighborhood
of $(0,0,0) \in \C^3$. Then there exists a finite sequence of blow-up maps along with transformed foliations
\begin{equation}
\fol=\fol_0 \stackrel{\Pi_1}\longleftarrow \fol_1 \stackrel{\Pi_2}\longleftarrow \cdots
\stackrel{\Pi_l}\longleftarrow \fol_n \label{blowup-resolution_globaldesingularization}
\end{equation}
satisfying all of the following conditions:
\begin{itemize}
  \item[(1)] The center of the blow-up map $\Pi_i$ is (smooth and) contained in the singular set of $\fol_{i-1}$, $i=1, \ldots , n$.

  \item[(2)] The singularities of $\fol_n$ are either elementary or persistently nilpotent.

  \item[(3)] The number of persistently nilpotent singularities of $\fol_n$ is finite and each of them can be turned into
  elementary singular points by performing a single weighted blow-up of weight~$2$.
\end{itemize}
}

As it is implicit in the above statement, persistent nilpotent singular points is a special type of nilpotent singular point
which will be detailed described in Section~\ref{persistentnilpotent-Section} (cf. Theorem~\ref{normalformpersistentnilpotent}).
As it will be seen, they also play a special role in the resolution theorem of \cite{danielMcquillan}. Namely, they appear
as singularities associated with a special type of $\Z /2\Z$-orbifold which, incidentally, require a weight~$2$ blow-up
to be turned into elementary ones. It is also worth pointing out that
both statements are sharp in the sense that well known examples by Sancho and Sanz show the use of a weight~$2$ blow-up
cannot be avoided (cf. Sections~\ref{Reviewliterature} and~\ref{persistentnilpotent-Section}).

In particular, both Theorem~A and the resolution theorem, Theorem~\ref{danielJDG}, in \cite{danielMcquillan} asserts
the existence of a birational
model for $\fol$ where all singularities of $\fol$ are elementary except for finitely many ones that can be turned
into elementary singular points by means of a single blow-up of weight~$2$. In this sense, differences between these two theorems
are down to the way in which these rational models are constructed. Alternatively, Theorem~A can simply be regarded as a new proof
of the resolution theorem in \cite{danielMcquillan}.

In terms of the construction of the mentioned rational models, we briefly mention that McQuillan and Panazzolo work
in the category of {\it weighted blow up}, along with the corresponding orbifolds, while in Theorem~A we restrict ourselves
as much as possible to the use of standard (i.e. unramified) blow-ups. Once again, additional information on these strategies
can be found in Section~\ref{Reviewliterature}.

At this point, it is convenient to introduce some terminology.
Throughout this paper the term {\it blow-up}\, will
refer to {\it standard (i.e. homogeneous) blow-ups}. This applies, in particular, to the statement of Theorem~A. As to blow-ups
with weights (i.e. non-homogeneous or ramified blow-ups) which are inevitably also involved in the discussion,
{\it these will be explicitly referred to as weighted (or ramified) blow-ups}.

Also, we will say that a (germ of) foliation $\fol$ can be
{\it resolved}\, if there is a sequence of blowing-ups as in~(\ref{blowup-resolution_globaldesingularization}) leading to
a foliation $\fol_n$ all of whose singularities are elementary. Similarly, a sequence of blowing-ups
as in~(\ref{blowup-resolution_globaldesingularization}) will be called a {\it resolution of $\fol$}\, if all the singular points
of $\fol_n$ are elementary. Whenever sequences of {\it weighted blow-ups}\, leading to a foliation having
only elementary singular points are considered, they may be referred to as a {\it weighted resolution of $\fol$}. With
this terminology, while
every germ of foliation on $(\C^3,0)$ admits a weighted resolution, as follows from \cite{danielMcquillan}
or Theorem~A, the mentioned examples of Sancho and Sanz show that not all of them admit a resolution. The reader
is referred to Section~\ref{Reviewliterature} for a detailed discussion on
the mutual interactions involving \cite{canorochetc}, \cite{danielMcquillan}, and our discussion revolving around Theorem~A.

We can now go back towards our initial motivation, namely to germs of foliations $\fol$ on $(\C^3,0)$ that are
associated with a semicomplete vector field. Since
the notion of semicomplete singularity was introduced along with its first applications to the (global) study of complex
vector fields (\cite{JR1}), it has been natural to ask whether all foliations in this class admit a resolution.
A special instance of this problem which is of interest in the
study of complex Lie group actions consists of asking whether the underlying foliation of a complete holomorphic
vector field (on some complex manifold of dimension~$3$) can be transformed into a foliation all of whose singular
points are elementary by means of a sequence of blow-ups as in~(\ref{blowup-resolution_globaldesingularization}).

To state our results concerning this special class of foliations, let us place ourselves once and for all
in the context of semicomplete vector fields. First, it is convenient to recall
that a singularity of a holomorphic vector field $X$ is said to be {\it semicomplete}\, if the integral
curves of $X$ admit a maximal domain of definition in $\C$, cf. \cite{JR1}. In particular, whenever $X$ is a {\it complete
vector field}\, defined on a complex manifold $M$, every singularity of $X$ is automatically semicomplete.
The answer to the above question is then provided by the following theorem:

\vspace{0.2cm}

\noindent {\bf Theorem B}. {\sl Let $X$ be a semicomplete vector field defined on a neighborhood of the origin
in $\C^3$ and denote by $\fol$ the holomorphic foliation associated with $X$. Then one of the following holds:
\begin{enumerate}
  \item The linear part of $X$ at the origin is nilpotent (non-zero).

  \item There exists a finite sequence of blow-ups maps along with transformed foliations
$$
\fol = \fol_0 \stackrel{\Pi_1}\longleftarrow \fol_1 \stackrel{\Pi_2}\longleftarrow \cdots
\stackrel{\Pi_r}\longleftarrow \fol_r
$$
such that all of the singular points of $\fol_r$ are elementary. Moreover, each blow-up map $\Pi_i$ is centered in the
singular set of the corresponding foliation $\fol_{i-1}$. In other words, the foliation $\fol$ can be resolved.

\end{enumerate}
}

\vspace{0.1cm}

Let us emphasize that item~1 in Theorem~B means that the linear part of $X$ is (nilpotent) {\it non-zero from the outset}.
In other words, if the foliation $\fol$ associated with $X$ cannot be resolved, then $X$ has a non-zero nilpotent linear part
and this property is ``universal'' in the sense that it does not depend on any sequence of blow-ups/blow-downs carried out.
In particular, we can choose a ``minimal model'' for our manifold and the corresponding transform of $X$ will still have
non-zero nilpotent linear part at the corresponding point.
Moreover, from Theorem~\ref{normalformpersistentnilpotent} about ``persistent nilpotent singularities'', it is easy
to obtain accurate normal forms for the vector field $X$ (cf. Definition~\ref{definitionpersistentnilpotent}).

Also, the statement of Theorem~B involves the linear part of the vector field $X$ rather than
the linear part of the associated foliation $\fol$. This makes for a stronger statement which is better
emphasized by Corollary~C below:

\vspace{0.1cm}

\noindent {\bf Corollary C}. {\sl Let $X$ be a semicomplete vector field defined on a neighborhood of $(0,0,0) \in \C^3$
and assume that the linear part of $X$ at the origin is equal to zero. Then item~(2) of Theorem~B holds.}

\vspace{0.1cm}

More precisely,
Theorem~B asserts that foliations associated with semicomplete vector fields in dimension~$3$ can be resolved
by a sequence of blow-ups centered in the singular set
except for a very specific case in which the vector field $X$ (and hence the foliation $\fol$) has a ``universal''
non-zero nilpotent linear part. As mentioned, these statements have the advantage of involving the vector field and not only the
underlying foliation. To clarify the meaning of this sentence, consider a holomorphic (semicomplete) vector field $X$
having the form $X = fY$, where $Y$ is another holomorphic vector field and $f$ is a holomorphic function. Whereas $X$ and
$Y$ induce the same singular foliation $\fol$, an immediate consequence
of Corollary~C is that $\fol$ must be as in item~(2) of Theorem~B {\it provided that $f$ vanishes at the origin}\,:
in fact, if $f$ and $Y$ are as indicated, then the linear part of $X$ vanishes at the origin at that $\fol$ is, indeed, singular
(clearly there is nothing to be proved if $\fol$ is regular). In other words, if $X = fY$ as above with $f(0,0,0) =0$ and
$X$ semicomplete, then the foliation associated with $X$ can certainly be resolved even if $Y$ has a nilpotent
singular point at the origin.

A few additional comments are needed to fully clarify the role of item~(1) in Theorem~B.
First note that more accurate normal forms are available for the vector fields in question: indeed,
Theorem~\ref{normalformpersistentnilpotent} provides accurate normal forms for all persistent nilpotent singular
points. In addition, {\it not all}\, nilpotent vector fields giving rise to
persistent nilpotent singularities are semicomplete
and, in this respect, the normal form provided by Theorem~\ref{normalformpersistentnilpotent} will further be refined later
on (see Section~\ref{Examples_and_comments}).

Next, taking into account the global setting of complete vector fields,
it is natural to wonder if there is, indeed, {\it complete
vector fields} inducing a foliation with singular points that cannot be resolved. As a consequence of Theorem~B,
such vector fields would definitely be pretty remarkable
since they must have a (non-zero) ``universal'' nilpotent singular point. To confirm that these global
situations do exist, however,  it suffices to note that the polynomial vector field
$$
Z = x^2 \partial /\partial x + xz \partial /\partial y + (y -xz) \partial /\partial z
$$
can be extended to a complete vector field defined on a suitable open manifold (see Section~\ref{Examples_and_comments} for detail).
As will be seen, the origin in the above coordinates
constitutes a nilpotent singular point of $Z$ that cannot be resolved by means of blow-ups as in item~(2) of Theorem~B,
albeit this
nilpotent singularity can be resolved by using a blow-up centered at the (invariant) $x$-axis.

Finally, the question raised above about the existence of singularities as in item~(1) of Theorem~B in global
settings can also be asked in the far more restrictive case of holomorphic vector
field defined on {\it compact manifolds}\,
of dimension~$3$. Owing to the compactness of the manifold, every such vector field is automatically complete. In this
setting, the methods used in the proof of Theorem~B easily yield:

\vspace{0.2cm}

\noindent {\bf Corollary D}. {\sl Let $\fol$ be the foliation associated with a vector field $X$ globally defined
on some compact manifold $M$ of dimension~$3$. Then every singular point of $\fol$ can be resolved.}

\vspace{0.1cm}

Let us close this introduction with a couple of remarks inspired by some questions asked to us by A. Glutsyuk.
Essentially his questions concern resolution strategies with minimal number of (weighted) blow ups which can also
be seen as an analogue of some questions previously considered in the context of Hironaka's theorem. In this respect,
it is clear that being able to work with weighted blow ups, as opposed to standard ones, increases the chances of
reducing the number of blow ups to resolve a given foliation. Indeed, it is easy to produce examples of this phenomenon
already in dimension~$2$
and in the context of Seidenberg's theorem. Hence, there is no chance that the strategy used in the proof
of Theorem~A will in general minimize the number of blow ups required to resolve a given foliation. However,
we ignore if Panazzolo's algorithm \cite{P}
has minimizing properties in the preceding sense.

A similar question directly motivated by the fact that in dimension~$2$ standard blow ups suffice to resolve
any foliation, consists of trying to minimize the number of weighted
blow ups needed to obtain the resolution. In this case, and at least for generic foliations, Theorem~A seems to provide
a satisfactory answer. Let us try to sketch an argument in this direction.
As it follows from Theorem~\ref{normalformpersistentnilpotent}, persistent nilpotent singular
points are naturally associated with certain formal separatrices (i.e. formal invariant curves) having some special properties.
Their ``position'' in the exceptional divisor obtained after finitely many blow ups
is thus determined by the corresponding formal separatrices. In particular, it is possible to talk about these singularities being in
``general position'' for a given germ in an intrinsic way, i.e. independently of the use of any sequence of (standard)
blow ups. At least when these singularities are in ``general position'' for a foliation $\fol$,
then Theorem~A should minimize the number of
weighted blow ups needed to turn $\fol$ into a foliation all of whose singular points are elementary. Indeed, each such singularity
requires at least one weighted blow up to be turned into elementary singular points and each such blow up can non-trivially
affect only one of these singularities thanks to the ``general position'' assumption. Thus the number of weighted blow ups
needed cannot be smaller than the number of persistent nilpotent singularities and the later is matched by the procedure in Theorem~A.
We ignore, however, if the ``general position assumption'' is really needed for this statement. Note that if there is
a foliation $\fol$ that can be resolved by using less weighted blow ups than those prescribed in Theorem~A, then $\fol$ should
conceal at least two persistent nilpotent singularities so ``close'' to each other that they can both be turned into elementary
singular points by means of a same weighted blow up.

Finally, we point out that a comprehensive discussion of the interactions
between the above results and the results in \cite{canorochetc}, \cite{danielMcquillan} is provided in
the next section along with a brief description of our methods and of the structure of this paper.

\noindent {\bf Acknowledgments.}
The authors are very grateful to C. Roche for sharing with us his expertise in the field
at various moments during the past few years. We warmly thank F. Cano
who told us long ago that a singularity that cannot be resolved by standard blow-ups must possess
a ``persistent'' formal separatrix. Equally, we are indebted to
D. Panazzolo for explaining to us the structure of \cite{danielMcquillan} and to A. Glutsyuk for his
sharp questions and interest in this paper. Last but not least, we
thank the referee for bringing some additional references to our attention and especially for informing us that
Panazzolo's algorithm in \cite{P} is such that the centers of the weighted blow-ups are {\it strictly invariant
with respect to the quasi-homogeneous filtration}\, thus settling an issue raised in a preliminary version
of this paper.

The second author was partially supported by CMUP, which is funded by FCT with national(MCTES) and European structural
funds through the programs FEDER, under the partnership agreementPT2020 (UID/MAT/00144/2013 and UID/MAT/00144/2019).

\section{Brief review of the literature}\label{Reviewliterature}

As already pointed out, the prototype of every {\it resolution (or desingularization)}\, theorem for
(one-dimensional) holomorphic foliations is Seidenberg's theorem which holds for
foliations defined on complex surfaces. It is then natural to begin this section by taking a closer look at Seidenberg theorem
for foliations on $(\C^2,0)$, see \cite{seiden}, \cite{arnold}, or \cite{bookYakovenko}.
Let $\fol$ denote a singular holomorphic foliation defined on a neighborhood of
$(0,0) \in \C^2$. Since the foliation is defined on a complex two-dimensional manifold, all blow-ups are necessarily
centered at points. Seidenberg theorem is a simple algorithm to transform $\fol$ into a foliation all of whose singular points
are elementary. The procedure in Seidenberg's algorithm can be summarized as follows. Consider $\fol$ as above
and its (necessarily isolated) singular point at $(0,0) \in \C^2$.
If $(0,0)$ is an elementary singular point
for $\fol$, then there is nothing to be done and the procedure is trivial. Otherwise, we blow-up $\fol$ at $(0,0)$ to obtain a
new foliation $\tilf_{(1)}$ whose singular points are denoted by $p_{(1),1}, \ldots , p_{(1),k}$.
If all the singular points $p_{(1),1}, \ldots , p_{(1),k}$
are elementary for $\tilf_{(1)}$ then the procedure ends. Otherwise, we carry on by blowing up each non-elementary
singular point of $\tilf_{(1)}$. Denote by $\tilf_{(2)}$ the foliation arising from these blow ups and let
$p_{(2),1}, \ldots , p_{(2),k}$ be its singular points.
The procedure stops at $\tilf_{(2)}$, if all the singular points $p_{(2),1}, \ldots , p_{(2),k}$
are elementary. If not, the procedure
is inductively continued. Seidenberg's theorem claims that this procedure is finite. In other words, we eventually
obtain a foliation all of whose singular points are elementary.

The seminal paper \cite{mamo} by Mattei and Moussu was probably the first work to fully realize how much of an
effective tool Seidenberg's theorem is for the study of foliations on complex surfaces. Indeed, in the
course of their paper, a systematic method to study these singular points was introduced. This method then led
to several remarkable results (see, for example, \cite{camacho} and \cite{camachosadLN}). The success of these
and other works, somehow made ``popular'' the slightly abridged version of Seidenberg's theorem claiming that
for every foliation $\fol$ on $(\C^2,0)$, there exists a birational model where the transformed foliation has only
elementary singular points. We call this an abridged form of Seidenberg's theorem because it focus on the existence
of the birational model without keeping track of the way in which this model is constructed.

It turns out, however, that besides foliations, there are many problems where the object of primary interest is a
holomorphic vector field. Examples of these problems include:
\begin{itemize}

  \item[(a)] The classification of semicomplete (germs of) vector fields in dimension~$2$ (see for example \cite{ghys} or
  \cite{Adolfo} and its reference list; see also \cite{dlousskyetc} for a nice application of the classification).

  \item[(b)] The study of parabolic curves and/or Fatou petals for {\it diffeomorphisms}\, of $(\C^2,0)$ tangent to
  the identity, (see at various levels \cite{Hakim1}, \cite{Hakim2}, \cite{MarcoAbate}, \cite{Brochero_cano}).

\end{itemize}
It is an important feature of Seidenberg's theorem that it is just as effective when applied to vector fields, but
in doing so, it is convenient to take into account the full statement of this theorem, rather than its ``abridged version''.
Indeed, the transform of a holomorphic vector field by an arbitrary birational map is, in general, a
{\it meromorphic vector field}. This problem, however, can be avoided as explained below.

In fact, it is easy to check that - in full generality - the transform of a holomorphic vector field by a
(standard) blow-up remains holomorphic provided that {\it the center of the blow up is invariant by the vector field
in question}. In view of this remark, the method to apply Seidenberg's theorem to vector fields can be
summarized in two steps as follows. First apply this theorem to the foliation associated with the
vector field so as to turn this foliation into another foliation having only elementary singular points.
The corresponding vector field
$\widetilde{X}$ then takes on the local form $X = fY$, where $Y$ is a holomorphic vector field with linear part having at
least one eigenvalue different from zero at the singular point and where $f$ is {\it a holomorphic function}\, whose divisor
of zeros can easily be computed. In the mentioned (local) nature of $f$, it is encoded the accurate construction
provided by Seidenberg's theorem, namely: every blow-up used in his construction is centered at a singular
point of the underlying foliation. In particular, the center of every blow up is invariant by the vector field
in question. There follows that the transformed vector field retains its holomorphic character and, whereas in general
non-empty even if we start with a vector field having isolated singular points, the divisor zeros of the final
vector field can easily be computed.

Note also that many question revolving around items~(a) and~(b) above are actively being studied in dimension~$3$.
In this respect, it reasonable to expect that Theorems~B may play a role in facilitating investigations related to
item~(a). As to item~(b), the same can be expected either from McQuillan-Panazzolo resolution theorem in
\cite{danielMcquillan}, see Theorem~\ref{danielJDG}, or from our Theorem~A.

Next, a few additional words about the interest of semicomplete vector fields seem appropriate.
The reader is referred to Section~\ref{semicompletestuff} for the definition of semicomplete vector fields
along with some of their basic properties. For the time being, it is enough to point out the following:
\begin{itemize}
  \item The germ defined by a {\it complete} vector field
at any of its singular points automatically belongs to the class of semicomplete vector fields. Thus, Theorem~B covers all complete
vector fields (on a $3$-dimensional ambient) and, in particular, all holomorphic vector fields defined on $3$-dimensional
compact complex manifolds. Therefore Theorem~B is particularly suited to the study of complex Lie groups actions on complex manifolds.

  \item On a different direction, and in connection with Nevanlinna theory, there is significant activity in Complex Analysis revolving
around differential equations possessing meromorphic solutions (see for \cite{EMS-??} for further information and references).
It turns out that every equation, or system of equations, possessing meromorphic solutions is automatically semicomplete.

  \item There is a huge amount of literature in Mathematical Physics devoted to special equations (or
systems of equations) possessing the so-called {\it Painlev\'e property}, (see for example \cite{SomethingPainleve}).
The notion of semicomplete being a sort of ``relative'' of the notion of Painlev\'e property, many examples of
systems of equations that are naturally semicomplete can be found in the context of Mathematical Physics and/or
of Special equations. These include the Painlev\'e equations
P-I, P-II, P-IV as well as the ``modified'' P-III and P-V. Similarly many Chazy equations are naturally semicomplete
\cite{adolfoChazy} and the same remark applies to Garnier's systems. Also the most interesting examples of
Halphen vector fields in the sense of \cite{GuillotPSL2}, and in particular the Halphen vector field appearing
in \cite{monopoles} and the vector field associated with Ramanujan functions $P$, $Q$, and $R$ (the Eisenstein series of weight~$2$,
$4$, and $6$, cf. \cite{ramanujan}) are all semicomplete as well.

\end{itemize}

Having reviewed Seidenberg theorem in detail and also recalled the interest of semicomplete vector fields, we can now move
on to reviewing {\it resolution theorems}\, for singularities of foliations in dimension~$3$. The discussion will also
help us to place our Theorems~A and~B in proper perspective with respect to previous works.

\subsection{Quick chronological review of previous results}

At a very basic level, any attempt of generalizing
Seidenberg's theorem to dimension~$3$ involves deciding whether we are interested in foliations
of {\it dimension one}\, or of {\it codimension one (i.e. of dimension two)}. Whereas for codimension one foliations
in $3$-dimensional manifolds there is a decisive answer that can hardly be improved on (see \cite{cano}), the story
involving foliations of dimension one is longer and more elusive. Unless otherwise stated, throughout the sequel the
term ``foliation'' always means a singular holomorphic foliation of dimension one.

Resolution results for foliations on $(\C^3, 0)$ started with \cite{canothesis} where the author proves
his {\it formal local resolution theorem}. Namely, he shows the existence of a winning strategy for the (formal) Hironaka game.
Among other things, in this work it appears for the first time a phenomenon involving singularities
possessing a certain {\it formal separatrix}
(a formal curve invariant by the foliation) which posed some serious difficulty to be resolved by means of standard blow ups.
Here, we remind the reader that in what follows the term
{\it blow up}\, always means standard blow ups - as opposed to weighted blow ups.

Building on the work of Cano, \cite{canothesis}, Sancho and Sanz were able to work out the first examples of foliations
that cannot be resolved by means of blow-ups with invariant centers.
Their examples consisted of nilpotent singular points
and are briefly recalled in Section~\ref{persistentnilpotent-Section}, see \cite{P} and \cite{CanoRoche} for detailed accounts.

Naturally, the examples provided by Sancho and Sanz may also appear only after a few blow-ups so that they, indeed, yield
numerous examples of foliations that cannot
be resolved by standard blow-ups. As an additional simple example of those,
we may take the family of foliations
$\fol_{\alpha, \beta, \lambda}$ associated with the respective vector fields
\[
 x( y - \lambda x + (1 - \beta)xz) \frac{\partial}{\partial x} +
(y^2 - \lambda xy + xz^3 - (\alpha + \beta) xyz)  \frac{\partial}{\partial y} +
z(y - \lambda x - \beta xz) \frac{\partial}{\partial z} \, .
\]
All foliations in this family have a degenerate (in fact, quadratic) singularity at the origin. Thus none of them
belongs to the Sancho-Sanz family since the latter {\it have a non-zero linear part}. However, by blowing-up
the origin, the blown-up foliation $\tilf_{\alpha, \beta, \lambda}$ shows a (nilpotent) singularity in belonging to the
family of examples by Sancho and Sanz: this singularity lies at the origin of coordinates $(u,v,z)$ with $x=uz$ and $y=vz$.
Therefore foliations in the family $\fol_{\alpha, \beta, \lambda}$ cannot be resolved by means of standard blow ups.
It is a remarkable feature of our Theorem~B that {\it none of these constructions}\, gives rise to a singularity associated
with a {\it semicomplete vector field}. In particular, none of them is realized in the context of complete holomorphic vector fields
on manifolds of dimension~$3$ (not necessarily compact). Theorem~B therefore points out a genuinely new phenomenon
in the area.

In view of Sancho Sanz examples, the following general question has quickly become popular among experts: {\it is it true
that every foliation $\fol$ that cannot be resolved by blow ups is such that they can be transformed into a foliation
exhibiting a singular point of Sancho Sanz type}\,? Similarly, the question on whether or not these singularities appear
in the context of {\it complete vector fields}\, has not gone unnoticed to most experts.

Clearly the resolution theorem of McQuillan-Panazzolo in \cite{danielMcquillan} (cf. Theorem~\ref{danielJDG})
provides an answer to the first question
while our Theorem~B answers the second one. The answer to the second question provided by Theorem~B is essentially sharp
as shown by the examples in Section~\ref{Examples_and_comments}. To the best of our knowledge,
there is no previous result in the literature pointing out the differences between the resolution
problem for general foliations and for the especial class of foliations covered by Theorem~B.

On the other hand, as far as the above questions are concerned, all the mentioned results require
the notion of Sancho Sanz singular point to be slightly generalized.
This generalization appears in \cite{danielMcquillan} under the form of those singularities
``intrinsically attached'' to orbifolds of type $\Z /2\Z$ while, in the present paper, they are called
persistent nilpotent singular points (Section~\ref{persistentnilpotent-Section})
and are characterized by Theorem~\ref{normalformpersistentnilpotent}.

Naturally Theorem~A also provides an answer to the first question mentioned above. In this regard,
the approach used to prove Theorem~A is, indeed, such that the answer to the above question is ``almost
equivalent'' to Theorem~A itself. To explain this assertion, and also because it naturally fits the structure of
our discussion, it is convenient to provide an explicit
statement in the form of Theorem~\ref{Microlocalversion_TheoremA} below.

Whereas strictly speaking Theorem~\ref{Microlocalversion_TheoremA} is
a particular case of Theorem~A, the two statements are basically equivalent thanks to the work of O. Piltant \cite{Piltant}.
In fact, most of our discussion on general foliations will revolve around the
proof of Theorem~\ref{Microlocalversion_TheoremA}. Once this theorem is established, Piltant's (axiomatic) patching theorem
allows us to derive Theorem~A by repeating word-by-word
a discussion already carried out in \cite{canorochetc}.

\begin{theorem}
\label{Microlocalversion_TheoremA}
Assume that $\fol$ cannot be resolved by a finite sequence of standard
blow-ups centered at singular sets. Then there exists a sequence of one-point blow ups (centered
at singular points) leading to
a foliation $\fol'$ with a singular point $p$ around which $\fol$ is given by a vector field of the form
$$
(y + zf(x,y,z)) \frac{\partial}{\partial x} + zg(x,y,z) \frac{\partial}{\partial y} + z^n \frac{\partial}{\partial z}
$$
for some $n \geq 2$ and holomorphic functions $f$ and $g$ of order at least~$1$ with $\partial g / \partial x (0,0,0) \ne 0$.
Furthermore we have:
\begin{itemize}
  \item The resulting foliation $\fol'$ admits a formal separatrix at $p$ which is tangent to the $z$-axis;

  \item The exceptional divisor is locally contained in the plane $\{ z=0\}$.
\end{itemize}
\end{theorem}

By now, we can go back to our chronological review of the literature. After the examples found by Sancho and Sanz,
the next truly major result in the area is due to D. Panazzolo \cite{P}. In \cite{P}, Panazzolo considers singularities
of real foliations in (real) dimension~$3$. He works in the real setting, rather than in the complex one,
due to the fact that his original motivation was Hilbert's problem on the number of limit cycles of a
polynomial vector field on $\R^2$. He then shows that
the corresponding germs of foliations can always be turned into a foliation all of whose singular points are elementary
{\it by means of a finite sequence of weighted blow ups centered at singular sets}. The proof of this fact is very elaborate
and ultimately relies on a construction associating to a singular foliation an array of six entries along with an order
on the resulting family of (possible) arrays. Panazzolo's theorem follows from showing that this quantity (array) always decreases
strictly under a suitable weighted blow up. Panazzolo's algorithm to choose the weighted blow up to be performed in
each situation is, in turn, based on the Newton diagram of the singular point. Here we also mention that his algorithm
is well adapted to transform vector fields, and not only foliations, in the sense that for the former we also need
to keep track of the divisors of zeros/poles. We will return to this point later in this section.

After Panazzolo's paper \cite{P}, McQuillan and Panazzolo extended the result to the complex case in their original
preprint \cite{danielMcquillanpreprint} whose published version is~\cite{danielMcquillan}. We refrain for providing more
information on \cite{danielMcquillan} here since a detailed discussion will be carried out in the next section.

A few years later, the topic of resolution of foliations on $(\C^3,0)$ was revisited by Cano-Roche-Spivakovsky
in \cite{canorochetc}. We will close this paragraph with a brief discussion of the material in \cite{canorochetc}
since our proof of Theorem~A builds on their approach.

From the very beginning, the general approach of resolution of singularities due to Zariski is followed in \cite{canorochetc}
and this makes their paper markedly different from \cite{danielMcquillan}. Since Zariski's point of view is followed,
the notion of {\it valuation}\, becomes
central in \cite{canorochetc} and the resolution problem is divided in two parts. Namely, there is the {\it local (resolution)
problem}\, which consists of ``simplifying'' - not necessarily all the singularities of a foliation - but only those
lying in the center of a given valuation (identified with its transforms, or extensions, through blow ups). Resolution results
for singularities lying in the center of a valuation are often referred to as {\it local uniformization theorems}.
Once a convenient local uniformization result is obtained, the second part of the problem deals with its ``globalization''. In
other words, once it is proved that for every valuation $\nu$, the singularities lying in the center of $\nu$
can be simplified (in some appropriate sense), we try to conclude that, in fact, all singularities of the foliation
can be simplified in the same sense.

In the present case, there is an axiomatic {\it gluing theorem} due to O. Piltant \cite{Piltant} providing a
very satisfactory general answer to the ``globalization problem''. Basically, as noted in \cite{canorochetc}, any
solution to the ``local problem'' that is obtained in a reasonably natural way can be turned into a global result
by this technique. Owing to Piltant result, the fundamental difficulty of the resolution problem lies in its
local version, namely, in obtaining a suitable local uniformization theorem.

The first main result - called Theorem~1 - in \cite{canorochetc} asserts that singularities in the center of a valuation
can always be simplified until they become {\it log-elementary}. We refer the reader to \cite{canorochetc} for
the accurate definition of log-elementary singularities since, for our purposes, it suffices to know that
they are at worst {\it quadratic}\,
in the sense that they are locally given by a representative vector field with non-zero second jet.

Theorem~1 is then turned into a global result - Theorem~2 in \cite{canorochetc} - by resorting to Piltant's theorem.
Summarizing, Theorem~2 in \cite{canorochetc} establishes the existence of a birational model for the initial foliation
where singular points are at worst log-elementary. Unlike \cite{danielMcquillan}, only standard blow ups are used in
the construction of the birational model in question. Nonetheless there is an evident disadvantage in the fact that
log-elementary singular points are still significantly harder to be dealt with than elementary singular points.

Apart from Theorems~1 and~2, the paper \cite{canorochetc} also contains a few more technical results making additional
non-trivial steps towards understanding those singularities that cannot be resolved by means of standard
blow ups. Aside some basic observations about the valuations that can pose obstacles to the local uniformization,
Theorem~3 of \cite{canorochetc} provides a sort of ``weak characterization'' of foliations
that {\it cannot be resolved}\, (by standard blow ups as it was always the case bar explicit mention on contrary).
If $\fol$ is one of these foliations, Theorem~3 in \cite{canorochetc} asserts
the existence of a valuation $\nu$ and of a formal surface $\widehat{W}$ having {\it transverse maximal contact}\, with
$\nu$, cf. \cite{canorochetc} or Section~\ref{Provingtheproposition}. Note that the condition about maximal
transverse contact can geometrically be interpreted by saying that the for every sequence of blow ups, the transform of $\widehat{W}$
will always pass through the center of $\nu$.

The material from \cite{canorochetc} sketched above will all enter in the proof of our Theorem~A. However, bar the results
established in \cite{canorochetc}, the remainder of the proof of Theorem~A will require only elementary methods from the
theory of foliations/singularities. Indeed, there is another characterization of foliations
that {\it cannot be resolved}\, which is more accurate than Theorem~3. This characterization is the content of
Proposition~\ref{BasedonCano_Roche_Spivakovsky} which was communicated to us a number of years ago by F. Cano.
Roughly speaking, if $\fol$ cannot be resolved, then $\fol$ must admit
a {\it formal separatrix} giving rise to a {\it sequence of infinitely near singular points}\, (which, in turn,
cannot be resolved, see Section~3 for terminology).

We believe that Proposition~\ref{BasedonCano_Roche_Spivakovsky} should be attributed to F. Cano although no proof
is available in the literature. For this reason, this paper includes a proof of this proposition relying
on Theorem~3 of \cite{canorochetc}. The proof given in Section~\ref{Provingtheproposition} seems to be original
in the sense that it may differ from the original argument envisaged by F. Cano.
In the present paper, the proof of Proposition~\ref{BasedonCano_Roche_Spivakovsky}
is split in two cases according to whether or not the formal surface is invariant by the foliation in question.
This, somehow, allows us to keep the discussion essentially elementary while taking some advantage of the $2$-dimensional
situation. With Proposition~\ref{BasedonCano_Roche_Spivakovsky} in hand, the remainder
of the proof of Theorem~\ref{Microlocalversion_TheoremA} is totally elementary with explicit computations.

\subsection{On McQuillan Panazzolo \cite{danielMcquillan}}

In this section we shall explain in detail the desingularization result proved in \cite{danielMcquillan} as
explained to us by D. Panazzolo. We will also compare the construction in \cite{danielMcquillan} with the
one carried out in this paper. Finally, we note that in the course of this section the discussion is focused
on {\it resolution theorems for general foliations}\,: comments on the additional ideas needed for Theorem~B are deferred to Section~\ref{commentsonthmB}.

Let us begin with some basic comments about weighted blow ups on a complex manifold of dimension~$3$ and the
transform of foliations. Unlike standard blow ups that keep the smooth nature of the space, the use of weighted blow
ups leads to spaces possessing orbifold-type singular points. Thus there is a loss of regularity but since
orbifold-type singular points are rather easy to handle, this is a minor issue. Up to allowing these singular
points to be present, the space resulting from the
(weighted) blow-up still is birationally equivalent to the initial one. In particular,
foliations can be transformed without any restrictions under weighted blow ups to yield new {\it birational models for them}.

The last sentence contrasts a bit with the case of vector fields and this deserves a specific comment.
Consider the (standard) blow up $\widetilde{X}$ of a holomorphic vector field $X$. It is easy to check that
the vector field $\widetilde{X}$ retains the holomorphic character provided that {\it the center of the
blow up map}\, is invariant by~$X$. In particular, this condition is satisfied if blow ups are centered
at the singular set of the underlying foliation. This statement, however, does not apply to general
weighted blow ups as follows from the example below.

\vspace{0.1cm}

\noindent $\bullet$ {\bf Example}. Consider the holomorphic vector field $X = F(x,y,z) \partial /\partial x +
G (x,y,z) \partial /\partial y + H (x,y,z) \partial /\partial z$ where $F(x,y,z) =y$ and $G$ and $H$ are such that
the $z$-axis $\{ x=y=0\}$ is contained in the singular set of $X$. Let $(x,t,z)$ be coordinates for the weighted blow-up
(of weight~$2$) centered at the $z$-axis in which the corresponding projection map $\Pi$ is given by
$\Pi (x,t,z) = (x^2, tx, z)$. A direct inspection shows that the corresponding transform $\Pi^{\ast} X$ of $X$ is given
by
\begin{eqnarray*}
\Pi^{\ast} X & = & \frac{1}{2x} F(x^2, tx, z) \frac{\partial}{\partial x} + \left[ -\frac{t}{2x^2} F(x^2, tx, z) +
\frac{1}{x} G(x^2, tx, z) \right]\frac{\partial}{\partial t} + \\
& & \, + H (x^2, tx, z) \frac{\partial}{\partial z} \, .
\end{eqnarray*}
Clearly $F(x^2, tx, z)/2x$ and $G(x^2, tx, z) /x$ are both holomorphic but $t F(x^2, tx, z) /2x^2$ is {\it strictly meromorphic}.
Therefore $\Pi^{\ast} X$ is meromorphic with poles over the exceptional divisor.

\vspace{0.1cm}

\noindent In fact, for the above blow up, the condition for the blow up of a holomorphic vector field $X$ to retain
its holomorphic nature can be explained as follows. For $\lambda \in \C^{\ast}$, consider the family of maps
$T_{\lambda} : \; \C^{3} \rightarrow \C^3$ given by $T_{\lambda} (x,y,z) = (\lambda^2 x, \lambda y ,z)$. Next, if
$T_{\lambda}^{\ast} X$ denotes the pull-back $T_{\lambda}^{\ast} X$, then the blow up of $X$ will be holomorphic
if $T_{\lambda}^{\ast} X$ converges to a holomorphic vector field as $\lambda \rightarrow 0$. The reader will have no difficulty
in working out the general case or to formulate equivalent conditions.

After this short introduction, we are ready to discuss the content of \cite{danielMcquillan}. Basically, this paper
consists of two parts, the first one relying heavily on Panazzolo's previous work \cite{P}. Recall that \cite{P}
provides a resolution of singularities (of real analytic foliations on $(\R^3,0)$) by means of sequences of weighted blow ups.
The first part of \cite{danielMcquillan} is devoted to showing that the algorithm of \cite{P} applies equally well in
the general case of holomorphic foliations on $(\C^3,0)$. As explained above, this provides a birational model
for the foliation in question on a space possessing orbifold-type singular points.
Furthermore, there is a natural notion
of elementary singular point for a (singular) foliation $\fol$ defined on this space. Namely, a singular point
of $\fol$ is said to be elementary if the foliation is represented by
a vector field with elementary singular points in an {\it orbifold coordinate}\, for the space. This result summarizes
the first part of \cite{danielMcquillan}.

In the second part of \cite{danielMcquillan}, the authors consider the problem of resolving the singular points of
the ambient space while keeping the singular points of the foliation elementary. Then they go on to show that this resolution
can always be found except when the singular point correspond to a $\Z /2\Z$-orbifold. Therefore, at least as far
as foliations are concerned, they manage to obtain a birational model for the foliation possessing only
$\Z /2\Z$-orbifold singular points and where all the singular points of the foliation in question are elementary.

Since resolution theorems are also of interest in the study of singular points of vector fields, rather than foliations,
it is natural to ask how the above procedure affects the divisor of zeros/poles of a vector field. This is, indeed,
a point that can easily be missed in \cite{danielMcquillan} since it very much hinges in a characteristic of the resolution
algorithm in \cite{P}, and we thank the referee for having clarified the issue for us. It turns out that the centers
of each weighted blow up used in \cite{P}, and reproduced in the first part of \cite{danielMcquillan}, are what is called
{\it strictly invariant with respect to the quasi-homogeneous filtration in question}, see \cite{P} for terminology.
This means that {\it the transform of holomorphic vector fields remains holomorphic}. Taking all these issues together,
the resolution theorem in \cite{danielMcquillan} can be formulated as follows.

\begin{theorem}
\label{danielJDG}
\noindent {\rm ({\bf [25]})}\hspace{0.1cm} Let $\fol$ be a singular holomorphic foliation on $(\C^3,0)$. There is a
sequence of weighted blow-ups
\begin{equation}
\fol_0 \stackrel{\Pi_1}\longleftarrow \fol_1 \stackrel{\Pi_2}\longleftarrow \cdots
\stackrel{\Pi_l}\longleftarrow \fol_l \label{weightedblowups}
\end{equation}
satisfying the following conditions:
\begin{itemize}
  \item The center of each weighted blow up is strictly invariant with respect to the quasi-homogeneous filtration in question
  (hence a holomorphic vector field remains holomorphic).

  \item The ambient space is an analytic space of dimension~$3$ whose singular points are $\Z /2\Z$-orbifold type
  and the total blow-up map $\Pi_1 \circ \cdots \circ \Pi_l$ is birational.

  \item The singular points of $\fol_l$ are elementary in orbifold coordinates.
\end{itemize}
\end{theorem}

Hence the basic difference between Theorem~\ref{danielJDG} and Theorem~A lies in the fact that Mcquillan-Panazzolo work in the
category of weighted blow ups while we stick with standard blow ups as much as possible.
In concrete applications, the choice of one statement over the other will probably go down to a matter of personal taste.
For example, it
is fair to claim that there is no fundamental reason to privilege standard blow ups over weighted ones so that
considering weighted blow ups increases your chances of finding a resolution in a smaller number of steps. It can then
be reckoned that in several applications of Seidenberg's theorem, the number of blow ups used is irrelevant while
having essentially a single type of blow up map might make it slightly easier to handle various types of index formulas.

There is an additional couple of differences between these two theorems that we believe are worth drawing the reader's
attention to. First, Panazzolo's paper \cite{P}, and hence \cite{danielMcquillan}, does provide an algorithm
to obtain a resolution model for a given foliation. In contrast to this, our Theorem~A is essentially not algorithmic
as it contains arguments based on contradictions arising from assuming the non existence of a sequence of blow ups
with certain required properties. The other difference is that, on the other hand, our Theorem~A answers the
question of deciding how close to a foliation with elementary singular points it is possible to go while {\it sticking}\,
with standard blow ups, which might be a natural curiosity for people with background in foliation/differential equations.

It should be clear that these differences between Theorem~\ref{danielJDG} and Theorem~A are, in many senses, minor ones.
So the feeling that the choice of which version to use will probably depend on the author's taste
or background seems to be somehow strengthened by them. A point, however, where everyone is likely to agree on is that
a new proof of an important result, as it is undoubtedly the case here, is always welcome and helps to increase the
general understanding of the problem in question.

\subsection{The structure of the paper and of proofs of Theorems~A and~B}\label{commentsonthmB}

Let us close this section by detailing the structure of this article and the inter-dependence of the different
sections.

We begin by recalling that a foliation $\fol$ can be
resolved (or it is resolvable) if there
is a finite sequence of blow ups (centered in singular sets) leading to a foliation all of whose singular points are
elementary. Recall also that the term {\it blow-up}\, always means a standard blow up. In particular, whenever
weights are used we explicitly refer to {\it weighted blow ups}.

As already indicated, the initial motivation of this paper was to prove Theorem~B. The idea to prove this theorem
was based on F. Cano's comment related to Proposition~\ref{BasedonCano_Roche_Spivakovsky}.
The sketch of the envisaged proof was as follows. let $\fol$
be a foliation tangent to a semicomplete vector field $X$ and assume aiming at a contradiction that $\fol$ cannot
be resolved. Owing to Proposition~\ref{BasedonCano_Roche_Spivakovsky}, there must exist a formal separatrix $S$ for $\fol$.
If this separatrix were convergent, we might restrict $X$ to $S$ and try to argue from some basic properties
of semicomplete vector fields as in \cite{JR1} {\it provided that} the restriction of $X$ to $S$ is not identically zero.
Having a merely formal separatrix, however, prevents us from making sense of the restriction of $X$ to $S$ (though, as will be seen,
it essentially rules out the inconvenient situation where the restriction of $X$ to $S$ vanishes identically).
To remedy for the formal character of $S$, it is natural to consider results about ``sectorial normalization'' for $\fol$
providing asymptotic estimates for its integral curves over conveniently chosen sectors. Naturally, for
this approach to be effective, some control about the ``angle'' of the mentioned sectors is needed.
In other words, it should be proved that the sector in question is ``large enough'', in some suitable sense.

Modulo taking for granted Proposition~\ref{BasedonCano_Roche_Spivakovsky}, the main difficulty to make the above argument
accurate clearly lies in obtaining a suitable sectorial normalization for the foliation $\fol$. As a matter of fact,
Ramis-Sibuya theorem in \cite{ramisS} provides a very general result on the existence of sectorial normalizations for formal
separatrices. Nonetheless, at this level of generality, it is virtually impossible to estimate the ``angle'' where the
integral curves satisfy the expected asymptotic conditions. A natural alternative was then to consider classical results
due to Malmquist: these come with suitable estimates for the ``angle'' of the section but they require the foliation
$\fol$ to have a particularly simple form, cf. \cite{Ince}, \cite{Malm}.

The assumptions in Malmquist theorem \cite{Malm} immediately led us to consider resolution procedures for $\fol$, namely the results
in \cite{danielMcquillan} and \cite{canorochetc}. The obstacle to apply \cite{canorochetc} was evident: their ``final models''
were still not ``simple enough'' to satisfy the conditions in Malmquist theorem.

Concerning the possibility of using McQuillan-Panazzolo theorem (Theorem~\ref{danielJDG}), we were with two issues.
The smaller issue had to do with the singularities associated with $\Z /2\Z$-orbifolds that cannot be turned into elementary
ones unless a blow up with weight~$2$ is performed. The characterization of these foliations presented in \cite{danielMcquillan}
is an invariant one while the use of Malmquist theorem and subsequent derivation of more ``quantitative'' information
requires slightly more explicit normal forms. Of course obtaining normal forms for these singularities from the characterization
provided in \cite{danielMcquillan} is rather straightforward so that this was not our main concern. Also, we mention
that the corresponding material is ``implicitly'' included in the present paper
(Sections~\ref{multiplicityalongseparatrix} and~\ref{persistentnilpotent-Section}).

On the other hand, we were more seriously concerned about the behavior of the divisor of zeros
of a vector field under the sequence of weighted blow ups provided in \cite{danielMcquillan}. Basically, at that point
in time, we had no confirmation of
the information presented in the first item of the statement of Theorem~\ref{danielJDG}. The issue, once again, stemmed from
our strategy to prove Theorem~B. More precisely, with the above material about elementary singularities for the
underlying foliation in place, Malmquist theorem becomes effective. Namely, this theorem yields
suitable asymptotic expansions whose ``angle'' is directly related to the order of the restriction to $S'$
(the transform of $S$) of a local representative for the foliation $\fol'$.
The nature of semicomplete vector fields, however, basically requires $X'$ to have
a non-empty divisor of zeros (and an empty divisor of poles) transverse to $S'$ for the desired contradiction to arise.
This explains our initial hesitation with respect to using weighted blow ups.

What precedes added to our general feeling that it would be nice to ``complete'' the work of Cano-Roche-Spivakovsky
\cite{canorochetc} to obtain a resolution theorem through Zariski classical approach. Besides,
Theorem~3 in \cite{canorochetc} already provided a characterization of
foliations that cannot be resolved which, albeit somewhat ``coarse'', looked promising in terms of enabling us
to prove Proposition~\ref{BasedonCano_Roche_Spivakovsky}.

The remainder of this paper will be devoted to properly implementing the above described strategy.

Section~\ref{multiplicityalongseparatrix} is very elementary and discusses the effect of blow ups on a sequence of singular
points determined by a formal separatrix along with its transforms. The basic idea is to consider the multiplicity of the
foliation along the separatrix in question and study the way this multiplicity varies under sequences of blow ups.
Whereas multiplicity of a foliation along a separatrix is a basic example of valuation, no general result on valuation
is required in the course of the discussion which requires only basic knowledge about blow ups.

Section~\ref{persistentnilpotent-Section} continues the discussion in
Section~\ref{persistentnilpotent-Section} and includes, in particular, the notion
of {\it persistent nilpotent singularity}. The main result of Section~\ref{persistentnilpotent-Section} being
precisely the characterization of persistent nilpotent singularity, namely Theorem~\ref{normalformpersistentnilpotent}.
The reader will not fail to note that our ``persistent nilpotent singularities'' correspond to the singularities
associated with $\Z /2\Z$-orbifold type singular points of \cite{danielMcquillan}. Furthermore, the normal form
provided by Theorem~\ref{normalformpersistentnilpotent} is equivalent to the invariant characterization of the latter
formulated in \cite{danielMcquillan}.

In Section~\ref{persistentnilpotent-Section} we also formulate Proposition~\ref{BasedonCano_Roche_Spivakovsky} in elementary
terms, i.e. avoiding any use of {\it valuations}. This section ends with the proof of Theorem~\ref{Microlocalversion_TheoremA}
obtained by combining the general discussion in Section~\ref{multiplicityalongseparatrix} with
Theorem~\ref{normalformpersistentnilpotent} and with Proposition~\ref{BasedonCano_Roche_Spivakovsky} (which is taken for
granted at this moment).
All this material is elementary and requires only some familiarity with (singular) holomorphic foliations at the level
of the first chapters of \cite{bookYakovenko}. Yet, we note that readers familiar with valuations and with
Piltant's theorem in \cite{Piltant} will probably be able to easily derive Theorem~A from Theorem~\ref{Microlocalversion_TheoremA}.

Section~\ref{semicompletestuff} is devoted to the proof of Theorem~B. This proof has basically four ingredients.
Besides Theorems~\ref{Microlocalversion_TheoremA} and~\ref{normalformpersistentnilpotent}, it also requires Malmquist theorem
on asymptotic expansion of solutions of certain systems of equations \cite{Malm} and, of course, some background revolving around
the notion of semicomplete singularity. Bar a more specific
result on semicomplete vector fields detailed in \cite{Adolfo}, the background material on semicomplete vector fields is covered
in \cite{JR1} and recalled in the beginning of Section~\ref{semicompletestuff} in an effort to make the discussion kind of
self-contained. As follows from the preceding, the background material required for this section is significantly
larger than for the previous ones, though the discussion is still mostly
accessible to readers familiar with the contents of \cite{bookYakovenko} and/or \cite{arnold}.

Section~\ref{semicompletestuff} is followed by Section~\ref{Examples_and_comments} which contains examples
illustrating the sharp nature of most of our results as well as some additional constraints on semicomplete
persistent nilpotent singularities arising from the holonomy of (formal) separatrices. In particular, in this
section we show how to extend the vector field
$$
Z = x^2 \partial /\partial x + xz \partial /\partial y + (y -xz) \partial /\partial z
$$
in a complete vector field on a suitable open complex manifold of dimension~$3$. We also provide some explicit examples
of persistent nilpotent singular points which cannot lead to any of the Sancho-Sanz vector fields by means of sequences
of blow ups (i.e. these examples cannot be ``further simplified'').

Finally almost all of Section~\ref{Provingtheproposition} is devoted to the proof of
Proposition~\ref{BasedonCano_Roche_Spivakovsky}. This section requires familiarity with resolution techniques based
on valuations, including Zariski
approach to the resolution problem. In more concrete and accurate terms, it requires a good knowledge of the
material contained in \cite{canorochetc}. The proof of Proposition~\ref{BasedonCano_Roche_Spivakovsky} then
completes the proof of our Theorem~\ref{Microlocalversion_TheoremA}. In turn, at this point,
Theorem~A becomes little more than a direct blend of Theorem~\ref{Microlocalversion_TheoremA} with
Piltant's patching theorem \cite{Piltant}.

\section{The multiplicity of a foliation along a separatrix}\label{multiplicityalongseparatrix}

For background in the material discussed below, the reader is referred to \cite{arnold} and to \cite{bookYakovenko}.
Consider a singular holomorphic foliation $\fol$ of dimension~$1$ defined on a neighborhood of the origin in $\C^3$. By definition,
$\fol$ is given by the local orbits of a holomorphic vector field $X$ whose singular set ${\rm Sing}\, (X)$ has codimension
at least~$2$. The vector field $X$ is said to be a {\it vector field representing $\fol$}. Albeit the representative
vector field $X$ is not unique, two of them differ by a multiplicative locally invertible function.
The singular set ${\rm Sing}\, (\fol)$ of
a foliation $\fol$ is defined as the singular set of a representative
vector field $X$ so that it has codimension greater than or equal to~$2$.

Conversely, with every (non-identically zero) germ of holomorphic vector field on $(\C^3, 0)$, it is associated
a germ of {\it singular holomorphic foliation $\fol$}. Up to eliminating non-trivial common factors of the components of $X$,
we can replace $X$ with another holomorphic vector field $Y$ whose singular set has codimension at least~$2$. The foliation
$\fol$ is then given by the local orbits of $Y$.
A global definition of singular (one-dimensional) holomorphic foliations can be formulated as follows.

\begin{defi}
Let $M$ be a complex manifold. A singular ($1$-dimensional) holomorphic foliation $\fol$ on $M$
consists of a covering $\{ (U_i, \varphi_i)\}$ of
$M$ by coordinate charts together with a collection of holomorphic vector fields $Z_i$ satisfying the following conditions:
\begin{itemize}
  \item For every $i$, $Z_i$ is a holomorphic vector field having singular set of codimension at least~$2$ which is
  defined on $\varphi_i (U_i) \subset \C^n$.

  \item Whenever $U_i \cap U_j \neq \emptyset$, we have $\varphi_i^{\ast}Z_i = g_{ij} \varphi_j^{\ast} Z_j$ for
  some nowhere vanishing holomorphic function $g_{ij}$ defined on $U_i \cap U_j$.
\end{itemize}
\end{defi}

Throughout this section and the next one, all blow ups of foliations (and of vector fields)
are {\it standard}. Moreover the centers of the blow ups are always contained in the singular set of
the foliation in question (associated foliation in the case of vector fields).

Consider now a holomorphic foliation $\fol$ defined on a complex manifold $M$ of dimension~$3$ and let $p \in M$
be a singular point of $\fol$. A {\it separatrix} (or analytic separatrix) for
$\fol$ at~$p$ is an irreducible analytic curve invariant by $\fol$,
passing through~$p$, and not contained in the singular set ${\rm Sing}\, (\fol)$ of $\fol$.
Along similar lines, a {\it formal separatrix} for
$\fol$ at~$p$ is a formal irreducible curve $S$ invariant by $\fol$ and centered at~$p$. In other words,
in local coordinates $(x,y,z)$ around $p$ where $\fol$ is represented by the vector field
$X = F \partial /\partial x + G \partial /\partial y + H \partial /\partial z$,
the formal separatrix $S$ is given by a triplet of formal series
$t \mapsto \varphi (t) = (\varphi_1 (t), \varphi_2 (t) , \varphi_3 (t))$ satisfying the following (formal) equations
\begin{equation}
\varphi_1'(t) (G \circ \varphi)(t) = \varphi_2'(t) (F \circ \varphi)(t) \; \;  {\rm and} \; \; \varphi_2'(t) (H \circ \varphi)(t) =
\varphi_3'(t) (G \circ \varphi)(t) \label{definitionseparatrix}
\end{equation}
where:
\begin{enumerate}
  \item $(F \circ \varphi)(t)$ (resp. $(G \circ \varphi)(t)$, $H \circ \varphi(t)$) stands for the formal series obtained by composing
  the Taylor series of $F$ (resp. $G$, $H$) at the origin with the formal series of $\varphi$ as indicated.

  \item In the preceding it is understood that at least one of the formal series $(F \circ \varphi)(t)$, $(G \circ \varphi)(t)$,
  and $(H \circ \varphi)(t)$ is not identically zero.
\end{enumerate}
Note that Puiseux theorem allows us to represent an analytic separatrix by a map of the form
$t \mapsto \varphi (t) = (\varphi_1 (t), \varphi_2 (t) , \varphi_3 (t))$ where the formal series $\varphi_i$ ($i=1,2,3$)
are actually convergent. Thus an analytic separatrix can be viewed as a particular case of a formal separatrix
and for this reason our terminology will be such that
whenever we refer to a {\it formal separatrix of $\fol$} the possibility
of having an actual analytic separatrix {\it will not}\, be excluded. If we need to emphasize that a formal separatrix is not analytic,
then we will say that the separatrix in question is {\it strictly formal}.
Finally, we also note that the second condition above is automatically
satisfied whenever $\varphi (t)$ is a strictly formal curve satisfying Equation~(\ref{definitionseparatrix}).

Consider again a singular point $p \in M$ of a holomorphic foliation $\fol$. Choose local coordinates $(x,y,z)$ around~$p$ and
assume that $\fol$ has a formal separatrix $S$ at $p$ which is given in the coordinates $(x,y,z)$ by the formal series
$\varphi (t) = (\varphi_1 (t), \varphi_2 (t) , \varphi_3 (t))$. Consider now a local holomorphic vector field $X$
defined around~$p$ and {\it tangent to}\, $\fol$ {\it but not necessarily}\, representing $\fol$.
Note that the (formal) pull-back of the restriction of $X$ to $S$ by $\varphi$ may be
considered since $S$ is a formal separatrix of $\fol$. This pull-back is a formal vector field in dimension~$1$ given by
\[
\varphi^{\ast} (X|_S) = g(T) \frac{\partial}{\partial T}
\]
where $g$ satisfies
\begin{equation}\label{conditon_g}
(X \circ \varphi)(T) = g(T) \varphi^{\prime}(T)
\end{equation}
as formal series.

We recall the classical notion of multiplicity of a foliation along a separatrix which is also well known as a
basic example of valuation.

\begin{defi}
\label{orderalongseparatrix}
The multiplicity of $X$ along $S$ is the order of the formal series $g$ at $0 \in \C$
\[
{\rm mult}\, (X,S) = {\rm ord}(g,0) \, .
\]
In other words, setting $g(T) = \sum_{k \geq 1} g_k T^k$, ${\rm mult} (\fol,S)$ is the
smallest positive integer $k \in \N^{\ast}$ for which $g_k \ne 0$. This multiplicity equals zero if and only if
the series associated with $g(T)$ vanishes identically.

In turn, the multiplicity of $\fol$ along $S$, ${\rm mult}\, (\fol,S)$, is defined as the multiplicity
along $S$ of a vector field $X$ representing $\fol$ around~$p$. Since $S$ as a separatrix for $\fol$ is not
contained in the singular set of $\fol$, the multiplicity of $\fol$ along $S$ is never equal to zero.
\end{defi}

It is immediate to check that the notions above are well defined in the sense that they depend
neither on the choice of coordinates nor on the choice of the representative vector field $X$.

We begin with a simple albeit important lemma. To fix notation, we will say that $S$ is a (formal) separatrix
for a vector field $X$ if $S$ is a (formal) separatrix for the foliation $\fol$ associated with~$X$. We also
recall that the centers of all blow-ups considered in what follows are contained in the singular sets of the
corresponding foliations.

\begin{lemma}\label{ordervf}
The multiplicity of a vector field along a formal separatrix is invariant by blow-ups (with centers contained in the
singular set of the foliation in question).
\end{lemma}

\begin{proof}
The statement means that the multiplicity of a vector field along a formal separatrix is invariant by blow-ups regardless
of whether they are centered at a singular point or at a (locally) smooth analytic curve contained in the singular set of $\fol$.
We will prove the mentioned invariance in the case of blow-ups centered at a point.
The case of blow-ups centered at analytic curves is analogous and thus left to the reader.

Let $X$ be a holomorphic vector field defined on a neighborhood of the origin of $\C^3$ and admitting a formal
separatrix $S$. Let $\pi : \widetilde{\C}^3 \rightarrow \C^3$ denote the blow-up map
centered at the origin and denote by $\tS$ the transform of $S$ by $\pi$.

%
%
%
%
%

Fixed standard $(x,y,z)$-coordinates on $\C^3$, the formal separatrix $S$ is given by a formal map
of the form $T \mapsto \varphi (T) = (\varphi_1 (T), \varphi_2 (T), \varphi_3(T))$.
Without loss of generality, we assume that
$\varphi$ takes on the form $\varphi (T) = (T^m + {\rm h.o.t}, T^n + {\rm h.o.t}, T^p)$
for $p \leq m, \, n$. In turn, the vector field $X$ is given by
\[
X = F(x,y,z) \frac{\partial}{\partial x} + G(x,y,z) \frac{\partial}{\partial y} + H(x,y,z) \frac{\partial}{\partial z} \, .
\]
Since $S$ is a formal solution of the differential equation associated with $X$, there follows that $\varphi^{\prime}(T)$
and $X \circ \varphi$ satisfy Equation~(\ref{definitionseparatrix}).
Comparing the last component of $\varphi^{\prime}(T)$ and of $(X \circ
\varphi)(T)$ we conclude that the multiplicative function $g$ appearing in Equation~(\ref{conditon_g}) is of order equal
to ${\rm ord}(H \circ \varphi,0) - p + 1$. Thus
\[
{\rm mult}\,(X, S) = {\rm ord} (H \circ \varphi,0) - p + 1 \, .
\]

Let us now compute the order of $\tX$ along $\tS$. For this we consider affine coordinates $(u,v,z)$ where
the blow-up map is given by $\pi(u,v,z) = (uz,vz,z) = (x,y,z)$. The transform of $X$ then becomes
$\tX = (1/z) Z$ where $Z$ is the vector field given by
\[
Z = ( F(uz,vz,z) - u H(uz,vz,z)) \partial /\partial x +
(G(uz,vz,z) - v H(uz,vz,z)) \partial /\partial y + zH(uz,vz,z) \partial /\partial z  \, .
\]
In turn, the transform of $S$ by $\pi$ is by definition the formal curve given by
$\psi (T) = \pi^{\ast} \varphi (T) = (T^{m-p} + {\rm h.o.t}, T^{n-p} + {\rm h.o.t}, T^p)$.
Since $\tS$ is a formal solution of the differential equation associated to $\tX$, there follows again that $\psi^{\prime}
(T)$ and $(\tX \circ \psi) (T)$ satisfy Equation~(\ref{definitionseparatrix}). By comparing their last components, we conclude that
\begin{align*}
{\rm mult}\,(\tX, \tS) &= {\rm ord} \left(H((T^{m-c} + {\rm h.o.t})T^c, (T^{n-c} + {\rm h.o.t})T^c, T^p),0\right) - p + 1 \\
&= {\rm ord} \left(H(T^{m} + {\rm h.o.t}, T^{n} + {\rm h.o.t}, T^p),0\right) - p + 1 \\
&= {\rm ord} (H \circ \varphi,0) \\
&= {\rm mult} \,(X, S) \, .
\end{align*}
The lemma is proved.
\end{proof}

Consider again a foliation $\fol$ defined on a neighborhood of the origin of $\C^3$ along with a formal separatrix
$S$. Whereas Lemma~\ref{ordervf} asserts that the multiplicity of a vector field along a formal separatrix is invariant
by blow-ups, the analogous statement does not necessarily hold for a foliation. Indeed,
let $X$ be a vector field representing $\fol$ around $(0,0,0) \in \C^3$ so that the zero-set of $X$ has codimension
at least~$2$. Finally let $\tX$ denote the pull-back of $X$ by the blowing-map centered at the
origin. If $X$ has order at least~$2$ at the origin, then the singular set of $\tX$ has codimension~$1$ since $\tX$ vanishes
identically on the corresponding exceptional divisor. More precisely,
in the affine coordinates $(u,v,z)$ where $x=uz$ and $y=vz$, we have $\tX = z^{\alpha} Z$ for a certain (holomorphic) vector
field $Z$ having singular set of codimension at least~$2$ and a certain integer $\alpha \geq 1$. In fact,
if $k$ stands for the order of $X$ at $0 \in \C^3$, then we have $\alpha = k$ or $\alpha = k-1$ according to whether or not
the origin is a dicritical singular point.
Here we remind the reader that a singular point is said to be {\it dicritical}\, if the exceptional divisor, given by $\{z=0\}$
in the above affine coordinates, is not invariant by $\tilf$. Next, note that the multiplicity of $X$ along $S$ coincides with the
multiplicity of $\tX$ along $\tS$ (Lemma~\ref{ordervf}). However, the multiplicity of $\tilf$ along $\tS$ is not the
multiplicity of $\tX$ along $\tS$
but rather the multiplicity of $Z$ along $\tS$. More precisely, we have
\begin{align*}
{\rm mult} \, (\tilf,\tS) &= {\rm mult} \, (Z,\tS) \\
&= {\rm mult} \, (\tX, \tS) - {\rm ord} \, (z^{\alpha} \circ \psi, 0) \\
&= {\rm mult} \, (X,S) - {\rm ord} \, (z^{\alpha} \circ \psi, 0) \, ,
\end{align*}
where $\psi$ stands for the triplet of formal series associated with $\tS$. Since the zero-set of $X$
has codimension at least~$2$, there also follows that ${\rm mult} \, (X,S) = {\rm mult} \, (\fol,S)$.
Summarizing, we have proved the following:

\begin{prop}\label{prop_multiplicity}
Let $\fol$ be a holomorphic foliation on $(\C^3,0)$ admitting a formal separatrix $S$.
If $\fol$ has order at least~$2$ at the origin, then
\[
{\rm mult} \, (\tilf, \tS) < {\rm mult} (\fol, S) \, ,
\]
where $\tilf$ (resp. $\tS$) stands for the transform of $\fol$ (resp. $S$) by the one-point blow-up centered at the origin.
\qed
\end{prop}

In order to state the analogue of Proposition~\ref{prop_multiplicity} for blow-ups centered at smooth (irreducible)
curves contained in ${\rm Sing}\, (\fol)$, a notion of order for $\fol$ with respect to the curves in question is needed.
A suitable notion can be introduced as follows.

Recall first that the order of the foliation $\fol$ at the origin is defined as the degree of the first non-zero homogeneous
component of a vector field $X$ representing $\fol$. The mentioned degree, as well as all of the corresponding non-zero
homogeneous component, may be recovered through the family of homotheties $\Gamma_{\dl} : (x,y,z) \mapsto (\dl x, \dl y, \dl z)$.
More precisely, the degree is simply the unique positive integer $d \in \N$ for which
\begin{equation}\label{homogeneous_component_prime}
\lim_{\dl \rightarrow 0} \frac{1}{\dl^{d-1}} \Gamma_{\dl}^{\ast} X
\end{equation}
is a non-trivial vector field. Furthermore, the non-trivial vector field obtained as this limit is exactly the first
non-zero homogeneous component of $X$. We shall adapt this construction to define the order of $\fol$ over
a curve.

Let then $C$ be a smooth curve contained in ${\rm Sing}\, (\fol)$. Our purpose is to define the order of $\fol$
with respect to $C$. Clearly there are local coordinates $(x,y,z)$ in which the curve in question coincides with
the $z$-axis, i.e. it is given by $\{x=y=0\}$ (as usual we only perform blow-ups centered at smooth curves; naturally this is not
a very restrictive condition since every curve can be turned into smooth by the standard resolution procedure).
The blow-up centered at $\{x=y=0\}$
is equipped with affine coordinates $(x,t,z)$ and $(u,y,z)$ where the corresponding blow-up map is given by $\pi_z (x,t,z) = (x,tx, z)$
(resp. $\pi_z (u,y,z) =(uy, y, z)$).
Consider now the family of automorphisms given by
\[
\Lambda_{\dl} : (x, y, z) \mapsto (\dl x, \dl y, z).
\]
The {\it order of $\fol$ with respect to $C$} (or the order of $\fol$ over $C$) is defined as the unique
integer $d \in \N$ for which
\begin{equation}\label{homogeneous_component}
\lim_{\dl \rightarrow 0} \frac{1}{\dl^{d-1}} \Lambda_{\dl}^{\ast} X
\end{equation}
yields a non-trivial vector field. Note that this integer~$d$ may be seen as the degree of $X$ with respect to the variables
$x, \, y$. In fact, assume that in coordinates $(x,y,z)$ the vector field $X$ is given by $X = X_1(x,y,z) \partx + X_2(x,y,z)
\party + X_3(x,y,z) \partz$. The pull-back of $X$ by $\Lambda_{\dl}$ becomes
\[
\Lambda_{\dl}^{\ast} X = \frac{1}{\dl} \left[ X_1(\dl x, \dl y, z) \frac{\partial }{\partial x} + X_2(\dl x, \dl y, z) \frac{\partial
}{\partial y}\right] + X_3(\dl x, \dl y, z) \frac{\partial }{\partial z}
\]
Denote by $k$ (resp. $l$) the maximal power of $\dl$ that divides $X_1(\dl x, \dl y, z) \partial /\partial x + X_2(\dl x, \dl y, z)
\partial /\partial y$ (resp. $X_3(\dl x, \dl y, z) \partial /\partial z$). The order $d$ defined above is simply the minimum between
$k$ and $l+1$.

The analogue of Proposition~\ref{prop_multiplicity} for blow-ups centered at smooth (irreducible) curves can now be stated
as follows.

\begin{prop}\label{prop_multiplicity_curves}
Let $\fol$ be a holomorphic foliation on $(\C^3,0)$ admitting a formal separatrix $S$. Let
$\tilf$ (resp. $\tS$) stands for the strict transform of $\fol$ (resp. $S$) by the blow-up centered at a smooth (irreducible)
curve contained in ${\rm Sing}\, (\fol)$. If $\fol$ has order at least~$2$ with respect to the blow-up center, then
\[
{\rm mult} \, (\tilf, \tS) < {\rm mult} (\fol, S) \, .
\]
\qed
\end{prop}

Let us close this section with a first application of Proposition~\ref{prop_multiplicity} to the reduction of singular points
(a slightly more general discussion involving Proposition~\ref{prop_multiplicity_curves} appears in
Section~\ref{persistentnilpotent-Section}). Let $\fol$ be a holomorphic
foliation defined on a neighborhood of the origin of $\C^3$ and let $X$ be a holomorphic vector field representing $\fol$. Recall that
a singular point $p$ of $\fol$ is said to be {\it elementary}\, if the linear part of $X$ at $p$, $DX(p)$, has at least one eigenvalue
different from zero. In the sequel, whenever there is no risk of misunderstanding, we will say that a
singular point $p$ is {\it nilpotent} if the linear part of $X$ at $p$ is {\it nilpotent
and non-zero}. Along similar lines, the expression {\it degenerate singularity} will be used to refer to singularities
where the linear part of $X$ is actually equal to zero.

Consider now a singular foliation $\fol_0$ along with a formally smooth separatrix $S_0$ at the origin ($(0,0,0) \simeq p_0$).
Consider the blow-up $\fol_1$ of $\fol_0$ centered at the origin. The transform $S_1$ of $S_0$ selects a singular point $p_1$ of $\fol_1$
in the exceptional divisor $\Pi_1^{-1} (0,0,0)$. In fact, if the point $p_1 \in \Pi_1^{-1} (0,0,0)$ selected by $S_1$ were regular
for $\fol_1$, then $\Pi_1^{-1} (0,0,0)$ would not be invariant by $\fol_1$ and the formal separatrix $S_1$ (and hence $S$) would
actually be analytic and $\fol_1$ would be regular on a neighborhood of $p_1$: this situation is excluded in what follows.

Next let $\fol_2$ be the blow-up of $\fol_1$ at $p_1$. Again the transform $S_2$ of $S_1$ will select a singular point
$p_2 \in \Pi_2^{-1} (p_1)$ of $\fol_2$. The procedure is then continued by induction so as to produce a (infinite)
sequence of foliations $\fol_n$
\begin{equation}
\fol_0 \stackrel{\Pi_1}\longleftarrow
\fol_1  \stackrel{\Pi_2}\longleftarrow \cdots
\stackrel{\Pi_n}\longleftarrow
\fol_n  \stackrel{\Pi_{n+1}}\longleftarrow \cdots \label{resolutionsequence-example1}
\end{equation}
along with singular points $p_n$ and formal separatrices $S_n$.

\begin{lemma}\label{lemma_nilpotent}
Consider a sequence of foliations $\fol_n$ as in~(\ref{resolutionsequence-example1}) along with a sequence of
formal separatrices $S_n$ and singular points $p_n$.
Assume that for every $n \in \N$, $p_n$ is not an elementary singular point of $\fol_n$. Then there exists $n_0 \in \N$
such that $p_n$ is nilpotent (non-zero) singularity of $\fol_n$ for every $n \geq n_0$.
\end{lemma}

\begin{proof}
The statement follows from Proposition~\ref{prop_multiplicity}. Indeed, by assumption, $p_n$ is not an elementary singular
point of $\fol_n$ (for every $n \in \N$). Assume, in addition, that $\fol_1$ is not nilpotent at $p_1$. This means
that the order of $\fol_1$ at $p_1$ is at least~$2$ so that the multiplicity of $\fol_2$ along $S_2$ is strictly
smaller than the multiplicity of $\fol_1$ along $S_1$. If $\fol_2$ is again non-nilpotent at $p_2$, then the order
of $\fol_2$ at $p_2$ is again at least~$2$. There follows that the multiplicity of $\fol_3$ along $S_3$ is strictly
smaller than the multiplicity of $\fol_2$ along $S_2$. When the procedure is continued, the
multiplicity of $\fol_{n+1}$ along $S_{n+1}$ will be strictly smaller than the multiplicity of $\fol_n$ along
$S_n$ whenever $p_n$ is not a nilpotent singularity of $\tilf_n$. Hence we
obtain a decreasing - though not necessarily strictly decreasing - sequence of non-negative integers. This sequence
must eventually become constant. If $n_0$ is the index
for which the sequence is constant for $n \geq n_0$, then Proposition~\ref{prop_multiplicity} ensures that
$\fol_n$ has order~$1$ at $p_n$ for every $n \geq n_0$. Since by assumption $p_n$ is not an elementary singularity
of $\fol_n$, we conclude that $p_n$ must be a nilpotent singularity of $\fol_n$ for $n \geq n_0$. The lemma
is proved.
\end{proof}

\section{On persistent nilpotent singularities}\label{persistentnilpotent-Section}

Throughout this section by {\it nilpotent singularity}\, it is always meant a singular foliation whose linear part
is {\it nilpotent and different from zero}.

Our purpose is to discuss nilpotent singular points that are persistent under blow-up transformations
and this will lead to the two main results of the section, namely Theorem~\ref{normalformpersistentnilpotent} and
Theorem~\ref{Microlocalversion_TheoremA}.
As mentioned, Theorem~\ref{normalformpersistentnilpotent} generalizes, in a relatively straightforward way,
the celebrated examples of vector fields obtained by Sancho and Sanz. Also, they are related to the $\Z /2\Z$-orbifold
singular points discussed in the last section of \cite{danielMcquillan} (see Section~\ref{Reviewliterature}).

First let us make it clear what is meant by {\it persistent nilpotent singularity}.
In the sequel, the centers
of the blow-ups maps are always {\it contained in the singular set of the foliation}. Moreover, they are
either a single point or a smooth analytic curve. The reader is also reminded that all blow-ups are assumed to be standard.

Let $\fol_0$ denote a singular foliation along with an
irreducible formal separatrix $S_0$ at a chosen singular point $p_0$. Consider a sequence of blow-ups and
transformed foliations which is obtained as follows. First, we choose a center $C_0$ with $p_0 \in C_0$ which is
contained in the singular set of $\fol_0$. Then we blow-up $\fol_0$ with center $C_0$ and let $\fol_1$ denote the blown-up
foliation. The transform $S_1$ of $S_0$ selects a point $p_1$ in the exceptional divisor
$\Pi_1^{-1} (C_0)$. In the case where $p_1$ is regular for $\fol_1$ the sequence of blow-ups stops at this level.
Otherwise, $p_1$ is a singular point for $\fol_1$ and another blow-up will be performed.
Let $C_1$ be a center contained in the singular set of $\fol_1$ and such that $p_1 \in C_1$.
The blow-up of $\fol_1$ with center $C_1$ leads to a foliation $\fol_2$.
Again the transform $S_2$ of $S_1$ will select a point $p_2 \in \Pi_2^{-1} (C_1)$. If $p_2$ is a regular point for
$\fol_2$, then the sequence of blow-ups stops. Otherwise we consider the blow-up of $\fol_2$ with a center $C_2$ passing through
$p_2$. The procedure is then continued by induction so as to produce a sequence of foliations $\fol_n$
\begin{equation}
\fol_0 \stackrel{\Pi_1}\longleftarrow
\fol_1  \stackrel{\Pi_2}\longleftarrow \cdots
\stackrel{\Pi_n}\longleftarrow
\fol_n  \stackrel{\Pi_{n+1}}\longleftarrow \cdots \label{resolutionsequence-example2}
\end{equation}
along with singular points $p_n$ and formal separatrices $S_n$. The mentioned sequence is finite if there exists $n \in \N$
such that $p_n$ is regular for $\fol_n$. A sequence of points $p_n$ obtained from a formal separatrix $S_0$ as
above is often called a {\it sequence of infinitely near singular points}.

\begin{defi}
\label{definitionpersistentnilpotent}
With the preceding notation, assume that $p_0$ is a nilpotent singular point for $\fol_0$.
The point $p_0$ is said to be a persistent nilpotent singularity if there exists a formal separatrix
$S_0$ of $\fol_0$ such that for every sequence of blowing-ups as in~(\ref{resolutionsequence-example2})
the following conditions are satisfied:
\begin{itemize}
  \item[($\imath$)] The singular points $p_n$ (selected by the transformed separatrices $S_n$) are all
  nilpotent singular points for the corresponding foliations;

  \item[($\imath \imath$)] The multiplicity ${\rm mult} (\fol_n, S_n)$ of $\fol_n$ along $S_n$ does not depend on~$n$.
\end{itemize}
\end{defi}

\begin{rem}
\label{nilpotentarisingatdicriticaldivisor}
{\rm Note that Condition~($\imath$) implies Condition~($\imath \imath$) if the blow-up is centered at the
nilpotent singular point in question. In the case of blow-ups centered at smooth curves, however,
it is possible to have a strictly smaller multiplicity, cf. the proof of Lemma~\ref{furthersimplification-2}.}
\end{rem}

In view of Seidenberg's theorem, persistent nilpotent singularities do not exist in dimension~$2$.
In dimension~$3$, however, the examples of singularities that cannot be resolved by blow-ups with invariant
centers found by Sancho and Sanz satisfy the conditions in Definition~\ref{definitionpersistentnilpotent}.
In fact, Sancho and Sanz have shown
that the foliation associated with the vector field
\[
X = x\left( x \frac{\partial}{\partial x} - \alpha y \frac{\partial}{\partial y}  - \beta z \frac{\partial}{\partial z}\right)
+ xz \frac{\partial}{\partial y} + (y-\lambda x) \frac{\partial}{\partial z}
\]
possesses a strictly formal separatrix $S =S_0$ such that for every sequence of blowing-ups
as above, the corresponding sequence of infinitely near singular points consists of nilpotent singularities.
Furthermore the foliations $\fol_n$ also
satisfy ${\rm mult}\, (\fol_n, S_n) =2$ for every $n \in \N$, where $S_n$ stands for the transform of $S$.
The set of persistent nilpotent singular points
is thus non-empty. Most of this section will be devoted to the characterization of these persistent singularities
and the final result will be summarized by Theorem~\ref{normalformpersistentnilpotent}. We begin with
the following proposition:

\begin{prop}\label{prop_normal form}
Let $\fol$ be a singular holomorphic foliation on $(\C^3,0)$ and assume that the origin is a persistent nilpotent singularity
of $\fol$. Then, up to finitely many one-point blow-us, there exist local coordinates and a
holomorphic vector field $X$ representing $\fol$ and having the form
\begin{equation}\label{nilpotent_normal_form_1}
(y + f(x,y,z)) \frac{\partial}{\partial x} + g(x,y,z) \frac{\partial}{\partial y} + z^n \frac{\partial}{\partial z}
\end{equation}
for some $n \geq 2 \in \N$ and some holomorphic functions $f$ and $g$ of order at least~$2$ at the origin. Moreover
the orders of the functions $z \mapsto f(0,0,z)$ and of $z\mapsto g(0,0,z)$ can be made arbitrarily large (in
particular greater than $2n$).
\end{prop}

\begin{proof}
Let $\fol$ be a nilpotent persistent singular point and denote by $S$
a formal separatrix giving rise to a sequence of infinitely near singular points as in
Definition~\ref{definitionpersistentnilpotent}.
Up to finitely many one-point blow-ups the formal separatrix $S$ can be assumed to be
smooth in the formal sense. Up to performing an additional one-point blow-up, we may also
assume that $\fol$ admits an analytic smooth {\it invariant surface} which is, in addition,
transverse to the formal separatrix $S$. In fact, the resulting exceptional divisor is necessarily
invariant under the transformed foliation since the previous singular point is nilpotent and non-zero
(recall that the exceptional divisor is not invariant by the blown-up foliation if and only if the first
non-zero homogeneous component is a multiple of the radial vector field).

Note also that even if we consider the blow-up of $\fol$ centered at a smooth analytic curve contained
in the singular set of the foliation (rather than the one-point blow-up)
the resulting exceptional divisor is still invariant by the transformed foliation. The argument is similar
to the preceding one: if this component were not invariant, then the first non-zero homogeneous component of the foliation
{\it with respect to this curve}\,
would be a multiple of vector field $x\partx + y\party$ at generic points in this center. Again this cannot happen
since the origin is a nilpotent singular point.
Finally, denoting by $E_n$ the total exceptional divisor associated with the birational transformation $\Pi_{\circ n} =
\Pi_n \circ \cdots \circ \Pi_2 \circ \Pi_1$. The argument above also applies to ensure that every irreducible component of the
exceptional divisor is invariant by the corresponding foliation $\fol_n$.

Summarizing, we can assume without loss of generality that all of the following holds:
\begin{itemize}
\item the formal separatrix $S$ is smooth;
\item $\fol$ possesses a (smooth) analytic invariant surface $E$;
\item the formal separatrix $S$ is transverse to $E$ (in the formal sense).
\end{itemize}

In view of the preceding, consider
local coordinates $(x,y,z)$ around $p_0$ where the smooth invariant surface $E$ is given by $\{z=0\}$.
Denote by $H$ a formal change of coordinates preserving $\{z=0\}$ as
invariant surface and taking the formal separatrix $S$ to the $z$-axis given by $\{x=0, \, y=0\}$. Since the vector
field obtained by conjugating $X$ through $H$ is merely formal, let $H_m$ denote
the polynomial change of coordinates obtained by truncating $H$ at order $m$. We then set
\[
Y_m = (DH_m)^{-1} (X \circ H_m) \, .
\]
The map $H_m$ is holomorphic and so is the vector field $Y_m$. Denote by $\fol_m$
the foliation associated with $Y_m$. The foliation $\fol_m$ clearly admits a formal separatrix $S_m$ whose order
of tangency with the $z$-axis goes to infinity with $m$ and thus can be assumed
arbitrarily large.

Under the above conditions, the vector field $Y_m$ has the form
\[
Y_m = A(x,y,z) \frac{\partial }{\partial x} + B(x,y,z) \frac{\partial }{\partial y} + C(x,y,z) \frac{\partial }{\partial z}
\]
with
\begin{itemize}
\item[(1)] $C(x,y,z) = z^n + g(z) + xP(x,y,z) + yQ(x,y,z)$, for some $n \in \N$, some holomorphic functions $P$ and $Q$ divisible
by $z$ and some holomorphic function $g$ divisible by $z^{n+1}$;
\item[(2)] $A(0,0,z)$ and $B(0,0,z)$ having order arbitrarily large, say greater than $2n$.
\end{itemize}
Note that the value of $n = {\rm ord} \, (C(0,0,z)) \geq 2$ depends only on the initial foliation $\fol$ and not on the choice
of $m \in \N^{\ast}$. In fact, the value of~$n$ is nothing but the multiplicity of $\fol$ along $S$ and hence
it is invariant by (formal) changes of coordinates. The orders of $A(0,0,z)$
and of $B(0,0,z)$ depend however on $m$. Note that the orders in question are related to the contact
order between $S_m$ and the $z$-axis. In particular these orders can be made arbitrarily large as well.

Naturally the foliation $\fol$ and $\fol_m$ are both nilpotent at the origin. Next we have:

\noindent {\it Claim}. Up to a linear change of coordinates in the variables $x,y$,
the linear part of $Y_m$ is given by $y \partial /\partial x$.

\noindent {\it Proof of the Claim}. The formal Puiseux parametrization $\varphi$ of $S_m$ has
the form $\varphi(T) = (T^r + {\rm h.o.t.}\, , \, T^s + {\rm h.o.t.}\, , \, T)$ where the integers
$r$ and $s$ are related to the contact order between $S_m$ and the $z$-axis. In particular both
$r$ and $s$ can be made arbitrarily large.
Now it is clear that both $\partial A/\partial z$ and $\partial B/\partial z$ must vanish at the origin
provided that $\varphi$ is invariant by the vector field $Y_m$. On the other hand, $\partial C/\partial x$ and
$\partial C/\partial y$ are both zero at the origin since $P$ and $Q$ are divisible by~$z$ (cf. condition~(1) above).
It is also clear that $\partial C/\partial z$ equals zero at the origin since $n \geq 2$. Thus both
the third line and the third column in the matrix representing the linear part of $Y_m$ at the origin
are entirely constituted by zeros. Using again the fact that
this matrix is nilpotent, the standard Jordan form ensures that a linear change of coordinates involving only the
variables $x,y$ brings the linear part of $Y_m$ to the form $y \partial /\partial x$. It is also immediate to check
that this linear change of coordinates does not affect the previously established conditions and/or normal forms.
The claim is proved.\qed

\smallbreak

Consider now the blow-up of $\fol$ centered at the origin. In coordinates $(u,v,z)$ where $(x,y,z) = (uz,vz,z)$, the
transform $\widetilde{Y}_m$ of $Y_m$ by the mentioned blow-up is given by
\[
\widetilde{Y}_m = \widetilde{A}(u,v,z) \frac{\partial }{\partial x} + \widetilde{B}(u,v,z) \frac{\partial }{\partial y} +
\widetilde{C}(u,v,z) \frac{\partial }{\partial z}
\]
where
$$
\widetilde{A}(u,v,z) = \frac{A(uz,vz,z) - uC(uz,vz,z)}{z} \; \; \; {\rm and} \; \; \;
\widetilde{B}(u,v,z) = \frac{B(uz,vz,z) - vC(uz,vz,z)}{z}
$$
and where $\widetilde{C}(u,v,z) = C(uz,vz,z)$.
In particular $\tilf_m$ is nilpotent at the origin, with the same linear part as $\fol_m$.
Furthermore the above formulas easily imply all of the following:
\begin{itemize}
\item[(a)] the order of $\widetilde{C}(0,0,z)$ coincides with the order of $C(0,0,z)$;
\item[(b)] the maximal power of $z$ dividing $\widetilde{C}(u,v,z) - z^n - g(z)$ is strictly greater than the maximal
power of $z$ dividing $C(x,y,z) - z^n - g(z)$;
\item[(c)] ${\rm ord} \, \widetilde{A}(0,0,z) = {\rm ord} \, A(0,0,z) - 1$ and
${\rm ord} \, \widetilde{B}(0,0,z) = {\rm ord} \, B(0,0,z) - 1$.
\end{itemize}
Also the transform $\tS_m$ of $S_m$ is a formal separatrix tangent to the $z$-axis. The tangency order is still
large (at least $2n$) since the order in question is related to the orders of $\widetilde{A}(0,0,z)$ and of
$\widetilde{B}(0,0,z)$ and these orders fall only by one unity (item~(c)). In turn, the multiplicity of $\tilf_m$
along $\tS_m$ coincides with the multiplicity of $\fol_m$ along $S_m$ (from item~(a)). Finally the function $C$ was
divisible by $z$. Now, according to item~(b), $\widetilde{C}$ is divisible by $z^2$. In fact, item~(b) ensures
that after at most $n$~one-point blow-ups, the corresponding singular point is still a nilpotent singularity
for which the component of the representative
vector field in the direction transverse to the exceptional divisor (given in local coordinates by $\{z=0\}$) has
the form $z^n I (u,v,z)$ where $I (u,v,z)$ is a holomorphic function satisfying $I(0,0,0) \ne 0$.
Dividing all the components of the vector field in question by $I$ then yields another representative vector field
with the desired normal form. The proposition is proved.
\end{proof}

An additional simplification can be made on the normal form~(\ref{nilpotent_normal_form_1}) of
Proposition~\ref{prop_normal form}. Namely:

\begin{lemma}
\label{furthersimplification-1}
Up to performing an one-point blow-up, the functions $f$ and $g$ in~(\ref{nilpotent_normal_form_1}) become divisible by $z$.
\end{lemma}

\begin{proof}
Again let $\pi$ denote the blow-up map centered at the origin
and set $\tX = \pi^{\ast} X$. In the above mentioned affine coordinates $(u,v,z)$, we have
$$
\tX = (y + \widetilde{f}(x,y,z)) \frac{\partial}{\partial x} + \widetilde{g}(x,y,z) \frac{\partial}{\partial y} + z^n \frac{\partial}{\partial z} \, ,
$$
where $\widetilde{f} = (f(uz,vz,z) - uz^n)/z$ and $\widetilde{g} = (g(uz,vz,z) - vz^n)/z$.
The functions $\widetilde{f}, \, \widetilde{g}$ are thus divisible by $z$ since
$f$ and $g$ have order at least~$2$ at the origin. The lemma follows.
\end{proof}

Next, we are going to determine conditions on the functions $f$ and $g$ for the singular point $p_0 \simeq (0,0,0)$ to be a
persistent nilpotent singularity. Thus let $\fol$ be the foliation associated with a vector field
$X$ having the normal form provided by Proposition~\ref{prop_normal form} and by Lemma~\ref{furthersimplification-1}.
In other words, the vector field $X$ is given by
\begin{equation}\label{nilpotent_normal_form_2}
X = (y + zf(x,y,z)) \frac{\partial}{\partial x} + zg(x,y,z) \frac{\partial}{\partial y} + z^n
\frac{\partial}{\partial z} \, ,
\end{equation}
where $f$ and $g$ are holomorphic functions of order at least~$1$ and $n \in \N$, with $n \geq 2$.
Let $S$ be a smooth formal separatrix of $\fol$
giving rise to the persistent nilpotent singular point (see Definition~\ref{definitionpersistentnilpotent}). Without loss
of generality, the contact order~$k_0 (\geq 4)$ between $S$ and the $z$-axis is assumed to be large
and, similarly, $f(0,0,z)$ and $g(0,0,z)$ are assumed to have order bigger than $2n \geq 4$.

Note that the curve locally given by $\{y=0, z=0\}$ coincides with the singular set of $\fol$. We are then allowed to
perform either an one-point blow-up centered at $p_0 \simeq (0,0,0)$ or a blow-up centered at the mentioned curve.
Now we have:

\begin{lemma}
\label{furthersimplification-2}
Assume that $\fol$ has a persistent nilpotent singularity at the origin and let $S$ denote the
corresponding formal separatrix. Then $g(x,0,0) = \lambda x + {\rm h.o.t.}$ for some constant $\lambda \in \C^{\ast}$.
\end{lemma}

\begin{proof}
Denote by $X_1$ (resp. $\fol_1$, $S_1$)
the transform of $X$ (resp. $\fol$, $S$) by the blow-up map $\pi_1$ centered on $\{y=0, z=0\}$. In local
coordinates $(x,v,z)$ where $y=vz$ the vector field $X_1$ is given by
\begin{equation}
X_1 = (vz + zf(x,vz,z)) \frac{\partial}{\partial x} + (g(x,vz,z) - vz^{n-1}) \frac{\partial}{\partial v} +
z^n \frac{\partial}{\partial z} \, . \label{VectorFieldX_1}
\end{equation}
Note that $g(x,0,0)$ does not vanish identically, otherwise $g(x,vz,z)$ would be divisible by~$z$ and hence the
vector field $X_1$ would vanish identically over the exceptional divisor (locally given by $\{ z=0\}$). This is impossible
since the multiplicity of $\fol_1$ along the transform of~$S$ would be strictly smaller than the multiplicity of
$\fol$ along $S$, hence contradicting Condition~($\imath \imath$) in Definition~(\ref{definitionpersistentnilpotent}).
In particular, the singular set of $\fol_1$ is locally given by $\{x=0, z=0\}$.

On the other hand, the formal separatrix $S_1$ is still tangent to the (transform of the)
$z$-axis since the contact between $S$ and the
(initial) $z$-axis was greater than~$2$ (in fact the contact between $S_1$ and the $z$-axis
is at least $k_0-1 \geq 3$). Hence, in the affine coordinates $(x,v,z)$, the foliation $\fol_1$ must
have a nilpotent singularity at the origin. Combining the conditions that $f(0,0,0) =0$, $n \geq 2$, and the fact
that the order of $g(0,0,z)$ is greater than~$2n$, the preceding
implies that $\partial g /\partial x$ does not vanish at the origin. In other words, $g(x,0,0) = \lambda x + {\rm h.o.t.}$
as desired.
\end{proof}

Note that the proof above also yields the following sort of converse to Lemma~\ref{furthersimplification-2}.

\begin{lemma}
\label{furthersimplification-3}
Keeping the preceding notation,
let $\fol$ be given by a vector field~$X$ as in~(\ref{nilpotent_normal_form_2}) and assume that $S$ is a formal
separatrix of $\fol$ with contact at least~$3$ with the $z$-axis. Assume that $g(x,0,0) = \lambda x + {\rm h.o.t.}$,
with $\lambda \neq 0$. Then the blow-up $\fol_1$ of $\fol$ centered at the curve $\{y=0, z=0\}$ has a nilpotent
singularity at the point of the exceptional divisor selected by $S_1$.\qed
\end{lemma}

Continuing the discussion of Lemma~\ref{furthersimplification-2}, consider again the vector field
$X_1$ in~(\ref{VectorFieldX_1}). We perform the blow-up centered at the curve locally given by $\{x=0, z=0\}$ - which
is contained in the the singular set of $\fol_1$ - and
denote by $X_2$ (resp. $\fol_2$, $S_2$) the transform of $X_1$ (resp. $\fol_1$, $S_1$).
In affine coordinates $(u,v,z)$ with $x = uz$ we have
\begin{equation}
X_2 = (v + f(uz,vz,z) - uz^{n-1}) \frac{\partial}{\partial x} + (g(uz,vz,z) - vz^{n-1}))
\frac{\partial}{\partial v} + z^n \frac{\partial}{\partial z} \, . \label{VectorFieldX_2}
\end{equation}
The contact of $S_2$ with the $z$-axis equals $k_0-1 \geq 3$ as follows from a simple computation
(cf. also Remark~\ref{WhynoonepointblowupConditions}). In particular the formal separatrix $S_2$ is still
based at the origin of the coordinates $(u,v,z)$ and it is tangent to the $z$-axis.
Since the above formula shows that $\fol_2$ has a nilpotent
singularity at the origin, we conclude:

\begin{lemma}
The foliation $\fol_2$ (resp. vector field $X_2$) has a nilpotent singularity at the point of the exceptional
divisor selected by $S_2$ (identified with the origin of the coordinates $(u,v,z)$).\qed
\end{lemma}

The reader will also note that the singular set of $\fol_2$ is still locally given by the curve $\{v=0,z=0\}$ which
clearly contains the origin. As already mentioned, $S_2$ is tangent to the $z$-axis.

\begin{obs}
\label{WhynoonepointblowupConditions}
{\rm Let us point out that $X_2$ locally coincides with the transform of $X$ by the one-point blow-up centered at
the origin. In this sense, to include blow-ups centered at curves
in the current discussion does not lead to additional conditions
to have nilpotent singular points.}
\end{obs}

Consider again a vector field $X$ having the form~(\ref{nilpotent_normal_form_2}). In the course of
the preceding discussion, it was seen that the vector fields obtained through two
successive blow-ups centered over the corresponding curves of singular points are respectively given by
\begin{equation}
X_1 = zr(x,v,z) \frac{\partial}{\partial x} + (x + zs(x,v,z)) \frac{\partial}{\partial v} +
z^n \frac{\partial}{\partial z} \, , \label{X_1-equation1}
\end{equation}
and by
\begin{equation}
X_2 = (v + zf_{(1)}(u,v,z)) \frac{\partial}{\partial u} + zg_{(1)}(u, v,z)
\frac{\partial}{\partial v} + z^n \frac{\partial}{\partial z} \, , \label{X_2-equation2}
\end{equation}
where $r, \, s, \, f_{(1)}$, and $g_{(1)}$ are all holomorphic functions vanishing at the origin of the corresponding
coordinates. As usual the coordinates $(x,v,z)$ are determined by $(x,y,z) = (x , vz ,z)$ while
$(x,y,z) = (uz, vz, z)$. Furthermore the functions $f_{(1)}$ and
$g_{(1)}$ satisfy
$$
f_{(1)}(u,v,z) = \frac{f(uz,vz,z) - uz^{n-1}}{z} \quad \text{and} \quad g_{(1)}(u,v,z) = \frac{g(uz,vz,z) - vz^{n-1}}{z} \, .
$$
The following relations arise immediately:
\begin{itemize}
\item[(1)] ${\rm ord} \, r(0,0,z) = {\rm ord} \, f(0,0,z)$ and ${\rm ord} \, s(0,0,z) = {\rm ord} \, g(0,0,z)$;

\item[(2)] ${\rm ord} \, f_{(1)}(0,0,z) = {\rm ord} \, f(0,0,z)-1$ and
${\rm ord} \, g_{(1)}(0,0,z) = {\rm ord} \, g(0,0,z)-1$;

\item[(3)] $\partial g_{(1)} / \partial u (0,0,0) = \partial g / \partial x (0,0,0) = \lambda \ne 0$.
\end{itemize}

Assume now that the above procedure is continued, i.e. assume $X_2$ is blown-up at the origin
(identified with the singular point selected by the transform of the initial formal separatrix).
As pointed out in Remark~\ref{WhynoonepointblowupConditions}, here it is convenient to keep in mind that
two consecutive blow-ups centered at curves contained in the singular set of the corresponding foliation
can be replaced by a single one-point blow-up. To continue the
procedure requires us to introduce new affine coordinates for each of these blow-ups and, in doing so,
notation is likely to become cumbersome. To avoid this, and since the computations are similar to the previous ones,
let us abuse notation and write $(x,y,z)$ for the coordinates $(u,v,z)$: naturally these ``new'' coordinates $(x,y,z)$
have little to do with the initial ones. Similarly, coordinates for the first blow-up are then given by $(x,v,z)$
while the second blow-up possesses coordinates $(u,v,z)$. Assuming these identifications are made at every step - i.e.
at every pair of blow-ups as indicated above - let $X_{2i}$ denote the vector field obtained after $i$-steps
where $i$ satisfies $i < k_0$ (recall that $k_0$ stands for the contact order of the formal separatrix
with the ``initial $z$-axis'').
In the (final) coordinates $(u,v,z)$, the vector field $X_{2i}$ takes on the form~(\ref{nilpotent_normal_form_2})
\[
X_{2i} = (v + zf_{(i)}(u,v,z)) \frac{\partial}{\partial u} + zg_{(i)}(u,v,z) \frac{\partial}{\partial v} +
z^n \frac{\partial}{\partial z} \, ,
\]
with $\partial g_{(i)}/\partial u (0,0,0) = \partial g/\partial x(0,0,0) = \lambda \neq 0$.
In more accurate terms, recall that the orders of $f_{(i)}(0,0,z)$
and of $g_{(i)}(0,0,z)$ are directly related to the contact order of the transform of the formal separatrix $S$ with
the corresponding $z$-axis. At every step (consisting of a pair of blow-ups), the orders of
$f_{(i)}(0,0,z)$ and of $g_{(i)}(0,0,z)$ decrease by one unity so that we have
\[
{\rm ord} \, (f_{(i)}(0,0,z)) = {\rm ord} \, (f(0,0,z)) - i \quad \text{and} \quad
{\rm ord} \, (g_{(i)}(0,0,z)) = {\rm ord} \, (g(0,0,z)) - i \, .
\]
Thus, for $i \geq k_0 ={\rm min} \{{\rm ord} \, (f(0,0,z)), {\rm ord} \, (g(0,0,z))\}$, the vector field
$X_{2i}$ no longer takes on the form~(\ref{nilpotent_normal_form_2}). At first sight this might suggest that
the initial nilpotent singularity may fall short of being persistent, yet it is exactly the opposite that is true: the
singularity is necessarily persistent.

To explain the last claim, we begin by observing that the $z$-axis is not intrinsically
determined by Formula~(\ref{nilpotent_normal_form_2}). In fact, the $z$-axis is only subject to
having some {\it high}\, contact order with the formal smooth separatrix $S$ and it is~$S$ - rather than
the $z$-axis - that has an intrinsic nature in our discussion. In particular, if $S$ were analytic, we
could make $S$ coincide with the $z$-axis which, in turn,
would imply that all the functions $f_{(i)}(0,0,z)$ and $g_{(i)}(0,0,z)$ vanish identically. It would then
follow at once that the singularity is persistent.

\begin{rem}
\label{natureofcenters}
{\rm It should be emphasized that our definition of persistent singularities requires
the centers of all the blow-ups
to be contained in the singular set of the corresponding foliations. This accounts for the difference between
choosing centers that are contained in the singular set and the slightly weaker condition of allowing {\it invariant centers}.
An analytic separatrix of a foliation is a legitimate invariant center so that, if we are allowed to perform
blow-ups with invariant centers, the preceding singularity would be turned into an elementary one by blowing-up the
foliation along the separatrix in question. This explains why in the example of Sancho and Sanz they want the
corresponding separatrix to be strictly formal and further illustrates the analogous comments made in the
Introduction.}
\end{rem}

We now go back to the vector field $X_{2i}$ which no longer has the form~(\ref{nilpotent_normal_form_2}).
To show that this vector field corresponds to a persistent nilpotent singularity when the separatrix $S$ is strictly
formal we will construct a change of coordinates where the vector field $X_{2}$ still takes on
the form~(\ref{nilpotent_normal_form_2})
but where the orders of the ``new'' functions $z \mapsto f_{(1)}(0,0,z)$ and $z \mapsto g_{(1)}(0,0,z)$ increase
strictly so as to restore the values of the initial orders. The desired change of coordinates can be made
polynomial by truncating a certain {\it formal}\, change of coordinates as
in the proof of Proposition~\ref{prop_normal form}. This is the content of
Lemma~\ref{lemma_normal_form_higher_tang_order} below.

Consider the vector field $X_{2}$ given in $(u,v,z)$ coordinates by Formula~(\ref{X_2-equation2}) along with
the initial vector field $X$ given in $(x,y,z)$ coordinates by Formula~(\ref{nilpotent_normal_form_2}).

\begin{lemma}\label{lemma_normal_form_higher_tang_order}
There exists a polynomial change of coordinates $H$ having the form $(u,v,z) = H(\tilde{x}, \tilde{y}, z) =
(h_1 (\tilde{x},z), h_2 (\tilde{y}, z), z)$
where the vector field $X_2$ becomes
\[
X_2 = (\tilde{y} + z\bar{f}(\tilde{x},\tilde{y},z)) \frac{\partial}{\partial \tilde{x}} +
z\bar{g}(\tilde{x}, \tilde{y},z) \frac{\partial}{\partial \tilde{y}} + z^n \frac{\partial}{\partial z}
\]
with
\begin{itemize}
\item[(a)] ${\rm ord} \, (\bar{f}(0,0,z)) \geq  {\rm ord} \, (f(0,0,z))$ and
${\rm ord} \, (\bar{g}(0,0,z)) \geq {\rm ord} \, (g(0,0,z))$;

\item[(b)] $\partial \bar{g}/\partial \tilde{x} (0,0,0) = \partial g/\partial x (0,0,0)$.
\end{itemize}
\end{lemma}

\begin{proof}
Denote by $S_2$ the transform of $S$ through the one-point blow-up centered at the origin which is therefore
a formal separatrix for the foliation associated with $X_2$. Since $S_2$ is smooth and tangent to the $z$-axis,
it can be (formally) parameterized by the
variable~$z$. In other words, $S_2$ is given by $\varphi(z) = (f(z), g(z), z)$ for suitable formal series
$f$ and $g$ with zero linear parts.

Consider now the formal map given in local coordinates $(\tilde{x},\tilde{y},z)$ by
$H(\tilde{x}, \tilde{y},z) = (\tilde{x}-f(z), \tilde{y}-g(z),z)$. The linear part of $H$ at the origin is
represented by the identity matrix so that $H$ is a formal change of coordinates which, in addition, preserves
the plane $\{z=0\}$. Furthermore, $H$
takes the formal separatrix $S_2$ to the $z$-axis. As previously mentioned, the formal vector field obtained by
conjugating $X_2$ through $H$ is strictly formal if $S_2$ is strictly formal. So, let $H_m$ stand for the polynomial
change of coordinates obtained from $H$ by truncating it at order~$m$ and let
$Y_m = (DH_m)^{-1} \, (X \circ H_m)$. Clearly the
map $H_m$ is holomorphic and so is the vector field $Y_m$. Moreover the
foliation $\fol_m$ associated to $Y_m$ possesses a formal separatrix $T_m$, whose tangency order with the $z$-axis
goes to infinity with~$m$.

It is straightforward to check that the vector field $Y_m$ has the form
\[
Y_m = (\tilde{y} + zf_m(\tilde{x},\tilde{y},z)) \frac{\partial}{\partial \tilde{x}} +
zg_m(\tilde{x},\tilde{y},z) \frac{\partial}{\partial \tilde{y}} + z^n \frac{\partial}{\partial z} \,
\]
with $\partial g_m /\partial \tilde{x} (0,0,0) = \partial g /\partial x (0,0,0)$, for every $m \in \Z$. Furthermore,
for $m$ sufficiently large we have ${\rm ord} \, (f_m(0,0,z)) \geq {\rm ord} \, (f(0,0,z))$ and ${\rm ord} \, (g_m(0,0,z))
\geq {\rm ord} \, (f(0,0,z))$ as well. The lemma is then proved.
\end{proof}

The results of this section can now be summarized as follows:

\begin{theorem}
\label{normalformpersistentnilpotent}
Let $\fol$ be a singular holomorphic foliation on $(\C^3,0)$ and assume that the origin is a persistent nilpotent singularity
of $\fol$. Let $S$ denote the corresponding formal separatrix of $\fol$. Then, up to finitely many one-point blow-ups, the
foliation $\fol$ is represented by a vector field $X$ having the form
\begin{equation}\label{normal_fom_teo}
(y + zf(x,y,z)) \frac{\partial}{\partial x} + zg(x,y,z) \frac{\partial}{\partial y} + z^n \frac{\partial}{\partial z}
\end{equation}
for some $n \in \N$, $n \geq 2$, and some holomorphic functions $f$ and $g$ of order at least~$1$ such that
\begin{itemize}
\item[(a)] The separatrix $S$ is tangent to the $z$-axis. In fact, the contact order of $S$ and the $z$-axis can
be made arbitrarily large. Equivalently the orders of $f(0,0,z)$ and of $g(0,0,z)$ are arbitrarily large.

\item[(b)] $\partial g / \partial x (0,0,0) \ne 0$.
\end{itemize}
Conversely, every nilpotent foliation $\fol$ represented by a vector field $X$ as above and
possessing a (smooth) formal separatrix $S$ tangent to the $z$-axis gives rise to a persistent nilpotent singularity.\qed
\end{theorem}

To close the section, let us accurately state Proposition~\ref{BasedonCano_Roche_Spivakovsky} so as to derive
Theorem~\ref{Microlocalversion_TheoremA}.

\begin{prop}
\label{BasedonCano_Roche_Spivakovsky}
Let $\fol$ be a (germ of) singular $1$-dimensional foliation defined on a neighborhood of the origin
in $\C^3$. Assume that $\fol$ cannot be transformed into a foliation all of whose singular points are
elementary by means of a finite sequence of blow-ups with centers contained in the singular set of the corresponding foliations.
Then, up to performing a finite sequence of blow-ups as above,  the foliation $\fol$
possesses a formal separatrix $S$ giving rise to a sequence of infinitely near singular points
and such that none of the points in this sequence is elementary.
\end{prop}

\begin{proof}[Proof of Theorem~\ref{Microlocalversion_TheoremA}]
Assume that $\fol$ is singular foliation on $(\C^3, 0)$ whose singularity cannot be resolved by blow-ups
centered at the singular set of $\fol$. Owing to Proposition~\ref{BasedonCano_Roche_Spivakovsky}, let
$S$ denote a formal separatrix of $\fol$ giving rise to a sequence of (non-elementary)
infinitely near singular points. Next
apply a sequence of one-point blow-ups to $S$ and to its transform. Since the multiplicity of the corresponding
foliations along the transforms of $S$ form a monotone decreasing sequence, this sequence becomes constant
after a finite number of steps. Denoting by $\fol_k$ (resp. $S_k$) the corresponding foliation (resp. transform
of $S$), there follows that $\fol_k$ has a nilpotent singular point at the point in the exceptional divisor
selected by $S_k$. Furthermore, owing to Remark~\ref{WhynoonepointblowupConditions}, the multiplicity in question
does decrease even if blow-ups centered at (smooth) singular curves are allowed. Thus the mentioned nilpotent
singular point of $\fol_k$ is a persistent one, i.e. it satisfies the conditions
in Definition~\ref{definitionpersistentnilpotent}. The results of the present section can now be
applied to this singular point and the statement follows from Theorem~\ref{normalformpersistentnilpotent}.
\end{proof}

\section{Proof of Theorem B}\label{semicompletestuff}

This section is devoted to the proof of Theorem~B and of its corollaries while Section~\ref{Examples_and_comments}
will contain some examples complementing our main results as well as a sharper version of the normal form given
in Theorem~\ref{normalformpersistentnilpotent} which is valid for foliations tangent to semicomplete vector fields.

Let $X$ be a holomorphic vector field defined on an open set $U$ of some complex manifold.
According to~\cite{JR1}, $X$ is said to be semicomplete on $U$ if for every point $p \in U$ there exists a connected
domain $V_p \subseteq \C$ with $0 \in V_p$ and a holomorphic map $\phi_p: V_p \to U$ satisfying the following conditions:
\begin{itemize}
\item $\phi_p(0) = 0$ and $\frac{d\phi}{dt}|_{t=t_0} = X(\phi_p(t_0))$.

\item For every sequence $\{t_i\} \subseteq V_p$ such that $\lim_{i \to +\infty} t_i \in \partial V_p$, the sequence
$\{\phi_p(t_i)\}$ leaves every compact subset of $U$.
\end{itemize}

We also refer to~\cite{JR1} for the basic properties of semicomplete vector fields used in the sequel.

First, a vector field that is semicomplete on $U$ is semicomplete on every open set $V \subseteq U$ as well.
In particular, the notion of {\it germ of semicomplete vector field}\, makes sense. Furthermore, if $X$ is a complete
vector field on a complex manifold $M$, then the germ of $X$ at every singular point is necessarily a germ of
semicomplete vector field.

There is a useful criterion (Proposition~\ref{prop_timeform})
to detect vector fields that fail to be semicomplete which is as follows. Let
$X$ be a holomorphic vector field defined on an open set $U$ and denote by $\fol$ its associated (singular)
holomorphic foliation. Consider a leaf $L$ of $\fol$ which {\it is not}\, contained in the zero set of
$X$. Then leaf $L$ is then a Riemann surface naturally equipped with a meromorphic
abelian $1$-form $dT$ dual to $X$ in the sense that $dT.X = 1$ on $L$. The $1$-form $dT$ is often
referred to as the {\it time-form}\, induced by $X$ on $L$. The following proposition is taken from
\cite{JR1}.

\begin{prop}\label{prop_timeform}
Let $X$ be a holomorphic semicomplete vector field on an open set $U$. Let
$L$ be a leaf of the foliation associated with $X$ on which the time-form $dT$ is defined (i.e. $L$
is not contained in the zero set of $X$). Then we have
\[
\int_c dT \ne 0
\]
for every path $c:[0,1] \to L$ (one-to-one) embedded in $L$.\qed
\end{prop}

The main result of this section is the following theorem.

\begin{theorem}\label{teo_semi-complete}
Let $X$ be a holomorphic vector field on $(\C^3,0)$ and denote by $\fol$ its associated foliation. Assume that the
origin is a persistent nilpotent singularity for $\fol$ and let $S$ denote the corresponding formal separatrix
of $\fol$. Assume also that at least one the following holds:
\begin{itemize}
\item The multiplicity  ${\rm mult} \, (\fol,S)$ of $\fol$ along $S$ is at least~$3$;
\item The linear part $J^1_0 X$ of $X$ at the origin equals zero.
\end{itemize}
then $X$ is not semicomplete on a neighborhood of the origin.
\end{theorem}

Our approach to Theorem~\ref{teo_semi-complete} begins with a couple of remarks. First, it is convenient
to remind the reader of the difference between vector fields and foliations in terms of the
dimension of their singular sets. In other words, a vector field may have non-trivial common factors
among its components so as to give rise to a divisor of zero while singular sets of foliations are always
of codimension at least~$2$. Recall also that a vector field $Y$ is said to be a representative
of the foliation $\fol$ if $Y$ is tangent
to $\fol$ and has singular set of codimension at least~$2$.

In view of the preceding, we will often write a holomorphic vector field $X$ under the form
$X = hY$ where $h$ is a holomorphic function and $Y$ is a holomorphic vector field with singular set
of codimension at least~$2$. With this notation, the vector field $Y$ is a (local) representative of the foliation
$\fol$ associated with $X$.

On a different note, it should also be pointed out that
the semicomplete character of a holomorphic vector field is preserved
under birational transformations. In particular, it is preserved under blow-ups. It is, however, not necessarily
preserved under {\it weighted}\, blow-ups if these {\it are regarded as finite-to-one maps}\, rather than from
the birational point of view associated with the orbifold action. Incidentally, blow-ups with weight~$2$ will be needed
in what follows.

Let us now fix a holomorphic vector field $X$ on $(\C^3,0)$ whose associated foliation $\fol$ is as in
Theorem~\ref{Microlocalversion_TheoremA}. Since blow-ups preserve the semicomplete
character of vector fields, up to transforming $X$ through finitely many one-point blow-ups, we can assume
that $\fol$ has a persistent nilpotent singularity with a formal separatrix $S$
giving rise to a sequence of infinitely near (nilpotent) singular points.

Summarizing what precedes, we can assume the existence
of local coordinates $(x,y,z)$ where $X$ is given by (cf. Sections~\ref{multiplicityalongseparatrix} and~\ref{persistentnilpotent-Section})
\begin{equation}\label{normal_form_sc}
X = z^k h(x,y,z) \left[(y + zf(x,y,z)) \frac{\partial}{\partial x} + zg(x,y,z)
\frac{\partial}{\partial y} + z^n \frac{\partial}{\partial z} \right]
\end{equation}
for suitable nonnegative integers $k, n$ and holomorphic functions $f$, $g$, and $h$ satisfying all of the following:
\begin{itemize}
\item $n \geq 2$ and $k \geq 0$;
\item $f (0,0,0) = g (0,0,0) =0$. Furthermore the orders at $0 \in \C$ of $f(0,0,z)$ and $g(0,0,z)$
are arbitrarily large, say larger than $2n$ (in particular $S$ is tangent to the $z$-axis);
\item $\partial g (0,0,0) /\partial x =\lambda  \ne 0$;

\item If $h (0,0,0) =0$, then every irreducible component of the set $\{ h=0\}$ is smooth, contains the separatrix
$S$, and {\it is not}\, invariant under $\fol$.
\end{itemize}

The above assertion involving the irreducible components of $\{ h=0\}$ requires a couple of comments. Naturally,
every analytic surface that {\it does not contain}\, the separatrix $S$ can be separated from $S$ by a suitably
chosen sequence of blow-ups. Similarly, these components can be made smooth without loss of generality. Finally,
for the fact that none of them is invariant under $\fol$, we refer to Corollary~\ref{Non-invariant-set_of_zeros}
which is a by-product of the proof of Proposition~\ref{BasedonCano_Roche_Spivakovsky} in the appendix.

Theorem~\ref{teo_semi-complete} is thus reduced to proving that $X$, as given in Formula~(\ref{normal_form_sc}),
is not semicomplete on a neighborhood of the origin provided that at least one of the following conditions holds:
the linear part of $X$ at the origin $J^1_0 X$ equals zero (equivalently $k \geq 1$) or
$n = {\rm mult} \, (\fol, S) > 2$, where $\fol$ stands for the foliation associated with $X$.

Let us begin by showing that $\fol$ can be resolved by using a single blow-up of weight~$2$. Here these
weight~$2$ blow-ups will be viewed as a two-to-one map. Note also that the lemma below includes some useful
explicit formulas for the transformed vector field.

\begin{lemma}\label{lemma_ramified_blowup}
Let $X$ be as in Formula~(\ref{normal_form_sc}) and denote by
$\Pi$ the blow-up of weight~$2$ centered at the curve $\{ y=z=0\}$ (the curve of singular points
of $\fol$). Let $\Pi^{\ast} \fol$ be the transform of $\fol$. Then the singular point of $\Pi^{\ast} \fol$
selected by $S$ in the exceptional divisor is elementary and the corresponding eigenvalues of $\Pi^{\ast} \fol$
are $0$, $1$, and $-1$. Furthermore
the transform $\Pi^{\ast} X$ of $X$ is a holomorphic vector field vanishing with order $2k+1$ on the exceptional divisor.
\end{lemma}

\begin{proof}
Let $(x,y,z)$ be the local coordinates where $X$ is given by Formula~(\ref{normal_form_sc}) and consider the indicated
weight~$2$ blow-up map $\Pi$. In natural coordinates $(u,v,w)$ the map $\Pi$ is given by
$$
\Pi (u,v,w) = (u, vw, w^2) \, ,
$$
where $\{ w=0\}$ is contained in the exceptional divisor.
Now a straightforward computation shows that $\Pi^{\ast} X$ is given in the $(u,v,w)$ coordinates by
\begin{align*}
\Pi^{\ast} X  &=  w^{2k} h \left[ \left( vw  +  w^2f(u,vw,w^2) \right) \frac{\partial}{\partial u}
+  \left(  wg(u,vw,w^2) - \frac{1}{2}vw^2  \right) \frac{\partial}{\partial v} +
\frac{1}{2}w^{2n-1}  \frac{\partial}{\partial w} \right] \\
& =  w^{2k+1} h \left[ (v + wf(u,vw,w^2)) \frac{\partial}{\partial u} +
\left( g(u,vw,w^2) - \frac{1}{2}vw \right) \frac{\partial}{\partial v} +
\frac{1}{2}w^{2n-2} \frac{\partial}{\partial w} \right] \, ,
\end{align*}
where the function $h$ is evaluated at the point $(u,vw,w^2)$. Since $h(0,0,0) \neq 0$, there follows
that the zero-divisor of $\Pi^{\ast} X$ locally coincides with the exceptional divisor (given by $\{ w=0\}$).
Moreover the order of vanishing of $\Pi^{\ast} X$ at the exceptional divisor is $2k+1$.
In turn, the foliation $\Pi^{\ast} \fol$ is represented by the vector field
\begin{equation}
Y = (v + wf(u,vw,w^2)) \frac{\partial}{\partial u} + \left( g(u,vw,w^2) - \frac{1}{2}vw \right) \frac{\partial}{\partial v}
+ \frac{1}{2}w^{2n-2} \frac{\partial}{\partial w} \label{Thevectorfield-Y-justtobequoted}
\end{equation}
whose linear part at the origin is given by $v \partial /\partial u + \lambda u \partial /\partial v$ since
$f(0,0,0) =0$ and $n \geq 2$ (here $\lambda = \partial g (0,0,0) /\partial x \neq 0$).
Thus the eigenvalues of $\fol$ at the origin are~$0$ and the two square roots of $\lambda$ which is clearly
equivalent to having eigenvalues $0$, $1$, and $-1$. The lemma is proved.
\end{proof}

Singularities of foliations on $\C^3$ possessing a single eigenvalue equal to~$0$ are called {\it codimension~$1$}\,
saddle-nodes. Semicomplete vector fields whose associated foliation is a codimension~$1$ saddle-node were studied in
detail in \cite{helena}. However, the case in question where the non-zero eigenvalues belong to the Siegel domain is
not covered in the paper in question. Also, in what follows, we will need more specific results, partially due to
the fact that the corresponding vector field is not necessarily semicomplete (cf. below). Yet, the reader will note
that our argument to prove
Theorem~\ref{teo_semi-complete} overlaps non-trivially with the ideas in \cite{helena}.

Recall that our purpose is to show that $X$ is not semicomplete on a neighborhood of $(0,0,0)$
provided that $k \ne 0$ or $n>2$. Note that, in general, this conclusion does not immediately follow from
proving that the vector field $\Pi^{\ast} X$ {\it is not} semicomplete on a neighborhood of the origin of
the coordinates $(u,v,w)$ since the map $\Pi$ is not one-to-one. It is, in fact, easy to construct examples
of semicomplete vector fields whose transforms under ramified coverings are no longer semicomplete. Yet, in
our context, the situation can be described in a more accurate form. Consider a regular leaf $L$
of the foliation $\Pi^{\ast} \fol$ which is equipped with the time-form $dT_{\Pi^{\ast}X}$ induced
by $\Pi^{\ast} X$. Assume that $c:[0,1] \to L$ is open path over which the integral of
$dT_{\Pi^{\ast}X}$ equals zero so that, in particular, $\Pi^{\ast} X$ is not semicomplete
(Proposition~\ref{prop_timeform}). If $X$
happens to be semicomplete, then we must necessarily have $\Pi (c(0)) = \Pi (c(1))$. Hence, the idea to prove
Theorem~\ref{teo_semi-complete} will be to find open paths $c$ satisfying the following two conditions:
\begin{itemize}
  \item $c : [0,1] \rightarrow L$ is contained in a leaf $L$ of $\Pi^{\ast} \fol$ and verifies $\Pi (c(0)) \neq \Pi (c(1))$;

  \item The integral of the time-form $dT_{\Pi^{\ast}X}$ over $c$ is equal to zero.
\end{itemize}
If $c$ is a path as above, then its projection by $\Pi$ is still an open path contained in
a leaf of $\fol$. Furthermore, the integral of the corresponding time-form {\it induced by $X$}\, over
$\Pi (c)$ is zero so that $X$ cannot be semicomplete.

Before proceeding further, it is convenient to recall the notion
of function asymptotic to a formal series. Let then $t \in \C$ be a variable and considers a formal series
$\psi (t)$. Consider also a circular sector $V$ of angle $\theta$, vertex at $0 \in \C$, and small radius.
A holomorphic function $\psi_V$ defined on $V \setminus \{ 0\}$ is said to be {\it asymptotic (on $V$)}\, to
the formal series $\psi (t)$ if for every~$i \in \N$ and for every sector $W \subset V$, of angle strictly
smaller than $\theta$ and sufficiently small radius, there exists a constant ${\rm Const}_{i,W}$ such that
$$
\Vert  \psi_V (t) -\psi_i (t) \Vert \leq {\rm Const}_{i,W} \;  \Vert t \Vert^{i+1} \, ,
$$
where $\psi_i$ stands for the $i^{\rm th}$-jet of $\psi$ at $0 \in \C$. The adaptation of the above
definition to vector-valued formal series $\psi (t) = (\psi_1 (t) , \ldots ,\psi_n (t))$ and functions
$\psi : V \rightarrow \C^n$ is straightforward and thus left to the reader.

The following lemma appears in \cite{ghys} (Lemma~3.12).

\begin{lemma}\label{asymptotic_vfield}
Let $V \subset \C$ denote a circular sector with vertex at $0 \in \C$ and angle $2\pi/l$,
where $l$ is a strictly positive integer. Assume that $\rho$ is a holomorphic function on $V \setminus \{ 0 \}$
such that
$$
\Vert \rho (x) - x^{l+2} \Vert \leq {\rm Const} \, \Vert x^{l+3} \Vert
$$
for a suitable constant ${\rm Const}$. Then for every $r > 0$, there exists an open path
$c$ embedded in the intersection of $V$ with the disc of radius~$r$ and center at $0 \in \C$
such that the integral of the $1$-form $dx/\rho (x)$ equals zero.
\end{lemma}

\begin{proof}
It suffices to sketch the argument and refer to \cite{JR1} for the detail concerning the effect of
higher order terms. Consider first the special case where $\rho (x) = x^{l+2}$. In this case the $1$-form
$dx/\rho (x)$ admits the function $x \mapsto -1/(l+1)x^{l+1}$ as primitive. Thus it is enough to choose
a path $c$ of the form $c(t) = x_0 e^{2i\pi t/(l+1)}$ where $x_0$ has small absolute value and is such that
the resulting path $c$ is still contained in $V$.

In the general case, the leading term of $\rho (x)$ is $x^{l+2}$. In fact, for $\Vert x \Vert$ small,
the difference $\Vert \rho (x) - x^{l+2} \Vert$ is bounded by ${\rm Const} \, \Vert x^{l+3} \Vert$ which is of order
larger than $x^{l+2}$ itself. The statement then follows by using the ``perturbation'' technique in \cite{JR1}.
\end{proof}

The proof of Theorem~\ref{teo_semi-complete} is divided in two cases according to whether or not we have
$h(0,0,0) \neq 0$.

\begin{proof}[Proof of Theorem~\ref{teo_semi-complete} when $h(0,0,0) \neq 0$]
With the notation of Lemma~\ref{lemma_ramified_blowup}, consider
the vector field $\Pi^{\ast} X$ and note that $\Pi^{\ast} X = w^{2k+1} h(u,uw,w^2) Y$ where $Y$ is given by
\begin{equation}
Y = (v + wf(u,vw,w^2)) \frac{\partial}{\partial u} + \left( g(u,vw,w^2) -
\frac{1}{2}vw \right) \frac{\partial}{\partial v} + \frac{1}{2}w^{2n-2} \frac{\partial}{\partial w} \, ,
\label{Foliation_afterweigt2_blowup}
\end{equation}
for suitable $k, n, f, g$, and $h$ as above. In particular, the vector field $Y$ is a representative of
the foliation $\Pi^{\ast} \fol$. Fixed a neighborhood
$U$ of the origin, we look for leaves $L$ of $\Pi^{\ast} \fol$ along with open paths
$c: [0,1] \to L$ contained in $U$ such that the two conditions below are satisfied:
\begin{itemize}
\item $\int_c dT_{\Pi^{\ast}X} = 0$;
\item $\Pi(c(0)) \ne \Pi(c(1))$.
\end{itemize}
The existence of the desired paths $c$ will be obtained with the help of a theorem due
to Malmquist in \cite{Malm} ({\it Th\'eor\`eme~1}, page 95) provided that $n \geq 3$ or $k \geq 1$.

To begin we can assume that $\lambda = \partial g (0,0,0) /\partial x =1$, up to a multiplicative constant,
so that the linear part of $Y$ at the origin has eigenvalues~$0$, $1$, and~$-1$. Consider then the linear
change of coordinates $(\bu,\bv,\bw) \mapsto (\bu + \bv, \bu - \bv, \bw)$. The pull-back $\overline{Y}$ of $Y$
in the coordinates $(\bu,\bv,\bw)$ becomes
\begin{equation}
\overline{Y} = \left[ (\bu + \bw A(\bu,\bv,\bw)) \frac{\partial}{\partial \bu} + \left( -\bv + \bw B(\bu, \bv, \bw) +
C(\bu + \bv) \right)
\frac{\partial}{\partial \bv} + \frac{1}{2}\bw^{2n-2} \frac{\partial}{\partial \bw} \right] \, \label{formulaforYbar-usingnow}
\end{equation}
for suitable holomorphic functions $A$ and $B$ of order at least~$1$ and a holomorphic function $C$ of order at least~$2$.
Similarly the vector field $\overline{\Pi^{\ast} X}$ corresponding to the pull-back
of $\Pi^{\ast} X$ in the coordinates $(\bu,\bv,\bw)$ satisfies
$\overline{\Pi^{\ast} X} = \bw^{2k+1} h( \bu + \bv , \bu -\bv, \bw ) \overline{Y}$.

Note that the singularity of the foliation associated with $\overline{Y}$
at the origin is a codimension~$1$ saddle-node, i.e. it has exactly one eigenvalue equal to zero.
In fact, it is a {\it resonant}\, codimension~$1$ saddle-node in the sense that the non-zero eigenvalues, $1$
and $-1$, are resonant. This type of singularity is closely related to a classical result due to Malmquist involving
systems of differential equations with an irregular singular point, cf. \cite{Malm}. We will state a slightly simplified
version of Malmquist results which is adapted to our problem. For $\delta \in \{0,1\}$, assume that
we are given a system of differential equations having the form
\begin{equation}\label{diff_eq_Malm}
\begin{cases}
\bw^{l+1} \frac{d\bu}{d\bw} = s_1 \bu + \beta_1 (\bu, \bv, \bw)\\
\bw^{l+1} \frac{d\bv}{d\bw} = s_2 \bv + \delta \bu + \beta_2 (\bu, \bv, \bw)
\end{cases}
\end{equation}
where $s_1. s_2 \ne 0$ and where $\beta_1, \, \beta_2$ are convergent series
(in particular conditions~(A) and~(B)
of \cite{Malm} are necessarily verified). Now let $\overline{\Psi} (\bw) = (\overline{\psi}_1 (\bw) ,
\overline{\psi}_2(\bw))$ be a formal solution for the system in question.
Malmquist then shows that for every $\varepsilon >0$, there exist circular sectors
of angle $2\pi/k - \varepsilon$ in the space of the $\bw$-variable with respect to which the system~(\ref{diff_eq_Malm})
admits a unique solution which is asymptotic to the formal solution $\overline{\Psi} (\bw)$.

The system~(\ref{diff_eq_Malm}) is naturally related to saddle-nodes singularities as those given by the
vector field $\overline{Y}$. In fact, the vector field $\overline{Y}$ is essentially equivalent to the system
of differential equations
\[
\begin{cases}
\bw^{2n-2} \frac{d\bu}{d\bw} = \bu + \bw A(\bu, \bv, \bw)\\
\bw^{2n-2} \frac{d\bv}{d\bw} = -\bv + \bw B(\bu, \bv, \bw) + C(\bu + \bv) \; .
\end{cases}
\]
Thus we have $s_1=1$, $s_2 = -1$ and $l = 2n-3$ and the formal solution $\overline{\Psi} (\bw)
=(\overline{\psi}_1 (\bw) , \overline{\psi}_2 (\bw))$ is obtained out of the (initial) formal separatrix $S$
(whose formal parameterization is simply $\bw \mapsto (\bw, \overline{\psi}_1 (\bw) , \overline{\psi}_2 (\bw))$).
Since these statements are clearly invariant by change of coordinates, we can return to the variables
$(u,v,w)$ where the vector field $\Pi^{\ast} X$ is defined. Keeping in mind that $w=\bw$, the angle of
the sector $V$ remains unchanged and the formal parameterization of $S$ will simply be denoted
by $\Psi (w) = (w, \psi_1 (w), \psi_2(w))$ (where $\psi_1 = \overline{\psi}_1 + \overline{\psi}_2$
and $\psi_2 = \overline{\psi}_1 - \overline{\psi}_2$).

Fix then an arbitrarily small neighborhood of the origin of the coordinates $(u,v,w)$.
Owing to Malmquist theorem, we can choose a solution of $Y$ asymptotic to the formal
series $\Psi (w) = (w, \psi_1 (w), \psi_2(w))$ of $S$ on the above mentioned sector $V$ (recall that
$w =\bw$). In particular there are points $w_0 \in \C$ with $\Vert w_0 \Vert$ arbitrarily small, and
there are leaves of $\Pi^{\ast} \fol$ to which paths
of the form $c(t) = (0,0, w_0 e^{2\pi i t/(2n-3)})$ can be lifted (with respect to the fibration given
by projection on the $w$-axis). Furthermore these lifted paths are contained
in arbitrarily small neighborhoods of the origin provided that $\Vert w_0 \Vert$ is small enough.
In other words, once a convenient circular sector $V$ of angle $2 \pi/(2n-3)$ is chosen, we can ``parameterize''
an open set of a certain leaf $L$ of $\Pi^{\ast} \fol$ by a map of the form
$w \mapsto (w, \psi_{1,V} (w) ,\psi_{2,V} (w))$, $w \in V$, where the holomorphic functions $\psi_{i,V}$ are
asymptotic on $V$ to the formal series $\psi_i (w)$, $i=1,2$.

The restriction to $L$ of $\Pi^{\ast} X = w^{2k+1} h( u,v,w ) Y$ can be considered in
the $w$-coordinate so as to become identified with a certain one-dimensional vector field
$Z(w) = \rho (w) \partial /\partial w$ defined on $V$. Since $h(0,0,0) \neq 0$ and the formal series
$\psi_i (w)$ have zero linear terms ($S$ is tangent to the $w$-axis),
there follows that $\rho$ has an asymptotic expansion of the form
$$
w^{2n+2k-1} + {\rm h.o.t.}\, \,
$$
up to a multiplicative constant,
where ${\rm h.o.t.}$ stands for terms of order higher than $2n+2k-1$. Since $k \geq 0$ and since
$V$ is a sector of angle $2n-3$, Lemma~\ref{asymptotic_vfield} implies the existence of an open embedded
path $c \subset V$ over which the integral of the time-form associated with $Z(w)$ equals zero. Hence
the vector field $\Pi^{\ast} X$ is {\it never semicomplete} (even if $n=2$ and $k=0$).

What precedes shows that $\Pi^{\ast} X$ is not semicomplete but we still need to show that
the initial vector field $X$ is not semicomplete. It is in this part of the argument that the assumption
$n \geq 3$ unless $k \geq 1$ will play a role. To conclude that $X$ is not semicomplete
we need to consider the possibility of having
$c(0)^2 = c(1)^2$ in the above mentioned path $c \subset V$. If this happens, it means that the difference of
argument between $c(0)$ and $c(1)$ is $\pi$. However, in the preceding discussion (cf. also Lemma~\ref{asymptotic_vfield}),
it was seen that the constructed path $c$ is such that the difference of argument between $c(0)$ and
$c(1)$ can be made arbitrarily close to $2\pi / (2n +2k -2) = \pi/(n+k-1)$. Recalling that $n \geq 2$,
there immediately follows that the desired path $c$ as above satisfying in addition $c(0)^2 \neq c(1)^2$
can be found provided that $n \geq 3$ or $k \geq 1$. Theorem~\ref{teo_semi-complete} is proved.
\end{proof}

Let us now prove Theorem~\ref{teo_semi-complete} in the remaining case.

\begin{proof}[Proof of Theorem~\ref{teo_semi-complete} when $h(0,0,0) = 0$]
Consider the foliation $\fol$ given in~(\ref{normal_form_sc}). We
know that every irreducible component of the set $\{ h=0\}$ is smooth, contains the separatrix
$S$, and {\it is not}\, invariant under $\fol$, cf. Corollary~\ref{Non-invariant-set_of_zeros}.
Whereas $S$ is a formal separatrix for the foliation
$\fol$ (and hence not contained in the singular set of $\fol$), the vector field $X$ vanishes identically
over $S$ since so does $h$. Hence, the argument employed in the previous case is no longer valid since $X$ does not
induce a time-form on $S$ (even if $S$ happens to be convergent). In particular, the existence of
an asymptotic leaf over which $X$ induces a time-form cannot be guaranteed.

To overcome this difficulty, we proceed as follows. To begin, we perform the above indicated blow-up $\Pi$ of weight~$2$
so that $\Pi^{\ast} \fol$ is given by~(\ref{Foliation_afterweigt2_blowup}). Denote
by $\widetilde{\{ h=0\}}$ the transform of $\{ h=0\}$ by $\Pi$ and let $S$ be identified with its own transform by $\Pi$.
The plane $\{ w=0\}$ is invariant under $\fol$ and, in addition,
the restriction of $\fol$ to this plane yields a foliation having a singularity with eigenvalues~$1$ and~$-1$ at
the origin. In particular, it follows that $\fol$ possesses exactly two separatrices, $S_1$ and $S_2$, contained in the plane
$\{ w=0\}$. Furthermore, both $S_1$ and $S_2$ are smooth and mutually transverse. Indeed, they are tangent to
the respective eigenvectors associated with~$1$ and with~$-1$.

\vspace{0.1cm}

\noindent {\it Claim}. The separatrix $S_1$ (resp. $S_2$) is not contained in strict transform of $\{ h=0\}$.

\vspace{0.1cm}

\noindent {\it Proof of the claim}. Without loss of generality we can assume that $h$ is irreducible (otherwise
we apply the argument to each irreducible component of $h$). Thus $\{ h=0\}$ is smooth and contains the formal
separatrix $S$ which is tangent to the $z$-axis. Thus $h$ is given by $h(x,y,z) = ax + by + {\rm h.o.t.}$, where
at least one between $a$ and $b$ is different from zero and where ${\rm h.o.t.}$ stands for higher order terms.
Next, recall that $\Pi (u,v,w) = (u,vw,w^2) = (x,y,z)$. Thus, if $a\neq 0$, then $\widetilde{h} (u,v,w)$ takes on
the form $\widetilde{h} (u,v,w) = au + {\rm h.o.t.}$ in the previous coordinates $(u,v,w)$. Hence, the
surface $\{ \widetilde{h} =0\}$ is tangent to the plane $\{ u=0\}$ at the origin. However, as previously seen,
the separatrices $S_1$ and $S_2$ are contained in $\{ w=0\}$ and tangent to
$\{ u=v\}$ and $\{ u=-v\}$, respectively. Therefore the claim holds provided that $a \neq 0$.

Assume now that $a=0$ so that $b \neq 0$. If $h(x,0,0)$ vanishes identically, then the {\it strict}\, transform
of $\{ h=0\}$ (i.e. ignoring the component associated with the exceptional divisor), is given by a function whose
linear part is $bv$. Thus, now, the corresponding surface is tangent to the plane $\{ v=0\}$ and again cannot contain
the separatrices $S_1$ or $S_2$. Finally, if $\tau (x) = h(x,0,0)$ does not vanish identically, then the intersection
of the mentioned surface with the plane $\{ w=0\}$ is given by $\tau (x) =0$. Once again it cannot contain the separatrices
$S_1$ and $S_2$. The claim is proved.\qed

The remainder of the proof consists of generalizing to the present setting a couple of well known properties
of $2$-dimensional saddle-nodes in the
spirit of \cite{Adolfo}. Consider the vector field $\Pi^{\ast} X = w^{2k+1} \widetilde{h} Y$, where
$Y$ is given by Formula~(\ref{Thevectorfield-Y-justtobequoted}). Owing to the above claim, the vector field $\widetilde{h} Y$
is regular at generic points of the separatrix $S_1$ (recall that $S_1$ is one of the two separatrices of the foliation
$\Pi^{\ast} \fol$ that are contained in the plane $\{ w=0\}$). If $T$ is a local
coordinate for $S_1$ around the origin, then the restriction to $S_1$ of the vector field $\widetilde{h} Y$ can naturally be
identified with a $1$-dimensional vector field of the form $g(T) \partial /\partial T$, where $g$ is a holomorphic function.
Furthermore, since $h(0,0,0) =0$, it follows that $g(0) = g'(0) =0$, i.e. the order at the origin of the restriction of
$\widetilde{h} Y$ to $S_1$ is at least~$2$.

Whereas the vector field $\Pi^{\ast} X = w^{2k+1} [\widetilde{h} Y]$ vanishes identically over $S_1 \subset \{ w=0\}$,
$\Pi^{\ast} X$ does induce an affine structure on $S_1$ (cf. \cite{Adolfo}). In the present case, this affine structure
has a singular point at the origin whose order coincides with the order of the vector field
$\widetilde{h} Y$. In other words, the affine structure in question is given around $0 \in S_1$ by the vector field
$g(T) \partial /\partial T$. This happens because the ``index'' of the separatrix $S_1$ has no effect into the computation of the
{\it ramification index}\, of this affine structure which is a phenomenon reminiscent from the fact that the Camacho-Sad index of
the strong invariant manifold of a saddle-node in dimension~$2$ is always zero (see \cite{camacho}, \cite{Adolfo}).

To effective use the mentioned affine structure, the holonomy map of $S_1$ will also be needed in the sequel. Let $\Sigma$
denote a local transverse section to $S_1$ equipped with coordinates $(\tilde{z},w)$. The corresponding holonomy map
$\sigma$ then fixes $(0,0)$ so that it can be viewed as a map from $(\C^2,0)$ to $(\C^2, 0)$. Now Lemma~\ref{localHolonomy-Generalization}
below asserts that $\sigma$ never coincides with the identity, though its derivative at $(0,0)$ is the identity map.

The preceding two paragraphs can be combined to prove Theorem~\ref{teo_semi-complete} as follows. Consider a loop
$c \subset S_1$ such that $c(0) =c(1) =S_1 \cap \Sigma$. Denote by $\tilde{c}_0$ the lift of $c$ in the leaf $L_0$ through
a point $(\tilde{z}_0, w_0) \in \Sigma$ sufficiently close to $(0,0) \simeq \Sigma \cap S_1$.
For $w_0 \neq 0$, the vector field $\Pi^{\ast} X$ is regular on $L_0$ so that
the corresponding time-form $dT_L$ can be considered. Now, the fact that the affine structure induced by $X$
on $S_1$ has order at least~$2$ at the origin, means that the integral of $dT_L$ over $\tilde{c}_0$ can be assumed to be equal
to zero without loss of generality, up to choosing $(\tilde{z}_0, w_0)$ sufficiently close to $(0,0)$.
In more accurate terms, whereas the mentioned integral may not be equal to zero,
it is always possible to ``slightly perturb'' the end-point of $\tilde{c}_0$ so as to obtain a new path over which
the integral of $dT_L$ is actually zero. Furthemore, since we can also assume that $\sigma (\tilde{z}_0, w_0)\neq (\tilde{z}_0, w_0)$,
the path $\tilde{c}_0$ is {\it open}: the perturbation of the end-point of the lift of $c$ cannot close this path up
since the points $\tilde{c}_0 (0)$ and $\tilde{c}_0 (1)$ are uniformly far apart in the intrinsic distance of $L_0$.

Summarizing the preceding, the integral of the time-form $dT_L$ induced by $\Pi^{\ast} X$ on $L_0$ vanishes
over the open path $\tilde{c}_0$. In particular, $\Pi^{\ast} X$ is not semicomplete. To conclude that the initial vector field
$X$ is not semicomplete as well, it is therefore necessary to check that the projection $\Pi (\tilde{c}_0)$ is still
open. Denote by $(\tilde{z}_1, w_1)$ the end-point in $\Sigma$ of $\tilde{c}_0$ (viewed as the lift of $c$, i.e. before
a possible perturbation of its end-point). Since $\sigma$ is tangent to the identity, it is clear that the argument
of $w_1$ converges to the argument of $w_0$ when $(\tilde{z}_0, w_0)$ converges to $(0,0)$. In particular $w_1 \neq -w_0$
so that $\Pi (\tilde{c}_0)$ is still open. Clearly the same conclusion still holds
after a possible ``perturbation'' of the end-point of $\tilde{c}_0$. This proves that $X$ itself is not semicomplete
and establishes Theorem~\ref{teo_semi-complete}.
\end{proof}

To round off the preceding discussion, we state and prove Lemma~\ref{localHolonomy-Generalization}.

\begin{lemma}
\label{localHolonomy-Generalization}
The local holonomy map $\sigma$ associated with the separatrix $S_1$ of $\Pi^{\ast} \fol$ does not
coincide with the identity. Furthermore it is derivative at the fixed point corresponding to $S_1$ is
the identity.
\end{lemma}

\begin{proof}
Recall that $\Pi^{\ast} (\fol)$ is given in suitable coordinates $(\bu,\bv,\bw)$ by the vector field
$\overline{Y}$ of Formula~(\ref{formulaforYbar-usingnow}). Recall also that $A$ and $B$ are
holomorphic functions of order at least~$1$ while $C$ is holomorphic of order at least~$2$.
In particular, the eigenvalues of $\Pi^{\ast} (\fol)$ at the origin are $1$, $-1$, and~$0$ which implies
that the derivative of $\sigma$ at the fixed point corresponding to $S_1$ is
the identity. The proof then amounts to checking that $\sigma$ cannot coincide with the identity.

Consider the restriction of $\Pi^{\ast} (\fol)$ to the invariant plane $\{ \bw =0\}$ and the corresponding restriction
of $\sigma$. Clearly we can assume that this restriction of $\sigma$ coincides with the identity, otherwise there
is nothing to be proved. In this latter case, however, there follows that the restriction of $\Pi^{\ast} (\fol)$
to $\{ \bw =0\}$ is linearizable after a unpublished result due to Mattei (for published generalizations,
see \cite{Yulij}, \cite{Helena2}). Thus, up to performing a change of coordinates $(u,v,z)$, we can assume that
$\Pi^{\ast} (\fol)$ is given by a vector field of the form
$$
(u+ zF(u,v,z)) \frac{\partial}{\partial u} + (-v + zG(u,v,z)) \frac{\partial}{\partial v} + z^n \frac{\partial}{\partial z}
$$
with $S_1$ given by $\{ u=z=0\}$. Set $v(t) = e^{2\pi i t}$ so that $dv/dt = 2\pi i e^{2\pi i t}$. Since
$$
\frac{du}{dt} = \frac{du}{dv} \frac{dv}{dt} \; \; \; {\rm and} \; \; \; \frac{dz}{dt} = \frac{dz}{dv} \frac{dv}{dt} \, ,
$$
we obtain the following system of equations:
\begin{equation}\label{computing-holonomy-System-11}
\begin{cases}
(-v zG(u,v,z))\frac{du}{dt} = (u+zF(u,v,z) 2\pi i v\\
(-v + zG(u,v,z)) \frac{dz}{dt} = z^n 2\pi i v \, .
\end{cases}
\end{equation}
Now let us apply the standard procedure to compute holonomy maps by means of power series.
First set $F(u,v,z) = \sum_{k,l \geq 0} F_{kl} u^k z^l$ and
$G(u,v,z) = \sum_{k,l \geq 0} G_{kl} u^k z^l$, where the coefficients $F_{kl}$ and $G_{kl}$ are functions of the
variable~$v$. Similarly, let $u = \sum_{i+j \geq 1} a_{ij} (t) u_0^iz_0^j$ and
$z = \sum_{i+j \geq 1} b_{ij}(t) u_0^iz_0^j$. With this notation, the holonomy map $\sigma$ is given by
$$
\sigma (u_0,z_0) = (u(1,u_0,z_0) ,z(1,u_0,z_0)) \, .
$$
Since for $t=0$, the resulting map coincides with the identity, there also follows that
$a_{10}(0) =b_{01}(0) =1$ whereas all the remaining coefficients $a_{ij}$ and $b_{ij}$ are equal to zero.

On the other hand, the following clearly holds:
\begin{equation}\label{computing-holonomy-System-22}
\frac{du}{dt} = \sum_{i+j \geq 1} a_{ij}'(t) u_0^i z_0^j \; \; \; {\rm and} \; \; \;
\frac{dz}{dt} = \sum_{i+j \geq 1} b_{ij}'(t) u_0^i z_0^j \, .
\end{equation}
Substituting the above formula for $dz/dt$ in~(\ref{computing-holonomy-System-11}) and recalling
that $v = e^{2\pi i t}$ yields:
\begin{eqnarray}
2\pi i e^{2\pi i t} \left( \sum_{i+j \geq 1} b_{ij} u_0^iz_0^j \right)^n & = &
\left[ -e^{2\pi i t} + \sum_{k+l \geq 0} G_{kl} (v) (\sum_{i+j \geq 1} a_{ij}u_0^iz_0^j)^k
(\sum_{i+j \geq 1} b_{ij}u_0^iz_0^j)^{l+1} \right] \times \nonumber \\
& & \, \times
\left(\sum_{i+j \geq 1} b_{ij}'(t) u_0^i z_0^j \right) \, . \label{computing-holonomy-System-33}
\end{eqnarray}
Comparing monomials in $u_0z_0$ in Equation~(\ref{computing-holonomy-System-33}), we first note that the
left side of this equation does not contain any monomial $u_0^i z_0^j$, with $i+j <n$, and non-zero coefficient.
From this, it follows that $b_{10}'(t) = b_{01}'(t) =0$ so that they are constant functions of~$t$. In view of the initial
conditions for $t=0$, we conclude that $b_{10} (t)=0$ while $b_{01} (t) =1$, for every $t \in \R$. The evident induction
argument then shows that $b_{i+j} (t)$ is constant equal to zero provided that $2 \leq i+j <n$. Finally, in the case
of $b_{0n}$ we obtain the equation
$$
- e^{2\pi i t} b_{0n}'(t) + \left[ {\rm terms\, \, involving \,\, b_{ij}' \, \, with\, \, i+j \leq n-1} \right]
= 2\pi i e^{2\pi i t} b_{01}^n \, .
$$
Since, for all $t \in \R$, we have that $b_{01} (t) =1$ and $b_{ij}'(t) =0$ provided that $i+j \leq n-1$,
we conclude that $b_{0n}'(t) =-2\pi i$ and hence $b_{0n} (t) =-2\pi i t$ since $b_{0n} (0) =0$. In particular,
$$
b_{0n} (1) = -2 \pi i
$$
so that $\sigma$ does not coincide with the identity. The proof of the lemma is completed.
\end{proof}

Let us close this section with the proof of Theorem~B along with its corollaries.

\begin{proof}[Proof of Theorem~B]
Let $X$ be a semicomplete vector field on $(\C^3 ,0)$ and denote by $\fol$ its associated foliation.
Assume that item~(2) in the statement of Theorem~B does not hold, i.e. that $\fol$ cannot be turned
into a foliation all of whose singular points are elementary by means of blow-ups centered at singular
sets. Thus owing to Theorem~\ref{Microlocalversion_TheoremA}, there is a sequence of one-point blow-ups
starting at the origin which leads
to a transform $\widetilde{\fol}$ of $\fol$ exhibiting a persistent nilpotent singular point $p$ along with a
formal separatrix $S$ of $\widetilde{\fol}$ at $p$. The corresponding transform of $X$ will be denoted
by $\widetilde{X}$.

Next assume aiming at a contradiction that the linear part of $X$ at the origin
is equal to zero. Then the transform
$\Pi_1^{\ast}(X)$ of $X$ under the first blow-up map vanishes identically over the exceptional divisor.
Since the subsequent blow-ups will always be performed at {\it singular points of the foliation}\,
which are contained in the zero-divisor of the vector fields in question, there follows that the zero-divisor
of $\widetilde{X}$ is not empty on a neighborhood of~$p$. Therefore the linear part of $X$ at~$p$
is equal to zero so that Theorem~\ref{teo_semi-complete} implies that $X$ is not semicomplete. The resulting
contradiction proves Theorem~B.
\end{proof}

\begin{obs}
{\rm The preceding proof makes it clear why the first condition in Theorem~\ref{danielJDG}, namely the fact
that each weighted blow up is strictly invariant with respect to the quasi-homogeneous filtration in question,
would be indispensable if Theorem~B were to be proved by means of Theorem~\ref{danielJDG}. If the singularities
of the associated foliation were ``reduced'' by an arbitrary choice of weighted blow ups, it would not be clear
whether or not the transform of
$X$ still is holomorphic with a non-empty
zero-divisor on a neighborhood of a persistent nilpotent singularity~$p$. This issue does not appear in the above
discussion since only standard blow-ups centered in the singular set of the corresponding foliations were used.}
\end{obs}

Corollary~C is an immediate consequence of Theorem~B while Corollary~D requires additional explanation.

\begin{proof}[Proof of Corollary~D]
Strictly speaking, this statement is actually more of a by-product of the proof of Theorem~\ref{teo_semi-complete} than
a corollary of Theorem~B. Consider a compact manifold $M$ and a holomorphic
vector field $X$ defined on $M$. Let $\fol$ denote the singular foliation associated with $X$ and
assume for a contradiction that $\fol$ possesses a singular point $p$ which cannot be resolved by a sequence of blow-ups
as in Theorem~B. Since finite sequences of (one-point) blow-ups as in Theorem~B change neither the compactness of $M$
nor the holomorphic nature of $X$, we can assume that $X$ admits the normal form~(\ref{normal_form_sc}).

Consider then the curve of singular points of $\fol$ locally
given by $\{ y=z=0\}$. Since $M$ is compact this curve of singular points $\mathcal{C}$ is global and compact on $M$.
Furthermore, up to resolving its singular points (as curve), we can assume $\mathcal{C}$
to be smooth. Thus $\mathcal{C}$ can globally be blown-up with weight~$2$ as in Lemma~\ref{lemma_ramified_blowup}.
Again this weighted blow-up is viewed as a two-to-one map as opposed to a birational one. Yet, the resulting
manifold $\widetilde{M}$ is still compact. Similarly the computations in Lemma~\ref{lemma_ramified_blowup} show
that the transform $\Pi^{\ast} X$ of $X$ is still holomorphic. Hence $\Pi^{\ast} X$ is complete
on $\widetilde{M}$. Thus, the restriction of $\Pi^{\ast} X$
to an open set $U \subset \widetilde{M}$ is a semicomplete vector field. A contradiction now
arises from noting that it was seen in the proof of Theorem~\ref{teo_semi-complete} that the vector field
$\Pi^{\ast} X$ is {\it never}\, semicomplete on a neighborhood of the codimension~$2$ saddle-node appearing
in connection with the transform of~$S$. This ends the proof of Corollary~D.
\end{proof}

\section{Examples and complements}\label{Examples_and_comments}

The first part of this section is devoted to detailing a couple of examples respectively related to Theorem~B
and to Theorem~\ref{Microlocalversion_TheoremA}. The remainder of the section will be devoted to a refinement of
Theorem~\ref{teo_semi-complete}
which, of course, can also be used to make Theorem~B slightly more accurate.

\subsection{A couple of examples}
We will first provide an example of complete vector field possessing a persistent nilpotent singular point so as to
substantiate the claim made in the Introduction. Then we will see some
explicit example of persistent nilpotent singularity that cannot be reduced to the examples of Sancho and Sanz.
The argument used here to show that the later example cannot be reduced to those of Sancho and Sanz
is elementary and differs from \cite{danielMcquillan}.

\vspace{0.2cm}

\noindent $\bullet$ {\bf Example}: The vector field
$$
Z = x^2 \partial /\partial x + xz \partial /\partial y + (y -xz) \partial /\partial z \, .
$$

Owing to the discussion in Sections~\ref{multiplicityalongseparatrix}
and~\ref{persistentnilpotent-Section}, it is clear that the foliation $\fol$ associated with $Z$
has a persistent nilpotent singular point at the origin which is associated with the {\it convergent}\,
separatrix $\{ y=z =0\}$. As a matter of fact, the separatrix giving rise
to a sequence of infinitely near (nilpotent) singular points is convergent in this case and, hence,
can be blown-up to resolve the singularity in question. Yet, the reader will note that our notion of
persistent nilpotent singular point only takes into consideration blow-ups centered in the singular
set of foliations so that the preceding observation is of relatively little importance for us.

We will prove that $Z$ can be extended as a complete vector field
on a suitable open manifold $M$. To begin with, note that the coordinate
$x(T)$ satisfies
$$
x(T) = \frac{x_0}{1-Tx_0}
$$
so that $x(T)$ is defined for every $T \neq 1/x_0$. In turn we have $d^2 y/dt^2 = z dx/dt + x dz/dt$
so that the vector field $Z$ yields
$$
\frac{d^2y}{dt^2} = xy = \frac{x_0}{1-Tx_0} y
$$
which has a regular singular point at $T =1/x_0$ and is non-singular otherwise. It then follows from the
classical theory of Frobenius (see for example \cite{Ince}) that $y (T)$ is holomorphic and defined for
all $T \in \C$. Now the vector field $Z$ also gives us
$$
\frac{dz}{dt} = y -xz = -\frac{x_0}{1-Tx_0} z + y(T) \, .
$$
Since $y(T)$ is holomorphic on all of $\C$, there follows that $z(T)$ is holomorphic on all of $\C$
as well.

Summarizing the preceding, the integral curve $\phi (T) = (x(T), y(T) ,z(T))$ of the vector field $Z$
satisfying $\phi (0) = (x_0, y_0, z_0)$ is defined for all $T \in \C \setminus \{ 1/x_0 \}$. Furthermore
as $T \rightarrow 1/x_0$, the coordinate $x(T)$ goes off to infinity while $y(T)$ and $z(T)$ are
holomorphic at $T = 1/x_0$. In particular, the vector field $Z$ is semicomplete on all of $\C^3$.

To show that $Z$ can be extended to a complete vector field on a suitable manifold $M$ is slightly more
involved. Denote by $\fol$ the foliation associated with $Z$ on $\C^3$. Note that the plane $\{ x=0\}$ is
invariant by $\fol$ and that $\fol$ is transverse to the fibers of the projection $\pi_1 (x,y,z) = x$
away from $\{ x=0\}$. The $x$-axis is also invariant by $\fol$ and $\fol$ can be seen as a linear system
over the variable $x$, namely we have $dy/dx = z/x$ and $dz/dx = y/x^2 - z/x$, cf. Chapter~III of \cite{bookYakovenko}.

Let $L$ be a leaf of $\fol$ which is not contained in $\{ x=0\}$. The restriction of
$\pi_1$ to $L$ is a local diffeomorphism from $L$ to the $x$-axis. In view of the
previous discussion, this local diffeomorphism can, in fact,
be used to lift paths contained in $\{ y=z=0 \} \setminus \{(0,0,0)\}$. Similarly,
owing to the description of $\fol$ as a linear system, the parallel transport along leaves of $\fol$
induces linear maps between the fibers of $\pi_1$ (isomorphic to $\C^2$). Finally the holonomy (monodromy)
arising from the invariant $x$-axis coincides with the identity (cf. Lemma~\ref{lemma_holonomy}). Thus we have proved
the following:

\begin{lemma}
\label{vectorfield_Z_linearsystems}
Away from $\{ x=0\}$, the
leaves of $\fol$ are graphs over the punctured $x$-axis. In particular,
the space of these leaves is naturally identified to $\C^2$ with coordinates $(y,z)$.\qed
\end{lemma}

The restriction of $Z$ to the invariant plane $\{x =0\}$ being clearly complete, to obtain an extension
of $Z$ as a complete vector field on a suitable open manifold $M$ we proceed as follows. Fix a leaf $L$
of $\fol$ with $L \subset \C^3 \setminus \{ x=0\}$ and denote by $Z_L$
the restriction of $Z$ to $L$. Consider the parameterization
of $L$ having the form $x \mapsto (x, A(x) ,B(x))$ where $x \in \C^{\ast}$ and
where $A$ and $B$ are holomorphic functions.
In the coordinate $x$, the one-dimensional vector field $Z_L$ becomes $x^2 \partial /\partial x$ and thus can
be turned in a complete vector field by adding the ``point at infinity'' to $L$ (i.e. $\{ u=0 \}$ in the
coordinate $u=1/x$). Therefore, to obtain the manifold $M$,
we simply add the ``point at infinity'' to every leaf $L$ of $\fol$ ($L \not\subset \{ x=0\}$). The description of
the leaves of $\fol$ as a linear system and the holomorphic behavior of the functions $y(T), z(T)$ as
$T \rightarrow 1/x_0$ makes it clear the resulting space can be equipped with the structure of a complex manifold $M$.
Moreover $Z$ is naturally complete on $M$ as desired.

\vspace{0.2cm}

\noindent $\bullet$ {\bf Example}: The (germ of) foliation $\fol_{\lambda}$ given by
$$
X_{\lambda} = (y - \lambda z) \frac{\partial}{\partial x} + zx \frac{\partial}{\partial y} +
z^3 \frac{\partial}{\partial z} \, ,
$$
with $\lambda \in \C$.

As mentioned the first examples of persistent nilpotent singularities were supplied
by Sancho and Sanz. It seems, however, interesting to provide an additional
explicit example along with a self-contained proof that it cannot be reduced to the examples of Sancho and Sanz.
For this, we begin by observing that neither the vector field $X_{\lambda}$ nor the
vector fields considered by Sancho and Sanz (see Section~\ref{persistentnilpotent-Section}) are in the normal form
indicated in Theorem~\ref{normalformpersistentnilpotent}, provided that $\lambda \neq 0$. Nonetheless, we have:

\begin{lemma}
\label{existence_formal_separatrix}
The foliation associated to $X_{\lambda}$ possesses a formal separatrix through the origin.
The formal separatrix is, in fact, strictly formal if $\lambda \neq 0$.
\end{lemma}

\begin{proof}
The leaves of the foliation associated with $X_{\lambda}$ can be viewed as the solutions of the following system of
differential equations:
\begin{equation}\label{system_diff_equations}
\begin{cases}
\frac{dx}{dz} = \frac{y - \lambda z}{z^3} \\
\frac{dy}{dz} = \frac{x}{z^2} \, .
\end{cases}
\end{equation}
We look for a formal solution $\varphi(z) = (x(z), y(z))$ of the system~(\ref{system_diff_equations}) in the
form $x(z) = \sum_{k \geq 0} a_k z^k$ and $y(z) = \sum_{k \geq 0} b_k z^k$.
By substituting these expressions in the first equation of~(\ref{system_diff_equations}) and comparing both sides,
we obtain
$$
b_0 = 0 \; , \; \; \; \; \; \; b_1 = \lambda \; , \; \; \; \; \; \; b_2 = 0 \; , \; \; \; \; \; \; {\rm and} \;
\; \; \; \; \; b_{k+3} = (k+1)a_{k+1} \; \; \; \; {\rm for} \; \;\ \; \; k \geq 0 \; .
$$
In turn, substitution and comparison in the second equation of~(\ref{system_diff_equations}) yields
$$
a_0 = a_1 = 0 \; \; \; \; \; \; {\rm and} \; \; \; \; \; \; a_{k+1} = kb_k \; \; \; \; {\rm for} \; \;\ \; \; k \geq 1 \; .
$$
Therefore $b_0 = b_2 =0$, $b_1 = \lambda$, and
$$
b_{k+3} = k(k+1) b_k
$$
for $k \geq 0$. It then follows that the coefficients of $y(z)$ having the form $b_{3l}$ and $b_{3l+2}$ are
zero for all every $l \geq 0$. Furthermore, for $l \geq 1$ we have
$$
b_{3l+1} = \lambda \Pi_{j=1}^l \frac{(3j-1)!}{(3j-3)!} \, .
$$
In the particular case where $\lambda = 0$, the series in question vanishes identically. This means that the curve,
given in coordinates $(x,y,z)$ by $\{x=0, \, y=0\}$ is a convergent separatrix for the foliation $\fol_0$.
Thus we assume from now on that $\lambda \ne 0$.

We want to check that the series $y = y(z) = \sum_{k \geq 0} b_k z^k$ diverges so as to
ensure that $z \mapsto (x(z),y(z),z)$ constitutes a {\it strictly formal} separatrix for $\fol_{\lambda}$. To do this,
just note that the series of $y(z)$ can be reformulated as
$z \sum_{k \geq 0} c_k z^{3k}$,
where $c_0 = 0$ and $c_k = \lambda \Pi_{j=1}^k \frac{(3j-1)!}{(3j-3)!}$. Up to considering the new variable
$w=z^3$, the radius of convergence of this later series is given by
$$
\lim_{k \to \infty} \frac{c_k}{c_{k+1}} = \lim_{k \to \infty} \frac{1}{(3k-1)(3k-3)} = 0
$$
and the lemma follows.
\end{proof}

Summarizing the preceding, the $z$-axis is invariant by $\fol_{\lambda}$ if $\lambda =0$. When
$\lambda \neq 0$, $\fol_{\lambda}$ admits a strictly formal separatrix
$S_{\lambda}$ parameterized by a triplet of formal series
$$
z \to \varphi(z) = (\sum_{k \geq 1} a_k z^k, \sum_{k \geq 1} b_k z^k, z) \, .
$$
Clearly $\varphi'(z) \ne (0,0,0)$ so that $S_{\lambda}$ is formally smooth. In the case $\lambda =0$, the foliation $\fol_0$
satisfies the conditions in Theorem~\ref{normalformpersistentnilpotent} and thus has a persistent
nilpotent singularity at the origin. If $\lambda \neq 0$, we note that $S_{\lambda}$ {\it is not}\, tangent
to the $z$-axis. However, arguing as in Lemma~\ref{lemma_normal_form_higher_tang_order},
there exists a polynomial change of coordinates $H$, of form
$H(\tilde{x}, \tilde{y}, z) = (h_1 (\tilde{x},z), h_2 (\tilde{y}, z), z)$, and such that
the formal separatrix $S_{\lambda}$ becomes tangent (with arbitrarily large tangency order) to
the $z$-axis. This gives the foliation $\fol_{\lambda}$ the normal form indicated in
Theorem~\ref{normalformpersistentnilpotent} and ensures that $\fol_{\lambda}$ gives rise to
a persistent nilpotent singularity with a strictly formal separatrix.

Regardless of whether or not $\lambda=0$, the multiplicity ${\rm mult} \, (\fol_{\lambda},S_{\lambda})$
of $\fol_{\lambda}$ along $S_{\lambda}$ is equal to~$3$. This contrasts with the examples of Sancho and
Sanz where the corresponding multiplicity is always~$2$. Since the multiplicity along a
formal separatrix is clearly invariant
by (formal) change of coordinates, there follows that the singularities $\fol_{\lambda}$ are not
conjugate to the singularities of Sancho and Sanz. Furthermore, as shown in
Sections~\ref{multiplicityalongseparatrix} and~\ref{persistentnilpotent-Section}, the value of
${\rm mult} \, (\fol_{\lambda},S_{\lambda})$ is invariant by blow-ups centered in the singular set of
the corresponding foliations. Therefore the singularities of $\fol_{\lambda}$ cannot give rise to
a singularity in the family
of Sancho and Sanz by means of any finite sequence of blow-ups as above.

\subsection{Local holonomy and semicomplete persistent nilpotent singularities}

To close this section, we turn our attention to semicomplete vector fields once again. Assume
that $X$ is a vector field whose associated foliation $\fol$ possesses a persistent nilpotent
singularity at $(0,0,0) \in \C^3$. Assume also that $X$ is semicomplete. Owing to
Theorem~\ref{teo_semi-complete}, the vector field $X$ has the normal form in
Theorem~\ref{normalformpersistentnilpotent} with $n=2$. In fact, denoting by $S$ the formal separatrix
of $\fol$ giving rise to a sequence of infinitely near nilpotent singularities, we have
${\rm mult} \, (X,S) = {\rm mult} \, (\fol,S) = 2$. These vector fields are thus very close to the
examples of Sancho and Sanz.

This raises the problem of classifying semicomplete vector fields in the Sancho and Sanz family.
In what follows we will conduct this classification only in the special case $\lambda =0$, i.e. when the formal
separatrix $S$ is actually convergent. Our purpose in doing so is to point out the role played by
the holonomy of this separatrix which shares some ideas with the proof of Theorem~\ref{teo_semi-complete} in the
case $h(0,0,0) =0$.
Furthermore, by dealing only with the case of convergent separatrices,
we avoid some technical difficulties that would require a longer discussion: whereas certainly interesting,
this discussion is not really indispensable from the point of view of this paper.
Finally, we also note that the material developed below includes
Lemma~\ref{lemma_holonomy} already used in the study of the vector field
$Z = x^2 \partial /\partial x + xz \partial /\partial y + (y -xz) \partial /\partial z$.

We begin by recalling the context of the proof of Theorem~\ref{teo_semi-complete}.
After performing the weight~$2$ blow-up, we have found an open path $c$ contained in a leaf of the
blown-up foliation $\Pi^{\ast} \fol$ over which the integral of the corresponding time-form
is equal to zero. Whereas this implies that $\Pi^{\ast} X$ is not semicomplete,
we were not able to conclude that $\Pi (c(0)) \neq \Pi (c(1))$ when $n=2$ and $k=0$. Thus, if $n=2$ and $k=0$ then
the possibility of having $X$ semicomplete cannot be ruled out. Let us then consider
this problem for the Sancho and Sanz family with $\lambda =0$, i.e. for the family of vector fields
having the form
$$
X = x^2 \frac{\partial}{\partial x} + (xz - \alpha xy) \frac{\partial}{\partial y} + (y - \beta xz)
\frac{\partial}{\partial z} \, .
$$
We note once and for all that the case $\alpha =0$ and $\beta =1$ correspond to the previously discussed
vector field~$Z$.

Let $V$ be a neighborhood of the origin where $X$ is assumed to be semicomplete. Denote by $S$ the
separatrix of $\fol$ given by the invariant axis $\{y = z = 0\}$. Fix a local transverse section $\Sigma_r$
through a base point $(r,0,0) \in V$. Denote by $L_p$ the leaf of $\fol$ passing through the point
$(r, p)$ with $p \in \Sigma_r$ (with the evident identifications).
If $p$ is close enough to~$(0,0)$, then the closed path $c(t) = (re^{2\pi i t},0,0)$ can be lifted,
with respect to the projection on the $x$-axis, into a path $c_p$ contained in $L_p$.
Furthermore we have
$$
\int_{c_p} dT_L = \int_{c} \frac{dx}{x^2} = 0 \, ,
$$
where $dT_L$ stands for the time-form induced on $L_p$ by $X$.
Thus, the vector field $X$ cannot be semicomplete {\it unless the holonomy map associated with $\fol$ and $S$
coincides with the identity}. Next, we have:

\begin{lemma}\label{lemma_holonomy}
Assume that $X$ and $\fol$ are as above. Then
the holonomy map associated with $\fol$ and $S$ coincides with the identity if and only if
$\alpha, \, \beta \in \Z$ with $\alpha \ne \beta$.
\end{lemma}

\begin{proof}
With the preceding notation, let $c_p(t) = (x(t), y(t), z(t))$ so that $x(t) = r e^{2\pi i t}$. The functions
$y(t)$ and $z(t)$ satisfy the following differential equations:
\begin{align*}
\frac{dy}{dt} &= \frac{dy}{dx} \, \frac{dx}{dt} = \frac{xz - \alpha xy}{x^2} \,  2\pi i x = 2\pi i (z - \alpha y) \\
\frac{dz}{dt} &= \frac{dz}{dx} \, \frac{dx}{dt} = \frac{y - \beta xz}{x^2} \,  2\pi i x = 2\pi i (e^{-2\pi i t}y - \beta z) \, .
\end{align*}
In terms of matrix representations, this system becomes
\[
\left[\begin{array}{c}
          \dot{y}\\
          \dot{z}
          \end{array}
\right] = \left[\begin{array}{cc}
          -2\pi i \alpha & 2\pi i \\
          2\pi i e^{-2\pi i t} & -2\pi i \beta
          \end{array}
\right] \,
\left[\begin{array}{c}
          y\\
          z
          \end{array}
\right] \, .
\]
The solution of this (non-autonomous) system can easily be obtained in terms of the coefficient matrix (denoted by $A(t)$
in the sequel). In particular,
$$
\left[\begin{array}{c}
          y(1)\\
          z(1)
          \end{array}
\right] = e^{\int_0^1 A(s) \, ds}
\left[\begin{array}{c}
          y(0)\\
          z(0)
          \end{array}
\right] \, ,
$$
where
$$
\int_0^1 A(s) \, ds = \left[\begin{array}{cc}
          -2\pi i \alpha  & 2\pi i \\
          0 & -2\pi i \beta
          \end{array}
\right] \, .
$$
Hence the matrix $B = \int_0^1 A(s) \, ds$ has two distinct eigenvalues if and only if $\alpha \ne \beta$. When
$\alpha = \beta$, the matrix $e^B$ has the form
$$
\left[\begin{array}{cc}
          e^{-2\pi i \alpha}  & 2\pi i e^{-2\pi i \alpha} \\
          0 & e^{-2\pi i \alpha}
          \end{array}
\right]
$$
so that the holonomy map is given by $(y,z) \mapsto e^{-2\pi i \alpha}(y + 2\pi i z, z)$ and hence
never coincides with the identity.

Suppose now that $\alpha \ne \beta$. We then have $B = PDP^{-1}$ where
\[
D = \left[\begin{array}{cc}
          -2\pi i \alpha  & 0 \\
          0 & -2\pi i \beta
          \end{array}
\right] \quad \text{and} \quad P = \left[\begin{array}{cc}
          1  & 1 \\
          0 & \alpha - \beta
          \end{array}
\right] \, .
\]
Therefore
\[
e^B = \left[\begin{array}{cc}
          e^{-2\pi i \alpha}  & \frac{1}{\alpha - \beta}\left( e^{-2\pi i \beta} - e^{-2\pi i \alpha} \right) \\
          0 & e^{-2\pi i \beta}
          \end{array}  \right] \, .
\]
This matrix (and thus the holonomy) coincides with the identity if and only if $\alpha, \, \beta \in \Z$. The lemma
follows.
\end{proof}

Lemma~\ref{lemma_holonomy} ensures that $X$ is not semi-complete if $\alpha = \beta$ or if one of the two
parameters $\alpha$ or $\beta$ is not an integer. The converse is provided
by Lemma~\ref{lastlemma_endofpaper} below.

\begin{lemma}
\label{lastlemma_endofpaper}
The vector field $X = x^2 \partial / \partial x + (xz - \alpha xy) \partial / \partial y
+ (y - \beta xz) \partial / \partial z$ is semicomplete for every pair $\alpha$, $\beta$ in $\Z$
with $\alpha \ne \beta$.
\end{lemma}

\begin{proof}
The argument is very much similar to the one employed for the vector field $Z$ ($\alpha =0$ and $\beta =1$).
Consider an integral curve $(x(T), y(T), z(T))$ of $X$. Clearly $x(T) = x_0 /(1-x_0T)$ which is a uniform
function on $\C \setminus \{ 1/x_0 \}$ (here we use the word uniform as opposed to multi-valued).
Thus we need to check that $y=y(T)$ and $z=z(T)$ are also uniform
functions of $T$.
This being clear for the integral curves contained in the invariant set $\{x = 0\}$,
consider the remaining orbits of $X$. These remaining orbits, or rather the leaves of the associated
foliation, can locally be parameterized by $x$, i.e. by a map of the form $x \mapsto (x, y(x) ,z(x))$. Since
$x$ is a uniform function of $T$, becomes reduced to showing that $y(x)$ and $z(x)$ are uniform functions
of~$x$. To do this, note that $dy/dx$ and $dz /dx$ are solutions of the linear system
\[
\begin{cases}
\frac{dy}{dx} =  \frac{z}{x} - \alpha \frac{y}{x}  \\
\frac{dz}{dx} =  \frac{y}{x^2} - \beta \frac{z}{x} \, .
\end{cases}
\]
This system has no singularities for $x \neq 0$. Furthermore the parallel transport along leaves gives
rise to linear maps. In particular the holonomy map arising from moving around the point $\{ x=0 \}$
is linear itself. This last map however is the identity thanks to Lemma~\ref{lemma_holonomy}.
The functions $y(x)$ and $z(x)$ are thus uniform functions of $x \in \C^{\ast}$ (for fuller details see
Chapter~III of \cite{bookYakovenko}). The lemma is proved.
\end{proof}

\section{Valuations and the proofs of Proposition~\ref{BasedonCano_Roche_Spivakovsky}
and of Theorem~A}\label{Provingtheproposition}

Most of this section is taken by the proof of Proposition~\ref{BasedonCano_Roche_Spivakovsky}.
As mentioned, this proof relies heavily on the work of Cano-Roche-Spivakovsky \cite{canorochetc}. In turn,
the proof of Theorem~A will follow easily from our previous results combined with Piltant's theorem
\cite{Piltant}.

Let $\fol$ denote a holomorphic foliation defined around $(0,0,0) \in \C^3$.
The singular set of $\fol$ will be denoted by ${\rm Sing}\, (\fol)$. Throughout this section, we only consider
sequences of (standard) blow-ups
\begin{equation}
\fol = \fol_0 \stackrel{\Pi_1}\longleftarrow \fol_1 \stackrel{\Pi_2}\longleftarrow \cdots
\stackrel{\Pi_k}\longleftarrow \fol_k  \, ,\label{aseqofblowups}
\end{equation}
satisfying the following condition: the center of each blow-up map $\Pi_i$ is either a single point in
${\rm Sing}\, (\fol)$ or a smooth analytic curve contained in ${\rm Sing}\, (\fol)$.

From now on, we also assume that $\fol$ is as in Proposition~\ref{BasedonCano_Roche_Spivakovsky}. In other words, no sequence
of blow-ups as in~(\ref{aseqofblowups}) leads to a foliation all of whose singular points are elementary.

We begin by making accurate a standard piece of terminology so as to avoid misunderstandings. Let $\fol$ be as above
and denote by $\nu$ a valuation over $\C$.
In many respects, the authors in \cite{canorochetc} follow the
Zariski approach to the resolution of singularities. A basic idea in Zariski's point of view
consists of trying to simplify the singularities of $\fol$ {\it only at the center of $\nu$},
as opposed to make all of these singularities simpler. To understand the meaning of the previous assertion, consider a blow
up $\pi$ of the ambient manifold where $\fol$ and $\nu$ are defined.
The valuation $\nu$ can be extended (pulled-back by $\pi$) to a valuation on the blown-up
manifold. This new valuation - still denoted by $\nu$ - has its center naturally
contained in the blown-up manifold where the blown-up foliation
$\fol_1$ is also defined. As a first step towards (global) simplification of the singular points of $\fol$, we may
only consider those singularities of $\fol_1$ {\it lying in the
center of the corresponding extension of $\nu$}. As usually happens in the literature, in the sequel
we will abuse notation and refer to the center of the extended
valuation as the center of $\nu$. In other words, whenever a sequence of blow-ups is considered, the
phrase {\it the center of $\nu$} has to be understood as the center of the extended valuation (which will still be
denoted by $\nu$) at each stage of the sequence of blow-ups in question.

In this sense, the so-called ``local'' uniformization (resolution) problem for foliations consists of finding a sequence of blow-ups
as in~(\ref{aseqofblowups})
leading to a foliation $\fol_k$ all of whose singular points {\it lying in the center of $\nu$}\, are elementary.

Concerning the paper \cite{canorochetc}, the first issue that needs to be pointed out is their strategy
to turn local results - in the above sense - into global ones. This strategy relies on Piltant's patching theorem
(see \cite{Piltant}) and its structure is summarized as follows.

Assume that for every given valuation $\nu$, the foliation $\fol$
can be transformed by a sequence of blow-ups as in~(\ref{aseqofblowups}) into a new foliation
$\fol'$ whose singularities {\it lying in the center of $\nu$}\, are all log-elementary (resp. elementary).
Then $\fol$ can also be turned into a foliation $\fol''$ {\it all of whose} singular points are log-elementary (resp. elementary).
This assertion is proved in Part~III of \cite{canorochetc}. The proof, in turn, amounts to checking
that Piltant's axioms \cite{Piltant} are satisfied in this setting. Whereas the authors of \cite{canorochetc}
focus on the case in which the singularities are log-elementary - so that they can deduce their Theorem~2 from
the local uniformization statement provided by their Theorem~1 - the argument is insensitive to whether we deal with
log-elementary or with elementary singular points.

Applying the preceding to a foliation $\fol$ as in the statement of Proposition~\ref{BasedonCano_Roche_Spivakovsky},
there follows the existence of a valuation $\nu$ for which no sequence of blow-ups as in~(\ref{aseqofblowups}) yields
a foliation having only elementary singularities {\it in the center of $\nu$}.
Some simple additional assumptions can be made without loss of generality. These are formulated
as Lemma~\ref{lemma1-1_forProposition} below.

\begin{lemma}
\label{lemma1-1_forProposition}
Assume that $\fol$ is a foliation as in Proposition~\ref{BasedonCano_Roche_Spivakovsky}.
Then there exists a valuation $\nu$ with residual field coinciding with $\C$ such that
for every finite sequence of blow-ups as in~(\ref{aseqofblowups}) the following holds:
\begin{enumerate}
  \item The center of (the corresponding extension of) $\nu$ is always a single point.

  \item The center of $\nu$ is never an elementary singular point for $\fol_k$.
\end{enumerate}
Furthermore the rank of $\nu$ is either~$1$ or~$2$.
\end{lemma}

\begin{proof}
Let $\nu$ be chosen so that it is not possible to turn $\fol$ into a foliation having elementary singularities
in the center of $\nu$ by means of a sequence of blow-ups as in~(\ref{aseqofblowups}). The existence of $\nu$ is
guaranteed by the previous discussions. Now, as
explained in \cite{canorochetc} (Section~9, Part II of \cite{canorochetc}), there is no loss of generality
in assuming that the residual field $k_{\nu}$ of $\nu$ coincides with $\C$. In fact, the remaining cases
are essentially cases in which the ambient manifold is of dimension $2$ which can be handled with minor modifications
of Seidenberg's theorem. In turn, since the residual field $k_{\nu}$ of $\nu$ coincides with $\C$,
there also follows that the center of $\nu$ consists of a single point and this still holds
for all extensions of $\nu$ obtained through blow-ups as above.

Finally, the fact that the rank of $\nu$ must be either~$1$ or~$2$ follows
directly from Proposition~4 in \cite{canorochetc}: the center of
a valuation having rank~$3$ can be turned into an elementary singularity by means of blow-ups as above.
The lemma is proved.
\end{proof}

From now on a valuation $\nu$ as in Lemma~\ref{lemma1-1_forProposition} is assumed to be fixed.

Since $k_{\nu} = \C$, Theorem~3 in
\cite{canorochetc} ensures the existence of a formal power series $\widehat{f}$ having {\it transverse maximal
contact}\, with $\nu$.
Recall that a formal power series $\widehat{f}$ is said to have transverse maximal contact with $\nu$ if
it is a Krull-limit of a sequence of (finite) power series $f_i$ at which $\nu$ takes strictly increasing values
(see below for a more geometric interpretation of this condition).
Naturally, standard desingularization of
surfaces implies that the formal surface $\widehat{W}$ is smooth at the center of
$\nu$. However, a more accurate result in proven in \cite{canorochetc}. Namely,
up to finitely many blow-ups (some ``preparation''), the formal series $\widehat{f}$ admits one of the following normal forms
(where $(x,y,z)$ is a regular system of parameters)
\begin{equation}
\widehat{f}  =  z + \sum_{i,j} c_{i,j} x^i y^j \; \; \; \; {\rm or} \; \; \; \;
\widehat{f} =  z + \sum_{i} c_{i} x^i \, . \label{normal_forms_Infinitesurface}
\end{equation}
The first case occurs when the rank of the valuation is~$2$ and, in this case, the variables
$x,y$ are such that $\nu (x)$ and $\nu (y)$ are $\Z$-linearly independent. If the rank of $\nu$
is~$1$, then we have the second case where $x$ is such that $\nu (x) \neq 0$ and $\widehat{f}$ does
not depend on the variable~$y$.

We can now make an important reduction in the statement of Proposition~\ref{BasedonCano_Roche_Spivakovsky}.

\begin{lemma}
\label{lemma2-2_forProposition}
Without loss of generality, we can assume that $\widehat{W}$ is (formally) invariant under $\fol = \fol_0$.
\end{lemma}

To prove Lemma~\ref{lemma2-2_forProposition}, let us begin by reminding the reader that the center
of $\nu$ (in the space where $\fol$ is defined) consists of a single point which can be assumed to coincide
with the origin of some local coordinates. We also choose a holomorphic vector field
\begin{equation}
X = F \partial /\partial x + G \partial /\partial y + H \partial /\partial z \, \label{vectorfield-X-1_Appendix}
\end{equation}
representing $\fol$ on a neighborhood of the center of $\nu$ (identified with the origin).

Fixed the power series $\widehat{f}$, the {\it basis of $\widehat{f}$}\, consists of those points $q$ at which
$\widehat{f}$ naturally defines a formal series. Clearly, the origin lies in the basis of $\widehat{f}$ but the basis of $\widehat{f}$
may or may not contain other points. For example, if $\widehat{f}  = z + \sum_{i} c_{i} x^i$,
then $\widehat{f}$ can be considered at every point belonging to the $y$-axis. This is, however, not necessarily true for
$\widehat{f} = z + \sum_{i,j} c_{i,j} x^i y^j$.

Let us now consider the {\it (formal) tangency locus}\, ${\rm Tang}\, (\fol, \widehat{W})$
between $\widehat{W}$ and $\fol = \fol_0$ {\it based at the origin}.
The tangency locus ${\rm Tang}\, (\fol, \widehat{W})$ (based at the origin) refers
to the formal equation
\begin{equation}
d \widehat{f} . X = \frac{\partial \widehat{f}}{\partial x} F + \frac{\partial \widehat{f}}{\partial y} G
+ \frac{\partial \widehat{f}}{\partial z} H = 0 \; . \label{Definition_tangencyLocus}
\end{equation}
A formal curve $t \mapsto (\gamma_1 (t) , \gamma_2 (t), \gamma_3 (t))$ based at the origin
is said to be contained in ${\rm Tang}\, (\fol, \widehat{W})$ if it satisfies the formal equation~(\ref{Definition_tangencyLocus}).
By definition, the origin also  belongs to ${\rm Tang}\, (\fol, \widehat{W})$ (recall that our definition of formal curve
requires at least one of the $\gamma_i$ not to vanish identically).

Similar considerations can be made at any point $q$ in the basis of $\widehat{f}$ provided that
the vector $X(q)$ belongs to the formal tangent space to $\widehat{W}$ at~$q$. This leads to the notion of tangency locus
between $\widehat{W}$ and $\fol = \fol_0$ {\it based at~$q$}. In the sequel, whenever the basis point is clear from the
context, we will simply say the tangency locus ${\rm Tang}\, (\fol, \widehat{W})$ between $\widehat{W}$ and $\fol$ without
further comments.

Assume now that $\widehat{W}$ {\it is not}\, formally invariant by $\fol$. Since $\fol$ has a singular point at the
origin, there follows the existence of a formal curve $S$ contained  in ${\rm Tang}\, (\fol, \widehat{W})$.
Naturally the formal curve $S$ must be viewed as given by a formal map
$t \mapsto (\gamma_1 (t) , \gamma_2 (t), \gamma_3 (t))$, with $\gamma_1(0) = \gamma_2 (0) = \gamma_3 (0) =0$ where
at least one of the power series $\gamma_1, \, \gamma_2$, and $\gamma_3$ does not vanish identically.
The proof of the existence of $S$ is a straightforward computation which, in fact, is the same as the standard result for the case
of an analytic surface: just note that monomials of degree, say~$d$, in the series of $\gamma$ depend only on a finite part of the formal
series of $\widehat{f}$.

Next recall the geometric interpretation of of the fact that $\widehat{W} = \{ \widehat{f} =0 \}$ has
transverse maximal contact with $\nu$. For this, consider the formal series $\widehat{f}$ and a blow-up map
$\Pi$ obtained as a composition of finitely many blow-up maps as in~(\ref{aseqofblowups}). The transform
of $\widehat{f}$ under $\Pi$ is nothing but the composition $\widehat{f} \circ \Pi$. Naturally this transform
does not make sense as formal
series at a generic point of the exceptional divisor associated with $\Pi$. However, this transform does make sense
at the center of $\nu$ provided that $\widehat{f}$ and $\nu$ have transverse maximal contact. Indeed, by definition,
$\widehat{f}$ is the Krull-limit of finite series $f_i$ over which $\nu$ takes strictly increasing values.
The transforms of the series $f_i$ are well defined (these series are finite) and they are still convergent for the
Krull-topology at the center of (the extension of) $\nu$. Their Krull-limit at the center of (the extension of)
$\nu$ then defines a formal series which naturally provides the transform of $\widehat{f}$ at the center of $\nu$.
This observation
is a fundamental issue that allows us to consider the transform under $\Pi$ of the formal surface $\widehat{W}$ as a formal surface
``passing through the center of $\nu$''. It also encodes the geometric interpretation of the condition of
having transverse maximal contact and is often abridged by saying that the ``formal surface
$\widehat{W}$ keeps passing through the center of $\nu$'' for every sequence of blow-ups as in~(\ref{aseqofblowups}).
This terminology will also be used in the remainder of our discussion.

Now fix a formal curve $S$ as above which is contained in ${\rm Tang}\, (\fol, \widehat{W})$. Up
to performing finitely many blow-ups, $S$ can be assumed to be formally smooth. Assume that $S$ {\it is not}\, contained
in the singular set of $\fol$. Then we
claim that for every sequence of blow-ups as in~(\ref{aseqofblowups}), the transform of the curve
$S$ passes through the center of the corresponding extension of $\nu$.
To check our claim, we proceed as follows. First note that the transform of
$\widehat{W}$ passes through the center of the extension of $\nu$ due to the transverse maximal contact assumption.
Since the transform of $\fol$ is singular at the point in question, the preceding discussion ensures the existence
of a (branch of) formal tangency curve $\widetilde{S}$ between
the transforms of $\fol$ and of $\widehat{W}$ stemming from the center of (extension of) $\nu$.
Now note that $S$ is contained in the center of a blow-up belonging to a reduction procedure as
in~(\ref{aseqofblowups}) because these centers are contained in the singular set of the corresponding foliations,
unlike the curve $S$ (by assumption). Since
a sequence of blow-ups, starting from $p=p_0$, as in~(\ref{aseqofblowups})
cannot produce new tangency points between the transforms
of $\fol$ and of $\widehat{W}$, we must conclude that $\widetilde{S}$ is the transform of
$S$. In particular, the curve $S$ satisfies the conditions in Proposition~\ref{BasedonCano_Roche_Spivakovsky}
unless {\it $S$ is fully constituted by singular points of $\fol$}.

\begin{obs}
\label{TangencyCurves-and-centerofvaluations}
{\rm It is convenient to point out a fact already implicit in the paragraph above. If
$\Pi$ is a blow-up map as in~(\ref{aseqofblowups}) which {\it is not}\, centered at~$S$, then the center of the corresponding
extension of $\nu$ is determined by the transform of the curve $S$. Namely, the transform of $S$ defines a single
point in the exceptional divisor associated with $\pi$ and this point is the center of the extended valuation.}
\end{obs}

\begin{proof}[Proof of Lemma~\ref{lemma2-2_forProposition}]
Owing to what precedes, we just need to consider the case in which $S$ is smooth and entirely constituted by
singular points of $\fol$. Therefore $S$ is actually an analytic curve contained in the
singular set of $\fol$ and hence can also be used as center for a blow-up map.

Up to performing finitely many (one-point) blow-ups, we can choose coordinates $(x,y,z)$ around
the center of $\nu$, identified with $(0,0,0)$, such that the following holds:
\begin{enumerate}
  \item The exceptional divisor is locally given by $\{ z=0\}$.
  \item The curve $S$ coincides with the $z$-axis.
  \item In view of the normal forms~(\ref{normal_forms_Infinitesurface}), we can also assume that
  $\widehat{W}$ is given either by $x = \sum_{i,j} c_{i,j} y^i z^j$ (with $c_{i,0} =0$ for every~$i$)
  or by $x = \sum_i c_i y^i$.
\end{enumerate}
Note that the curve $\gamma \subset \{ z=0 \}$ determined by $\{z=0\} \cap \widehat{W}$ may or may not be contained in
the tangency locus of $\widehat{W}$ and $\fol$.

\vspace{0.1cm}

\noindent {\it Claim 1}. We can assume that $\gamma$ is not contained in the tangency locus of $\widehat{W}$ and $\fol$.

\noindent {\it Proof of Claim 1}. Note that the above described situation is invariant under (one-point) blow-ups
at the center of $\nu$ (in turn determined by the intersection of the transform of $S$ with the exceptional divisor,
cf. Remark~\ref{TangencyCurves-and-centerofvaluations}).

Now assume that $\gamma$ is contained in the tangency locus of $\widehat{W}$ and $\fol$. In particular
$\gamma$ is invariant by $\fol$ (here it is included the possibility of having $\gamma$ contained in the singular set
of $\fol$). Next, let $X$ be a local vector field
representing $\fol$ on a neighborhood of $(0,0,0) \in \C^3$ (identified with the center of $\nu$) and consider its first non-zero
homogeneous component $X_n$ of $X$ at $(0,0,0)$. The tangent vector to $\gamma$ at $(0,0,0)$ is clearly invariant by $X_n$
since $\gamma$ is invariant under $\fol$. The same applies to the vector $(0,0,1)$ since $\fol$ is singular all along
the $z$-axis (identified with $S$). Therefore the plan spanned by $(0,0,1)$ and the tangent vector to $\gamma$ at $(0,0,0)$
is invariant by $X_n$. Naturally the plane in question is nothing but the tangent space to the surface $\widehat{W}$, since
$\widehat{W}$ is smooth.

Summarizing what precedes, whenever $\gamma$ is contained in the tangency locus of $\widehat{W}$ and $\fol$, there follows
that the tangent space to $\widehat{W}$ is invariant under the first non-zero
homogeneous component of a vector field representing $\fol$ around the center of $\nu$.
However, as previously pointed out, we can apply to this situation any sequence of
one-point blow-ups at the center of $\nu$. Since $\widehat{W}$ is not invariant under $\fol$,
we can find a suitable sequence such that the tangent space to the transform of $\widehat{W}$ is no longer
invariant by the first non-zero homogeneous component of
a local vector field representing the corresponding transform of $\fol$. Therefore the corresponding curve $\gamma$
will not be contained in the tangency locus of the corresponding transforms of $\fol$ and $\widehat{W}$. The claim
is therefore proved.\qed

We now go back to the initial local coordinates $(x,y,z)$.
Owing to Claim~1, we can assume that the tangency locus of $\widehat{W}$ and $\fol$ on a neighborhood of $(0,0,0)$
is reduced to the above defined curve~$S$ (locally coinciding with the axis~$z$). Recalling that $S$ is smooth and
contained in the singular set of $\fol$, we will perform blow-ups centered at~$S$. First, we fix a holomorphic
vector field $X$ as in~(\ref{vectorfield-X-1_Appendix}) which represents $\fol$ around the origin. Since
$S$ is contained in the singular set of $\fol$, there follows that $H$ vanishes identically over the $z$-axis.
Furthermore, the component in the direction $\partial /\partial z$ of the transform of $X$ under a blow-up
centered at $S \simeq \{ x=y=0\}$ is simply the transform of the function $H$ under the blow-up in question.
In other words, as long as this type of cylindrical blow-up is performed, the singular set ${\rm Sing}\, (\fol_1)$
of the resulting blown-up foliation $\fol_1$ is contained in the transform of the surface $\{ H =0\}$.

Since blow-ups centered at~$S$ will be performed, we need to extend the content of Remark~\ref{TangencyCurves-and-centerofvaluations}
to this type of blow-up. Indeed, if $\Pi$ denotes the blow-up centered at~$S$, then
the center of (the extension of) $\nu$ is determined by the fact that it lies in the intersection of
$\Pi^{-1} (0,0,0)$ with the transform
of $\widehat{W}$. More precisely the formal curve $\gamma$ obtained by intersecting $\widehat{W}$ and the plane $\{ z=0\}$
determines a point $p_1$ in $\Pi^{-1} (S) \cap \{ z=0\}$ (where, by abusing notation, $\{ z=0\}$ stands for both
the initial plane $\{z=0\}$ and its transform under $\Pi$).
This point $p_1$ is the (new) center of $\nu$.

At $p_1$, let $S_1$ denote the tangency locus between the blown-up foliation $\fol_1$ and the (transform of)
$\widehat{W}$. As previously seen, this tangency locus is well defined. Furthermore this tangency
locus must contain a formal curve $S_1$ since $p_1$ has to be a singular point of $\fol_1$. In the present situation,
$S_1$ is clearly smooth, contained in the exceptional divisor $\Pi_1^{-1} (S)$, and transverse to $\{ z=0\}$.
If $S_1$ is not contained in ${\rm Sing}\, (\fol_1)$, then $S_1$ is a formal separatrix satisfying the
conditions of Proposition~\ref{BasedonCano_Roche_Spivakovsky} and thus there is nothing else to be proved. Therefore
$S_1$ can be assumed to be contained ${\rm Sing}\, (\fol_1)$. As previously seen, this implies, in particular,
that $S_1$ is contained in the transform of the surface $\{ H =0\}$.

The procedure above can now be repeated with center $S_1$ (which is an analytic curve contained in ${\rm Sing}\, (\fol_1)$).
We have:

\vspace{0.1cm}

\noindent {\it Claim 2}. If $\{ H=0\}$ is not invariant, then we can assume that $S_1$ is not contained in
${\rm Sing}\, (\fol_1)$.

\noindent {\it Proof of Claim 2}. For every generic point $z_0 \in S$, the intersection of $\Pi_1^{-1} (z_0)$
and the transform of $\{ H=0\}$ consists of a bounded number of points. Indeed, the context is essentially equivalent to
the $2$-dimensional one: in particular Seidenberg theorem would lead to elementary singular points sitting over
{\it generic points}\, of the $z$-axis. Thus, by iterating sufficiently many times this type of blow-ups, the generic point
of the intersection between the transform of $\{ H=0\}$ and the (cylindrical) exceptional divisor will have to be regular
for the corresponding blown-up foliation $\fol_k$ (since $\{ H=0\}$ is not invariant by $\fol$).
At this point, the tangency locus of $\fol_k$ and the corresponding
transform of $\widehat{W}$ yields the desired separatrix proving Proposition~\ref{BasedonCano_Roche_Spivakovsky}.
This establishes Claim~2.\qed

Thanks to Claim~2, we can assume that the surface $\{ H=0\}$ is invariant by $\fol$. The rest of the proof consists of
proving that $\{ H=0\}$ must have transverse maximal contact with $\nu$, so that it will suffice to deal with invariant
surfaces having transverse maximal contact with $\nu$. This is, however, clear by now. In fact, for all blow-ups
centered in the curves $S_i$ (as above), $\{ H=0\}$ will have to pass through the center of the valuation. The other two
possible types of blow-ups are as follows:
\begin{itemize}
  \item One-point blow-ups of the center of $\nu$. Again, the new center of $\nu$ will be determined by the transform of
  $S$ (locally given by the $z$-axis). As seen in Sections~2 and~3, in the corresponding coordinates $(u,v,z) \mapsto (uz,vz,z)$,
  the component of $X$ in the direction $\partial /\partial z$ transforms like the function $H$. Thus we obtain
  the desired separatrix proving Proposition~\ref{BasedonCano_Roche_Spivakovsky} unless $\{ H=0\}$ continues to pass
  through the center of the valuation.

  \item Blow-ups centered at smooth curves (contained in $\{ z=0\}$) and passing through the origin (identified with the center
  of $\nu$). Once again, in suitable coordinates, the component of $X$ in the direction $\partial /\partial z$ transforms
  like the function $H$.
\end{itemize}
Summarizing what precedes, we obtain the desired separatrix unless the (analytic) invariant surface given by
$\{ H=0\}$ passes through the center of the valuation for every sequence of blow-up maps as in~(\ref{aseqofblowups}).
This, however, implies that $\{ H=0\}$ has transverse maximal contact with $\nu$ and completes the proof of the lemma.
\end{proof}

The remainder of this appendix is devoted to proving Proposition~\ref{BasedonCano_Roche_Spivakovsky} in the case where
the formal surface $\widehat{W}$ with transverse maximal contact with $\nu$ is, in addition, invariant under
$\fol$. This case is very close to the $2$-dimensional situation considered in Seidenberg's theorem, as it will be
detailed in what follows.

We consider local coordinates around the center of $\nu$ (identified with $(0,0,0) \in \C^3$)
along with the formal surface $\widehat{W} = \{ \widehat{f} =0 \}$.
Since the formal surface $\widehat{W}$ is smooth, there is a formal change of coordinates in which $\widehat{W}$
becomes identified with a coordinate plane. Naturally, in these formal coordinates, the foliation becomes only formal.
In other words, the vector field $X$ of~(\ref{vectorfield-X-1_Appendix}) representing $\fol$ around the origin
(i.e. the center of $\nu$) becomes only {\it formal}. This, however, is a minor issue since for resolution problems
there is essentially no difference between working with a formal vector field or with an actual holomorphic one.

We can then consider local coordinates
$(x,y,z)$ where the surface $\widehat{W}$ coincides with the plane $\{ z=0\}$ at the expenses of consider
$X$ as a formal vector field. In particular, in the coordinates $(x,y,z)$ the (formal) vector field $X$
takes on the form~(\ref{vectorfield-X-1_Appendix}), where $F$, $G$, and $H$ are formal series with $H$ being divisible by~$z$.

\begin{obs}
\label{dontthingisgonnabequoted}
{\rm The reader will note that the choice of formal coordinates $(x,y,z)$ where the surface $\widehat{W}$ becomes identified with the
plane $\{ z=0\}$ is basically a convenient way to abridge notation. Indeed, we can directly work with the initial coordinates
and with the formal generator $\widehat{f}$ of the surface $\widehat{W}$ but this would make the notation slightly
cumbersome.

The use of formal coordinates as above actually helps to make the argument more transparent since, in most of the discussion, there is
no difference between dealing with formal or holomorphic vector fields. Along this direction, we will occasionally allow ourselves
to argue as if $X$ is a holomorphic vector field: this will only be done, however, when the general procedure is straightforward
enough to avoid confusion.}
\end{obs}

Recalling that Seidenberg's theorem applies equally well to formal vector fields, we can consider the case of
the restriction of $X$ to the invariant plane $\{ z=0\}$. Applying Seidenberg's
theorem in the present context, however, requires us to distinguish between the restriction of a $3$-dimensional foliation
to an invariant plane and the foliation {\it on the invariant plane induced by the restriction of the mentioned foliation}.
In other words, in dimension~$2$, singularities are always
isolated: if the coordinate functions of a vector field have a common factor, then this factor can be eliminated as far as the
underlying foliation is concerned. This is no longer true if we are looking at vector fields defined on a $3$-dimensional
ambient. In more accurate terms, if $X$ is as in~(\ref{vectorfield-X-1_Appendix}), the functions $F(x,y,0)$ and $G(x,y,0)$ may
have a non-trivial common factor that does not divide, for example, $H$. This gives rise to a curve of singular points of
$X$ contained in the plane $\{ z=0\}$ which, indeed, constitutes a curve of singularities for the foliation
in dimension~$3$. However, if we look at the foliation induced by restriction of the previous one to the plane
$\{ z=0\}$, then the resulting $2$-dimensional foliation can be extended as a regular foliation to all but finitely
many points in the curve in question. Lemma~\ref{lemma3-3_forProposition} below makes these comments accurate.

\begin{lemma}
\label{lemma3-3_forProposition}
Without loss of generality we can assume that the (formal)
vector field $X$ representing the foliation on a neighborhood of $(0,0,0)$ (identified with the center of $\nu$) has
the following form:
\begin{equation}
X = x^ny^m (f(x,y) \partial /\partial x + g(x,y) \partial /\partial y) + z (r (x,y,z) \partial /\partial x +
s (x,y,z) \partial /\partial y) + z h \partial /\partial z \, , \label{vectorfield-X-2_Appendix}
\end{equation}
where $h$ is a formal series in $x$, $y$, and $z$. Furthermore, the following holds:
\begin{itemize}
  \item The vector field $Y = f(x,y) \partial /\partial x + g(x,y) \partial /\partial y$ - viewed as a $2$-dimensional
  vector field on the plane $\{ z=0\}$ - either is regular or has an
  (isolated) elementary singular point at $(0,0)$.

  \item Both $m$ and $n$ are nonnegative integers. In addition, if $m >0$ (resp. $n >0$), then the axis $\{ x=0\}$
  (resp. $\{ y=0 \}$) is invariant by $Y$.
\end{itemize}
\mbox{}\qed
\end{lemma}

Considering the normal form~(\ref{vectorfield-X-2_Appendix}), it is clear that at least one between $m$ and $n$ must
be strictly positive otherwise the origin is an elementary singular point of $X$. Similarly, $h (0,0,0)$ must be
equal to zero otherwise $X$ has a non-zero eigenvalue in the direction $\partial /\partial z$.

Next there is no loss of generality in assuming that the functions $r$ and $s$ are divisible by $x^ny^m$.
Indeed, let $\Pi$ denote the blow-up centered at the $x$-axis and consider coordinates $(x,y,v)$ where $\Pi$
becomes $\Pi (x,y,v) = (x,y,yv)$. In this case, the transform of $\{ z=0\}$ coincides with the plane $\{ v=0\}$ while the
transform of $X$ becomes
\begin{eqnarray}
\Pi^{\ast} X & = &
x^ny^m Y + yv \left[ r (x,y,yv) \partial /\partial x +
s (x,y,yv) \partial /\partial y \right] + \label{Furthernormalizations_forProposition-1} \\
& & + \,  v \left[ -x^ny^{m-1} g(x,y) - v s (x,y,yv) + h(x,y,yv) \right] \partial /\partial v \, . \nonumber
\end{eqnarray}
Thus the ``new'' functions $r$ and $s$ have, in particular, acquired a factor of~$y$. Hence, by iterating blow-ups as above
centered either at the $x$-axis or at the $y$-axis the claim follows.

Formula~(\ref{Furthernormalizations_forProposition-1}) also yields:

\begin{lemma}
\label{lemma4-4_forProposition}
Without loss of generality, the function $h$ admits the decomposition $h = h^{(0)} (x,y) + z h^{(z)} (x,y,z)$
where $h^{(z)}$ is divisible by $x^ny^m$.
\end{lemma}

\begin{proof}
Again it follows from formula~(\ref{Furthernormalizations_forProposition-1})
that every factor of~$z$ in $h$ acquires a factor of~$y$. Similarly, every
new function ``$s$'' has an additional factor of~$y$ (as previously seen). The lemma then follows by repeating
the indicated procedure.
\end{proof}

\begin{obs}
\label{mightcometobequoted}
{\rm Clearly the preceding shows that $r$, $s$, and $h^{(z)}$ can be assumed to be divisible by $x^ay^b$, for every a priori
given $a,b \in \N$.

Whereas it will not really be necessary in what follows, we may also note
that the function $(x,y) \mapsto x^ny^{m-1} g(x,y)$ may be assumed to be
divisible by~$x^ny^m$. Indeed, if $(0,0)$ is a regular point for $Y$ then we can assume that $g$ vanishes identically.
Otherwise $(0,0)$ is an elementary (irreducible) singularity of~$Y$ and, since we are working with formal vector fields,
there is no loss of generality in assuming that $g$ is divisible by~$y$ in this case.}
\end{obs}

In the sequel we can also assume that $h$, and hence $h^{(0)}$, is not divisible by either~$x$ or~$y$, otherwise
we can reduce at least one between~$m$ and $n$. Now
let $P$ denote the (homogeneous) polynomial obtained as the first non-zero homogeneous component of the formal
series of $h^{(0)}$. The degree of $P$ will be denoted by $d \geq 1$. Therefore the polynomial $P$ has the form
$$
P = \sum_{i=0}^d c_i x^i y^{d-i} \, .
$$
Hence the set $\{ P =0 \}$ consists of $k$ straight lines $C_1, \ldots, C_k$ through $(0,0) \in \C^2$, with $k \leq d$.
In fact, if we add multiplicity to each one of the lines $C_j$, then we will have $k=d$.
The union $C_1 \cup \ldots \cup C_k$ of the mentioned lines naturally forms the tangent cone to the set $\{ h^{(0)} =0 \}$
viewed as contained in the plane $\{ z=0\}$ (otherwise the cone in question is simply the ``cylinder'' over
$C_1 \cup \ldots \cup C_k$). Clearly the set $\{ h^{(0)} =0 \} \cap \{ z=0\}$ is just a finite number of irreducible (possibly singular)
analytic curves.

\begin{proof}[Proof of Proposition~\ref{BasedonCano_Roche_Spivakovsky}]
We keep the preceding notation. Since $r$, $s$, and $h^{(z)}$ are divisible by $x^ny^m$ and since $d \geq 1$, there follows
that the $z$-axis is contained in the singular set of $\fol$. This axis can thus be used as center for a blow-up.

Here it is convenient to point out a simple issue concerning a blow-up $\Pi$ centered at the $z$-axis and the corresponding
center of $\nu$. In contrast with the previous cases, the center of (the extension of) $\nu$ is not immediately
detected in the present situation. Clearly, the new center of $\nu$ is contained in the rational curve
$\Pi^{-1} (0,0,0)$ but this curve is entirely contained in the transform of $\widehat{W} \simeq \{ z=0\}$ so that,
a priori, any point in $\Pi^{-1} (0,0,0)$ can be the center of $\nu$. In the sequel, by abuse notation, $\{ z=0\}$ will
denote both the initial coordinate plane and its transform under $\Pi$.

Recall that $X$ as in~(\ref{vectorfield-X-2_Appendix}) represents $\fol$ around $(0,0,0)$. In turn, the transform
of $x^ny^m[ (f + z r) \partial /\partial x  + (g + zs) \partial /\partial y]$ by $\Pi$ vanishes
over the exceptional divisor to the order~$m+n$. Similarly the transform under $\Pi$ of
$h^{(0)} (x,y) \partial /\partial z$ (resp. $z h^{(z)} (x,y,z) \partial /\partial z$) vanishes over the exceptional
divisor to the order~$d$ (resp. $m+n$ but actually arbitrarily larger if it were necessary). Hence,
the transform of $X$ under $\Pi$ vanishes over the exceptional divisor with order $\min \{ d, m+n \}$.
To make the subsequent discussion clearer, it is convenient to first consider two cases:

\noindent {\it Case 1}. Assume that $d > m+n$.

In this case, after eliminating the common power of the generator of the ideal associated with the exceptional divisor,
we see that the transform $\fol_1$ of the foliation $\fol$ is regular at $\Pi^{-1} (0,0,0) \cap \{ z=0\}$ except at the
two points of $\Pi^{-1} (0,0,0) \cap \{ z=0\}$ determined by the axes $x$ and $y$ (both singular for $\fol$). In
particular, the center of $\nu$ has to be one of these two points. Furthermore, since this center has not become an elementary
singular point, the procedure can be continued up to applying again the construction used in
Lemma~\ref{lemma4-4_forProposition} to make sure that
the corresponding (new) functions $r$, $s$, and $h^{(z)}$ are as indicated.

When continuing this procedure, note that the degree of the new polynomial ``$P$'' will be strictly smaller than~$d$
provided that $k \geq 2$. We will return to this point in the more general discussion below.

\noindent {\it Case 2}. Assume that $d \leq m+n$. This is the more interesting case. The foliation $\fol_1$ is represented
by a vector field $X_1$ whose component in the direction $\partial /\partial z$ has the form
$$
z[\widetilde{P} + z  \widetilde{h}^{(z)}] \partial /\partial z
$$
where $\widetilde{P}$ (resp. $ \widetilde{h}^{(z)}$) is the transform of $P$ (resp. $h^{(z)}$) after eliminating
the above mentioned common factor arising from the exceptional divisor. In the case of $\widetilde{P}$, there follows
in particular that $\widetilde{P}$ - viewed as a polynomial of variables $x,y$ on the plane $\{ z=0\}$ - vanishes
exactly of the transforms of the lines $C_1, \ldots , C_k$. In particular, the center of $\nu$ must be coincide with one
of the points in $\Pi^{-1} (0,0,0) \cap \{ z=0\}$ determined by the lines in question.

Up to applying the technique in the proof of Lemma~\ref{lemma4-4_forProposition} at every stage, a sequence of blow-ups
as above can be performed. The outcome of this sequence of blow-ups, which also summarizes the preceding two cases,
is the following claim:

\noindent {\it Claim}. Up to performing the indicated blow-up $\Pi$ and considering the corresponding
transform ($\fol_1$) of $\fol$ under $\Pi$, we can assume without loss of generality that the vector field
$X$ representing $\fol$ as in~(\ref{vectorfield-X-2_Appendix}) satisfies one of the following conditions:
\begin{itemize}
  \item if $n$ and $m$ are strictly positive. Then $k=1$ (and $h, h^{(0)}$ are not divisible by either $x$ or $y$).
  \item if $n=0$, then $m >0$ and $k \in \{ 1, 2\}$. However, if $k=2$, then $h$ is divisible by~$x$ which is therefore
  naturally associated with one of the lines $C_1$, $C_2$. Finally, again $h, h^{(0)}$ are not divisible by~$y$.
\end{itemize}
\mbox{}\qed

The above considered blow-up procedure actually leads to a slightly more accurate situation. Recalling that
$h$ does not vanish over the corresponding transforms of the (initial) invariant axes, we will arrive to one of the
following two situations:
\begin{enumerate}
\item the tangent cone of $h^{(0)}$ in the plane $\{ z=0\}$
does not pass through the points in $\Pi^{-1} (0,0,0) \cap \{ z=0\}$ determined by the invariant axes of the vector field~$Y$.

\item Consider the points in $\Pi^{-1} (0,0,0) \cap \{ z=0\}$ which are determined by  by the invariant axes of the vector field~$Y$.
Then the only component of the tangent cone of $h^{(0)}$ passing through these points
coincides with the curve $\Pi^{-1} (0,0,0) \cap \{ z=0\}$.
\end{enumerate}

In the first situation, $\fol$ has elementary singular points in $\Pi^{-1} (0,0,0) \cap \{ z=0\}$ except at the point $p$
determined by the intersection of $\Pi^{-1} (0,0,0) \cap \{ z=0\}$. In this case the point $p$ is regular for the transform of
the vector field~$Y$. Thus, there are local coordinates (still denoted by $(x,y,z)$) around $p$, with
$\{ y =z=0 \} \subset \Pi^{-1} (0,0,0) \cap \{ z=0\}$ where $X$ becomes:

\vspace{0.1cm}

\noindent {\it Case}\, ({\it a})
\begin{equation}
X = y^m \partial /\partial x + z (r \partial /\partial x + s \partial /\partial y) + z(h^{(0)} + z h^{(z)})
\partial /\partial z \, . \label{FormFinal-Casecalled-a}
\end{equation}
Moreover the cone tangent to $h$ (or to $h^{0)}$) coincides with the $y$-axis.

In the second situation, we can eliminate the curve $\Pi^{-1} (0,0,0) \cap \{ z=0\}$ from the singular
set of $x^n y^m Y$. Thus, $X$ becomes:

\vspace{0.1cm}

\noindent {\it Case}\, ({\it b})
\begin{equation}
X = x^n [ (f + z r) \partial /\partial x  + (g + zs) \partial /\partial y] + z(h^{(0)} + z h^{(z)})
\partial /\partial z \, . \label{FormFinal-Casecalled-b}
\end{equation}
Moreover $h^{(0)}$ vanishes only at $\Pi^{-1} (0,0,0) \cap \{ z=0\}$, i.e. $h = y^l h_1$ where $h_1 (0,0,0) \neq 0$ and $l \geq 1$.

The remainder of the proof amounts to showing how to handle each case.

Assume first that {\it Case}\, ({\it a}) happens. We consider a blow-up $\Pi$ centered at the $z$-axis which,
after the discussion revolving around Lemma~\ref{lemma4-4_forProposition}, is constituted by singular points of $X$.
In particular, the function $h$ viewed as the component of $X$ in the direction $\partial /\partial z$ is transformed
as function. Hence the blown-up vector field $X_1$ will therefore vanish identically over the cylindrical exceptional divisor. Indeed,
the blow-up of $y^m \partial /\partial x + z (r \partial /\partial x + s \partial /\partial y)$ will vanish with
order $m-1$ over the exceptional divisor whereas the zero-set of the transform of $z(h^{(0)} + z h^{(z)})
\partial /\partial z$ will consist of the union of the exceptional divisor with the {\it strict}\, transform of $h=0$. Letting
$k$ denote the minimum of the vanishing orders of these two vector fields over the exceptional divisor, the blown-up
foliation $\fol_1$ is induced by the vector field $X_1$ divided by the $k^{\rm th}$-power of the generator of the ideal associated
with the exceptional divisor.

Now, we must have $k < m-1$ (strictly) since, otherwise, all singularities of $\fol_1$ lying in
$\Pi^{-1} (0,0,0) \cap \{ z=0\}$ will have a non-zero eigenvalue associated with directions contained in $\{ z=0\}$.

Thus $k < m-1$, and after dividing $X_1$ by the $k^{\rm th}$-power of the generator of the exceptional divisor, we see
that the component of the resulting vector field in the direction $\partial /\partial z$ vanishes only on the strict transform
of $\{ h=0\}$. Thus every point in $\Pi^{-1} (0,0,0) \cap \{ z=0\}$ is either regular or elementary for the foliation
$\fol_1$ except the point determined by the strict transform of $h$. Denoting by $p_1$ the point of $\Pi^{-1} (0,0,0) \cap \{ z=0\}$
in question, we first note that $p_1$ must coincide with the (new) center of $\nu$. Moreover,
the structure of the blown-up foliation $\fol_1$ around $p_1$ is again as in {\it Case}\, ({\it a}). The integer
$m$, however, has decreased strictly. Therefore, after finitely many blow-ups as above, the center of $\nu$ will be turned either
in a regular point or in an elementary singularity for the corresponding foliation. In any event, this gives a contradiction
showing that this situation cannot happen.

Finally, assume now that {\it Case}\, ({\it b}) happens. Recall that $h=y^l h_1$ with $h_1 (0,0,0) \neq 0$.
Consider again the blow-up $\Pi$ centered at the $z$-axis and note that we will be able to divide
the transform of $X$ by the generator of the exceptional divisor to the power $k = \min \{ m,l \}$.

Assume that $m\geq l$. Then the blown-up foliation $\fol_1$ will have elementary singular points (with a non-zero
eigenvalue in the direction $\partial /\partial z$) over the entire curve $\Pi^{-1} (0,0,0) \cap \{ z=0\}$
except at the point determined by the $x$-axis ($\{ y=0\}$). Around this point, we have again the situation
described in {\it Case}\, ({\it b}) except that the value of~$m$ is replaced by $(m-l)$ and, therefore, has reduced
strictly.

Conversely, if $l > m$, then $h$ will still vanish over the curve $\Pi^{-1} (0,0,0) \cap \{ z=0\}$ (and, in fact,
over the cylindrical exceptional divisor). However, the transform of the vector field
$(f + z r) \partial /\partial x  + (g + zs) \partial /\partial y$ will provide us with regular or elementary singular
points over $\Pi^{-1} (0,0,0) \cap \{ z=0\}$, except at the point $q$ determined by the $y$-axis ($\{ x=0\}$).
Note that, around $q$, the tangent cone of $\{ h=0\}$ is reduced to $\Pi^{-1} (0,0,0) \cap \{ z=0\}$ since the
transform of its initial component given by $\{ y=0\}$ intersects $\Pi^{-1} (0,0,0) \cap \{ z=0\}$ at a point
different from~$q$. Thus, we have again the situation
described in {\it Case}\, ({\it b}) except that, now, the value of~$l$ is replaced by $(l-m)$.

Hence, at every blow-up as above, at least one between $m$ and $l$ becomes strictly smaller. Once one of them becomes
zero, we obtain an elementary singular point at the center of $\nu$ which of course is impossible. The
proof Proposition~\ref{BasedonCano_Roche_Spivakovsky} is now completed.
\end{proof}

The reader will note that the proof of Proposition~\ref{BasedonCano_Roche_Spivakovsky} also shows that the
formal surface $\widehat{W}$ with transverse maximal contact with $\nu$ {\it cannot be invariant by $\fol$}. Indeed, the
preceding discussion shows that singularities in the center of $\nu$ - a valuation with transverse maximal
contact with $\widehat{W}$ - can be turned into elementary ones provided that $\widehat{W}$ is invariant.
This observation can be reformulated as follows:

\begin{corol}
\label{Non-invariant-set_of_zeros}
Let $\fol$ be as in Proposition~\ref{BasedonCano_Roche_Spivakovsky} and consider a formal separatrix
$S$ giving rise to a sequence of infinitely near non-elementary singular points. Then $S$ cannot be contained
in a formal surface invariant under $\fol$.
\end{corol}

\begin{proof}
Clearly a formal surface containing $S$ will always pass through the center of (extended) valuation associated with~$S$
so that the statement follows from the previous discussion.
\end{proof}

Note that Corollary~\ref{Non-invariant-set_of_zeros} is hardly surprising since it is, in fact,
very much in line with the main result of \cite{CanoRoche}.

Now we close this paper with the proof of Theorem~A.

\begin{proof}[Proof of Theorem A]
Let $\fol$ be a foliation on $(\C^3,0)$ and fix a valuation $\nu$. Assume that no sequence of blow ups as
in~(\ref{aseqofblowups}) transforms $\fol$ in a foliation whose singular points contained in the center of
(the extension of) $\nu$ are all elementary. Then $\nu$ can be assumed to satisfy the conditions of
Lemma~\ref{lemma1-1_forProposition}. In turn, there follows from Proposition~\ref{BasedonCano_Roche_Spivakovsky} and
from Theorem~\ref{Microlocalversion_TheoremA} that $\fol$ can be turned in a foliation possessing a persistent
nilpotent singularity at the center of $\nu$. Thus we have improved the Local uniformization theorem
(Theorem~1) of \cite{canorochetc} to the following statement: the foliation $\fol$ can be transformed in
a foliation whose singularities at the center of $\nu$ either are elementary or are persistent nilpotent
singular points.

Next the blow up procedure we have used all along our construction clearly satisfies the same ``naturality'' conditions
satisfied by the procedure in \cite{canorochetc}. Hence this procedure verifies Piltant's axioms in \cite{Piltant},
cf. pages~256-257 of \cite{canorochetc}. Thus we obtain the following global result:

\noindent {\it Claim}: Every foliation $\fol$ on $(\C^3,0)$ can be transformed by a sequence of blow ups
as in~(\ref{aseqofblowups}) into a foliation $\fol_k$ whose singular points either are elementary
or are persistent nilpotent singularities.\qed

To complete the proof of Theorem~A we proceed as follows. Note that the set formed by the
persistent nilpotent singular points of $\fol_k$ consists of isolated points thanks to Theorem~\ref{normalformpersistentnilpotent}.
Up to reducing the neighborhood of $(0,0,0) \in \C^3$ under consideration, this set is therefore
finite. Finally each of these (finitely many) singular points can be turned into an elementary singular
point by means of a blow up with weight~$2$, cf. Lemma~\ref{lemma_ramified_blowup}. Theorem~A is proved.
\end{proof}

\vspace{0.1cm}

\begin{flushleft}
{\sc Julio Rebelo} \\
Institut de Math\'ematiques de Toulouse\\
118 Route de Narbonne\\
F-31062 Toulouse, FRANCE.\\
rebelo@math.univ-toulouse.fr

\end{flushleft}

\vspace{0.1cm}

\begin{flushleft}
{\sc Helena Reis} \\
Centro de Matem\'atica da Universidade do Porto, \\
Faculdade de Economia da Universidade do Porto, \\
Portugal\\
hreis@fep.up.pt \\

\end{flushleft}

\end{document}